\documentclass{article}
\usepackage{psfrag}
\usepackage[parfill]{parskip}
\usepackage{float, graphicx}
\usepackage[]{epsfig}
\usepackage{amsmath, amsthm, amssymb,}
\usepackage{epsfig}
\usepackage{verbatim}
\usepackage{multicol}
\usepackage{url}
\usepackage{latexsym}
\usepackage{mathrsfs}
\usepackage[colorlinks, bookmarks=true]{hyperref}
\usepackage{graphicx}
\usepackage{amsmath}
\usepackage{enumerate}
\usepackage[normalem]{ulem}
\usepackage{bm}
\usepackage{dsfont}
\usepackage{stmaryrd}
\usepackage{enumitem}
\usepackage{color}
\usepackage{xcolor}
\usepackage{hyperref}
\usepackage[noadjust]{cite}

\makeatletter

\def\@settitle{\begin{center}%
    \bfseries
 \normalfont\LARGE\@title
  \end{center}%
}
\def\@setauthors{\begin{center}%
 \normalsize\@author
  \end{center}%
}

\makeatother

\numberwithin{equation}{section}

\renewcommand{\cal}{\mathcal}

\newcommand{\cC}{{\cal C}}
\newcommand{\cD}{{\cal D}}
\newcommand{\cE}{{\cal E}}

\newcommand{\cL}{{\cal L}}
\newcommand{\cM}{{\cal M}}

\newcommand{\cR}{{\mathcal R}}

\newcommand{\cU}{{\mathcal U}}

\newcommand\cW{{\mathcal W}}

\newcommand{\sfa}{{\sf a}}
\newcommand{\sfb}{{\sf b}}
\newcommand{\sfs}{{\sf s}}
\newcommand{\sfS}{{\sf S}}
\newcommand{\sfA}{{\sf A}}
\newcommand{\sfB}{{\sf B}}
\newcommand{\sfC}{{\sf C}}

\newcommand{\fa}{{\mathfrak a}}
\newcommand{\fb}{{\mathfrak b}}

\newcommand{\fc}{{\mathfrak c}}

\newcommand{\fr}{{\mathfrak r}}

\newcommand{\fC}{{\frak C}}

\newcommand{\fM}{{\frak M}}

\newcommand{\bma}{{\bm{a}}}
\newcommand{\bmb}{{\bm{b}}}

\newcommand{\bme}{{\bm{e}}}

\newcommand{\bmv}{{\bm{v}}}

\newcommand{\bmx}{{\bm{x}}}
\newcommand{\bmy}{{\bm{y}}}

\newcommand{\rd}{{\rm d}}

\newcommand{\ri}{\mathrm{i}}

\newcommand{\bC}{{\mathbb C}}

\newcommand{\bE}{\mathbb{E}}

\newcommand{\bP}{\mathbb{P}}

\newcommand{\bR}{{\mathbb R}}

\newcommand{\bZ}{\mathbb{Z}}

\DeclareMathOperator{\Tr}{Tr}

\DeclareMathOperator{\supp}{supp}

\DeclareMathOperator{\dist}{dist}

\DeclareMathOperator{\dom}{\mathcal{D}}

\DeclareMathOperator{\OO}{O}
\DeclareMathOperator{\oo}{o}

\renewcommand{\Re}{\mathop{\mathrm{Re}}}
\renewcommand{\Im}{\mathop{\mathrm{Im}}}

\newcommand{\deq}{\mathrel{\mathop:}=} 
 
\renewcommand{\leq}{\leqslant}
\renewcommand{\geq}{\geqslant}

\newcommand{\td}{\tilde}
\newcommand{\del}{\partial}

\newcommand{\beq}{\begin{equation}}
\newcommand{\eeq}{\end{equation}}

\theoremstyle{plain} 
\newtheorem{theorem}{Theorem}[section]
\newtheorem*{theorem*}{Theorem}

\newtheorem*{lemma*}{Lemma}

\newtheorem*{corollary*}{Corollary}
\newtheorem{proposition}[theorem]{Proposition}
\newtheorem*{proposition*}{Proposition}
\newtheorem{assumption}[theorem]{Assumption}
\newtheorem*{assumption*}{Assumption}
\newtheorem{claim}[theorem]{Claim}
\newtheorem{ansatz}[theorem]{Ansatz}

\newtheorem*{definition*}{Definition}

\newtheorem*{example*}{Example}
\newtheorem{remark}[theorem]{Remark}

\newtheorem*{remark*}{Remark}
\newtheorem*{remarks*}{Remarks}

\def\author#1{\par
    {\centering{\authorfont#1}\par\vspace*{0.05in}}
}

\def\titlefont{\fontsize{13}{15}\bfseries\boldmath\selectfont\centering{}}
\def\authorfont{\fontsize{13}{15}}

\let\affiliationfont\rhfont

\def\address#1{\par
    {\centering{\affiliationfont#1\par}}\par\vspace*{11pt}
}

\def\body{
\setcounter{footnote}{0}
\def\thefootnote{\alph{footnote}}
\def\@makefnmark{{$^{\rm \@thefnmark}$}}
}

\def\title#1{
    \thispagestyle{plain}
    \vspace*{-14pt}
    \vskip 79pt
    {\centering{\titlefont #1\par}}%
    \vskip 1em
}

\newcommand{\ft}{{\frak t}}

\begin{document}

\title{Edge Universality for Nonintersecting Brownian Bridges}

\vspace{1.2cm}

 \author{Jiaoyang Huang}
\address{New York University\\
   E-mail: jh4427@nyu.edu}

~\vspace{0.3cm}

\begin{abstract}
In this paper we study fluctuations of extreme particles of nonintersecting Brownian bridges
starting from $a_1\leq a_2\leq \cdots \leq a_n$ at time $t=0$ and ending at $b_1\leq b_2\leq \cdots\leq b_n$ at time $t=1$, where $\mu_{A_n}=(1/n)\sum_{i}\delta_{a_i}, \mu_{B_n}=(1/n)\sum_i \delta_{b_i}$ are discretization of probability measures $\mu_A, \mu_B$.
Under regularity assumptions of $\mu_A, \mu_B$, we show as the number of particles $n$ goes to infinity, fluctuations of extreme particles at any time $0<t<1$, after proper rescaling, are asymptotically universal, converging to the Airy point process.
\end{abstract}

\section{Introduction}

One-dimensional Markov processes conditioned
not to intersect form an important class of models which arise in the study of random matrix
theory, growth processes, directed polymers and random tiling (dimer) models \cite{ferrari2010random, ferrari2010interacting, spohn2005kardar,johansson2005random,weiss2017reflected}. 
Among them, nonintersecting Brownian bridges, from conditioning $n$ standard Brownian bridges not to intersect, have been most studied.
Their scaling limits give rise to determinantal point processes, which are believed to be universal objects for large families of interacting particle systems.

In the particular case, when nonintersecting Brownian bridges start at general position and end at the same position, say the origin, after a space-time transformation \eqref{e:dual}, it is the distribution of  Dyson's Brownian motion \cite{MR0148397} on the real line. The positions of the particles at any time slice have the same distribution as the eigenvalues of the Gaussian unitary ensemble with external source \cite{johansson2001universality}.
In this case, as the number of particles  goes to infinity, under
proper scaling,  the local statistics of nonintersecting Brownian bridges are universal, i.e. governed by the sine kernel inside the limit shape (in the bulk) \cite{johansson2001universality, MR3914908,MR3687212,aptekarev2005large,bleher2004large},  by the Airy kernel at the edge of the limit shape \cite{tracy1994level, landon2017edge,aptekarev2005large,bleher2004large},
and by the Pearcey kernel at the cusp \cite{adler2007pdes, bleher2007large, brezin1998universal, brezin1998level,tracy2006pearcey}. 
These kernels are universal since they appear in many other problems.
In particular, the Airy process appears ubiquitously in
the Kardar--Parisi--Zhang (KPZ) universality class \cite{corwin2012kardar}, an important class of interacting particle systems and random growth models.
The analysis of nonintersecting Brownian bridges greatly improves the understanding of the Airy process and the KPZ universality class, see \cite{corwin2014brownian,dauvergne2018directed}.

In this paper, we study the edge scaling limit of nonintersecting Brownian bridges with general boundary condition.
They consist of $n$ standard Brownian bridges $\{(x_1(t), x_2(t),\cdots, x_n(t))\}_{0\leq t\leq 1}$ starting from $a_1\leq a_2\leq \cdots \leq a_n$ at time $t=0$ and ending at $b_1\leq b_2\leq \cdots\leq b_n$ at time $t=1$, 
conditioned not to intersect during the time interval $0<t<1$, i.e. $x_i(0)=a_i, x_i(1)=b_i$ for $1\leq i\leq n$ and $x_1(t)<x_2(t)<\cdots<x_n(t)$ for $0<t<1$. We parametrize the starting and ending configurations as measures,
\begin{align}\label{e:dicret}
\mu_{A_n}=\frac{1}{n}\sum^n_{i=1}\delta_{a_i}, \quad
\mu_{B_n}=\frac{1}{n}\sum^n_{i=1}\delta_{b_i}.
\end{align}


If the measures $\mu_{A_n}$ and $\mu_{B_n}$ converge weakly to $\mu_A$ and $\mu_B$ respectively, as the number of particles $n$ goes to infinity, 
under mild assumptions of $\mu_A, \mu_B$, it follows from \cite{MR1883414, MR2091363} that the empirical particle density of nonintersecting Brownian bridges with boundary data $\mu_{A_n},\mu_{B_n}$, converges
\begin{align*}
\frac{1}{n}\sum_{i=1}^n \delta_{x_i(t)}\rightarrow \rho_t^*(x)\rd x,\quad 0\leq t\leq1,
\end{align*}
in the weak sense. The measure valued process $\{\rho_t^*(x)\}_{0\leq t\leq1}$ is given explicitly by a variational problem:
\begin{align}\label{e:funcS0}
\{(\rho_t^*, u_t^*)\}_{0\leq t\leq 1}=\arg\inf\frac{1}{2}\left(\int_0^1 \int_\bR \left(u_t^2 \rho_t +\frac{\pi^2}{3}\rho_t^3 \right)\rd x \rd t +\Sigma(\mu_A)+\Sigma(\mu_B)\right),
\end{align}
the $\inf$ is taken over all the pairs $\{(\rho_t,u_t)\}_{0\leq t\leq 1}$ such that $\del_t \rho_t+\del_x(\rho_tu_t)=0$ in the sense of distributions, $\{\rho_t(\cdot)\}_{0\leq t\leq 1}\in{\mathcal{C}}([0,1],\cM(\bR))$  and its initial and terminal data are  given by 
\begin{align*}
\lim_{t\rightarrow 0+}\rho_t(x)\rd x=\rd\mu_A,\quad \lim_{t\rightarrow 1-}\rho_t(x)\rd x=\rd\mu_B,
\end{align*}
where convergence holds in the weak sense.
 We will discuss more on the variational problem in Section \ref{s:VP}. Let $\Omega=\{(x,t)\in \bR\times (0,1), \rho_t^*(x)>0\}$ be the region where $\rho_t^*(x)$ is positive. As in Figure \ref{f:DBM}, it turns out that the nonintersecting Brownian bridge occupies $\Omega$. 
The boundary of the region $\Omega$ is piecewise analytic, with possibly some cusp points. 
Let $\{(\sfa(t),t)\}_{0\leq t\leq 1}$ be the left boundary of $\Omega$. For any $0<t<1$, in the neighborhood of $x=\sfa(t)$,  $\rho_t^*(x)$ has square root behavior 
\begin{align}\label{e:square}
\rho_t^*(x)= \frac{\sfs(t)\sqrt{[x-\sfa(t)]_+}}{\pi}+\OO(|x-\sfa(t)|^{3/2}),
\end{align}
and similar statement holds for the right boundary of $\Omega$.

\begin{figure}
\begin{center}
 \includegraphics[scale=0.22,trim={0cm 5cm 0 7cm},clip]{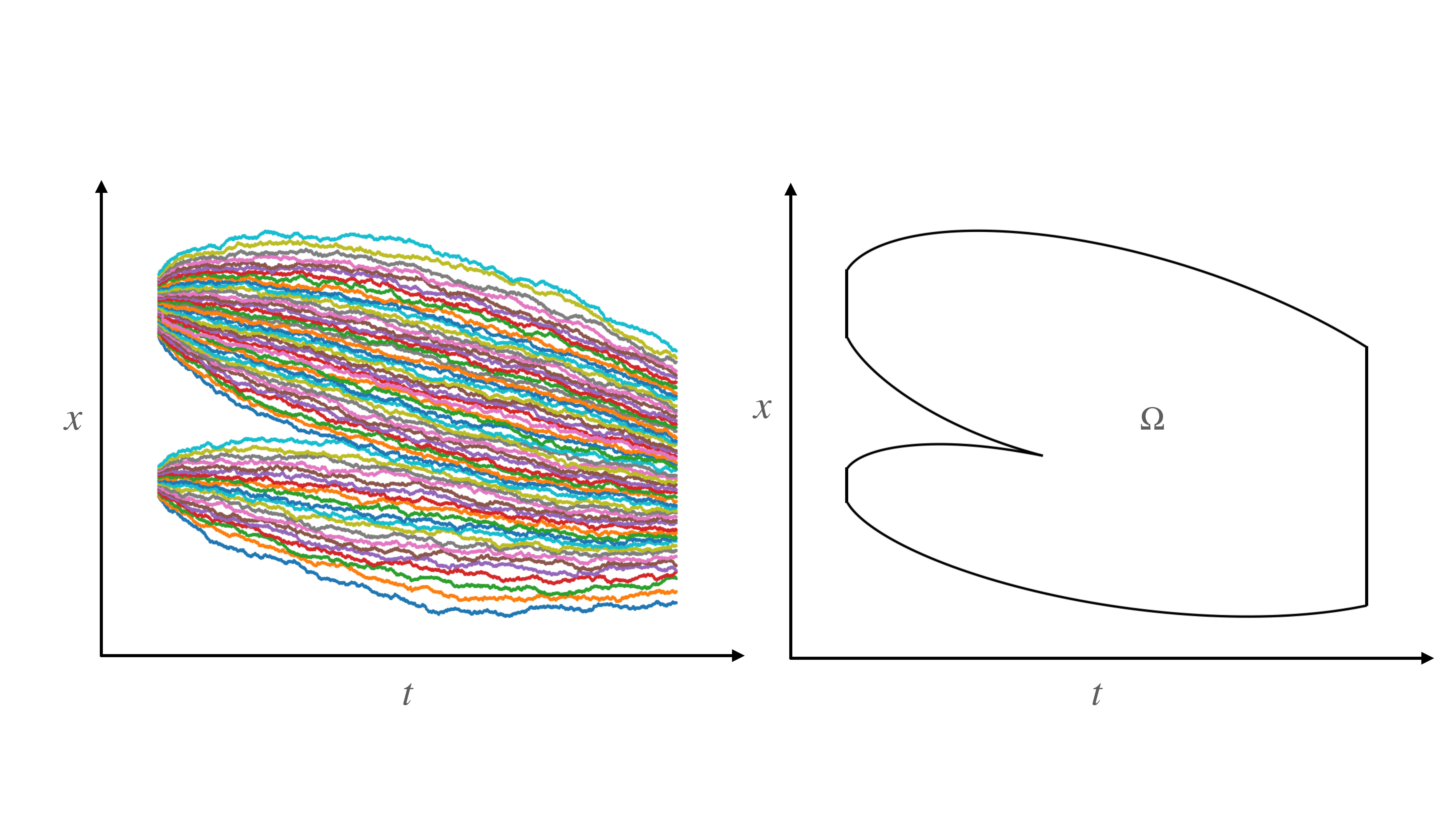}
 \caption{An example of nonintersecting Brownian bridges, and its support $\Omega$.}
 \label{f:DBM}
 \end{center}
 \end{figure}

In this work, we study  fluctuations of extreme particles of  nonintersecting Brownian bridges.  Under regularity assumptions of $\mu_A, \mu_B$, we show as the number of particles $n$ goes to infinity, fluctuations of extreme particles at any time $0<t<1$, after proper rescaling, converge
to the Airy point process.

\begin{theorem}\label{t:universality}
Given probability measures $\mu_A,\mu_B$ satisfying regularity Assumptions \ref{a:reg} and \ref{a:ncritic}, and $\mu_{A_n}, \mu_{B_n}$ the discretization of $\mu_A,\mu_B$ satisfying Assumptions \ref{a:A_n} and \ref{a:B_n}. Fix small $\ft>0$. 
For nonintersecting Brownian bridges with boundary data given by $\mu_{A_n}, \mu_{B_n}$, as $n$ goes to infinity, fluctuations of extreme particles at time $\ft\leq t\leq1-\ft$, after proper rescaling, converge to the Airy point process,
\begin{align*}
(\sfs(t)n)^{2/3}(x_1(t)-\sfa(t), x_2(t)-\sfa(t),x_3(t)-\sfa(t),\cdots)\rightarrow \text{Airy Point Process}
\end{align*}
The same statement holds for particles close to the right edge.
\end{theorem}

In Theorem \ref{t:universality}, we need $\mu_{A_n}$ and $\mu_{B_n}$ to be sufficiently close to their limits $\mu_A$ and $\mu_B$. Especially, we do not allow outliers. Nonintersecting Brownian bridges with all but finitely many
leaving from and returning to $0$, have been studied in \cite{adler2010airy,adler2009dyson}. In this setting, the edge scaling limit is a new Airy process with wanderers,  governed by
an Airy-type kernel, with a rational perturbation.

A Brownian bridge $W(t)$ on $[0,1]$ from $a$ to $b$ i.e. $W(0)=a$ and $W(1)=b$ can be represented by a standard Brownian motion $B(t)$ starting from $a$ with drift $b$
\begin{align}\label{e:dual}
W(t)=(1-t)B\left(\frac{t}{1-t}\right).
\end{align}
Under the transformation \eqref{e:dual},  nonintersecting Brownian bridges starting from $a_1\leq a_2\leq \cdots \leq a_n$ at time $t=0$ and ending at $b_1\leq b_2\leq \cdots\leq b_n$ at time $t=1$ become nonintersecting Brownian motions with drift \cite{biane2005littelmann, takahashi2012noncolliding}, i.e. $n$ standard Brownian motions starting from $a_1\leq a_2\leq \cdots \leq a_n$ at time $t=0$ and drifts $b_1\leq b_2\leq \cdots \leq b_n$ conditioned not to intersect. Our main result Theorem \ref{t:universality} implies the edge universality for nonintersecting Brownian motions with drift.

%
%

For nonintersecting Brownian bridges starting at time $t=0$
at $q$ and ending at time $t=1$ at $p$ prescribed positions, with $p,q\geq 2$,
it was proven in \cite{delvaux2009phase} that the correlation functions of particle positions  have a determinantal form, with a kernel 
expressed in terms of mixed multiple Hermite polynomials, which can be characterized by a Riemann-Hilbert problem of size $(p+q)\times(p+q)$.
In the same setting, a partial differential equation for the probability to find all the particles in a given set has been obtained in \cite{adler2012non}. However, it remains challenging to do an asymptotic analysis for those Rieman-Hilbert problems and partial differential equations. The only case where the asymptotic analysis was done is for $p=q=2$. In \cite{daems2008asymptotics}, the authors consider nonintersecting Brownian bridges, where $n/2$ particles go from $a$ to $b$, and $n/2$ particles
go from $-a$ to $-b$. For a small separation of the starting and ending positions,  they showed the kernels for the local statistics in the bulk and near the edges converge to the sine and Airy kernel in the large $n$ limit. For large separation of the starting and ending positions, those results have been extended in \cite{delvaux2009phase}.
In the critical regime where the Brownian bridges fill two tangent ellipses in the time-space plane, a new correlation kernel, called the tacnode kernel, was obtained in \cite{adler2013nonintersecting, delvaux2011critical,ferrari2012non,johansson2013non}.

%
%
%

It seems to be a highly non-trivial problem to obtain concrete results about the scaling limit of nonintersecting Brownian bridges with general starting and ending configurations using its determinantal structure.  In this paper we take a dynamical approach to study nonintersecting Brownian bridges. 
Let $\Delta_n$ be the Weyl chamber 
\begin{align}\label{e:Weylchamber}
\Delta_n=\{(z_1, z_2, \cdots, z_n)\in \bR^n: z_1<z_2<\cdots<z_n\}.
\end{align}
We reinterpret nonintersecting Brownian bridges as a random walk over $\Delta_n$ with drift. The transition probability density function of $n$ dimensional nonintersecting Brownian motions from $\bmx=(x_1, x_2,\cdots, x_n)\in \Delta_n$ to $\bmy=(y_1,y_2,\cdots, y_n)\in \Delta_n$ is given by the Karlin-McGregor formula \cite{MR114248} 
\begin{align}\label{e:tp0}
p_t(\bmx, \bmy)=\det\left[\sqrt{\frac{n}{2\pi t}}e^{-n(x_i-y_j)^2/(2t)}\right]_{1\leq i,j\leq n}.
\end{align}
With the transition kernel \eqref{e:tp0}, we can rewrite nonintersecting Brownian bridges as the following random walk over $\Delta_n$ with drift: 
\begin{align}\label{e:random0}
\rd x_i(t)
&=
\frac{1}{\sqrt n} \rd B_i(t)+\frac{1}{n}\del_{x_i}\log p_{1-t}(\bmx(t),\bmb)\rd t,\quad 1\leq i\leq n,
\end{align}
where $\bmx(t)=(x_1(t), x_2(t), \cdots, x_n(t))\in \Delta_n$,  $\{B_1(t), B_2(t), \cdots, B_n(t)\}_{0\leq t\leq 1}$ are standard Brownian motions, and the drift $p_{1-t}(\bmx,\bmb)$  is the heat kernel in the Weyl chamber,
\begin{align}\label{e:heat0}
-\del_t  p_{1-t}(\bmx,\bmb)
=\frac{1}{2n}\Delta p_{1-t}(\bmx,\bmb),\quad p_{1-t}(\bmx,\bmb)=\det \left[\sqrt{\frac{n}{2\pi (1-t)}}e^{-n(x_i-b_j)^2/(2(1-t))}\right]_{0\leq i,j\leq n} .
\end{align}

The limiting profile of nonintersecting Brownian bridges is characterized by the variational problem \eqref{e:funcS0}. We can use it to define a measure valued  Hamiltonian system: given any probability measure $\mu\in \cM(\bR)$, and time $0\leq t\leq 1$ let
\begin{align}\label{e:varWt0}
W_t(\mu)=-\frac{1}{2}\inf\left(\int_t^{1}\int_\bR \rho_s\left(u_s^2+\frac{\pi^2}{3}\rho_s^2\right)\rd x\rd s+\Sigma(\mu)+\Sigma(\mu_B)\right),
\end{align}
where the non-commutative entropy $\Sigma(\mu)$ is defined in \eqref{e:ncentropy},
the $\inf$ is taken over all the pairs $\{(\rho_s,u_s)\}_{t\leq s\leq {1}}$ such that $\del_s \rho_s+\del_x(\rho_su_s)=0$ in the sense of distributions, $\{\rho_s\}_{t\leq s\leq 1+\tau}\in{\mathcal{C}}([t,1],\cM(\bR))$  and its initial and terminal data are  given by 
\begin{align*}
\lim_{s\rightarrow t+}\rho_s(x)\rd x=\mu,\quad \lim_{s\rightarrow 1-}\rho_s(x)\rd x=\mu_B,
\end{align*}
where convergence holds in the weak sense. Then the Hamilton's principal function $W_t(\mu)$ satisfies the following Hamilton-Jacobi equation
\begin{align}\begin{split}\label{e:Wt0}
-\del_t W_t(\mu)
&=\frac{1}{2}\int \left(\frac{\del}{\del x} \frac{\delta W_t}{\delta \mu}\right)^2\rd \mu(x)
+\int \frac{\del}{\del x} \frac{\delta W_t}{\delta \mu} H(\mu)(x)\rd \mu(x),
\end{split}\end{align}
where $H(\mu)(x)$ is the Hilbert transform of the measure $\mu$. There is a Riemann surface associated with the variational problem \eqref{e:varWt0}.  Properties of $W_t(\mu)$ can be understood using tools for Riemann surfaces, i.e. Rauch variational formula and  Hadamard's Variation Formula.

Following Matytsin's approach \cite{MR1257846}, we make an inverse Cole-Hopf transformation to convert the linear heat equation \eqref{e:heat0} to a nonlinear Hamilton-Jacobi equation, which is almost the same as \eqref{e:Wt0} by taking $\mu=(1/n)\sum_{i}\delta_{x_i}$. Based on this observation, we make the following ansatz
\begin{align}\label{e:ansatz0}
\frac{1}{n}\frac{\del}{\del_{x_k}}\log\frac{ p_{1-t}(\bmx,\bmb)}{\prod_{i<j}(x_i-x_j)}=\left.\frac{\del}{\del x} \frac{\delta W_t}{\delta \mu}\right|_{x=x_k}+\cE^{(k)}_t(\bmx),\quad 1\leq k\leq n.
\end{align}
Then we solve for the correction terms $\cE^{(k)}_t(\bmx)$ by using Feynman-Kac formula. It turns out the correction terms $\cE^{(k)}_t$ for $0\leq t\leq 1$ are of order $\OO(1/n^2)$. They have negligible influence on the random walk \eqref{e:random0}. 
By plugging the ansatz \eqref{e:ansatz0} into \eqref{e:random0}, and ignoring the correction terms $\cE^{(k)}_t$, the system of stochastic differential equations becomes Dyson's Brownian motion, with drifts depending on the particle configuration $\bmx(t)$. We analyze it using the method of characteristics, following the approach developed in \cite{adhikari2020dyson,MR4009708}.

The heat kernel $p_{1-t}(\bmx,\bmb)$ and the Hamilton's principle function $W_t(\mu)$ are singular when $t$ approaches $1$, which makes the corresponding Hamilton-Jacobi equations hard to analyze. To overcome this problem, instead of studying  nonintersecting Brownian bridges from $\bma$ to $\bmb$ directly, we studied a weighted version of it. The weighted version corresponds to nonintersecting Brownian bridges with random boundary data at time $t=1$. There exists a natural choice of weight, such that at time $t=1$, the particle configuration of the weighted nonintersecting Brownian bridges concentrates around $\bmb$. We can analyze the weighted nonintersecting Brownian bridges using the above  approach. Then by a coupling argument, we can transfer the edge universality result for  the weighted nonintersecting Brownian bridges to the edge universality of nonintersecting Brownian bridges from $\bma$ to $\bmb$.

We now outline the organization for the rest of the paper. In Section \ref{s:VP}, we recall the variational problem which characterizes the limiting profile of nonintersecting Brownian bridges from \cite{MR2034487}. In Section \ref{s:burger}, We use the rate function of the variational problem to define a measure valued Hamiltonian system, and study its properties using tools from Riemann surfaces, i.e. Rauch variational formula and Hadamard's Variation Formula. In Section \ref{s:walk}, we define the weighted nonintersecting Brownian bridges, and reinterpret it as a drifted random walk. Based on the similarity to the Hamilton-Jacobi equation for the measure valued system, we make an ansatz for the drift term of the random walk.
 In Section \ref{s:eq}, we solve for the correction term in the ansatz using Feynman-Kac formula.
In Sections \ref{s:rigidity} we prove the optimal rigidity for the particle locations for this weighted nonintersecting Brownian bridges. 
And  optimal rigidity estimates for nonintersecting Brownian bridges follow from a coupling argument in Section \ref{s:rigidityBB}.
Using optimal rigidity estimates as input in Section \ref{s:edgew} we prove edge universality for nonintersecting Brownian bridges.

\paragraph{Notations}

We denote $\cM(\bR)$ the set of probability measures over $\bR$, and ${\mathcal{C}}([0,1],\cM(\bR))$ the set of continuous measure valued process over $[0,1]$.
We use $\fC$ to represent large universal constant, and $\fc$ a small universal constant,
which may depend on other universal constants, i.e., the constants $\fa, \fb, \ft$ in Assumptions \ref{a:ncritic} and \eqref{a:A_n}, and
may be different from line by line. 
We write $a\vee b=\max\{a,b\}$ and $a\wedge b=\min\{a,b\}$.
We write that $X = \OO(Y )$ if there exists some universal constant such
that $|X| \leq \fC Y$ . We write $X = \oo(Y )$, or $X \ll Y$ if the ratio $|X|/Y\rightarrow \infty$ as $n$ goes to infinity. We write
$X\asymp Y$ if there exist universal constants such that $\fc Y \leq |X| \leq  \fC Y$. We say an event holds with overwhelming probability, if for any $\fC>0$, and $n\geq n_0(\fC)$ large
enough, the event holds with probability at least $1-N^{-\fC}$.

\paragraph{Acknowledgements}
The research of J.H. is supported by the Simons Foundation as a Junior Fellow
at the Simons Society of Fellows.

\section{Variational Principle}\label{s:VP}

The transition probability density \eqref{e:tp0} of nonintersecting Brownian motions is closely related to the Harish-Chandra-Itzykson-Zuber integral formula \cite{MR84104, MR562985}.
Let $A_n, B_n$ be two $n \times n$ diagonal matrices, with diagonal entries given by $a_1\leq a_2\leq \cdots \leq a_n$ and $b_1\leq b_2\leq \cdots\leq b_n$ respectively, the Harish-Chandra-Itzykson-Zuber integral formula exactly computes the following integral
\begin{align}\begin{split}\label{e:HCIZ}
\int e^{n\Tr(A_nUB_nU^*)}\rd U
=\frac{\prod_{j=1}^{n-1}j!}{n^{(n^2-n)/2}}\frac{\det[e^{ na_ib_j}]_{i,j}}{\Delta(a_1,a_2,\cdots,a_n)\Delta(b_1,b_2,\cdots,b_n)},
\end{split}\end{align}
where $U$ follows the Haar probability measure of the unitary group and $\Delta$ denotes the Vandermonde determinant
\begin{align}\label{e:Vand}
\Delta(a_1, a_2,\cdots,a_n)=\prod_{1\leq i<j\leq n}(a_j-a_i),\quad \Delta(b_1, b_2,\cdots,b_n)=\prod_{1\leq i<j\leq n}(b_j-b_i).
\end{align}

We can rewrite the transition probability \eqref{e:tp0} of nonintersecting Brownian bridges by rescaling the Harish-Chandra-Itzykson-Zuber integral formula \eqref{e:HCIZ}:
\begin{align*}\begin{split}
\frac{p_1(\bma, \bmb)}{\Delta(\bma)\Delta(\bmb)}
&=\frac{\det[\sqrt{n/2\pi }e^{-n (a_i-b_j)^2/2}]_{i,j}}{\Delta(a_1,a_2,\cdots,a_n)\Delta(b_1,b_2,\cdots,b_N)}\\
&=\frac{{n}^{(n^2-n)/2}}{\prod_{j=1}^{n-1}j!} \prod_{i=1}^n\sqrt{n/2\pi }e^{-n (a_i^2+b_i^2)/2}\int e^{n\Tr(A_nUB_nU^*)}\rd U.
\end{split}\end{align*}

If the spectral measures $\mu_{A_n},\mu_{B_n}$ of $A_n, B_n$ converge weakly towards $\mu_{A}$ and $\mu_{B}$ respectively, under mild assumptions, it was proven in \cite{MR1883414, MR2091363}, see also \cite{MR2034487,MR3380685}, that the Harish-Chandra-Itzykson-Zuber integral converges
\begin{align*}
\lim_{n\rightarrow\infty}\frac{1}{n^{2}}\log\int e^{n\Tr(A_n UB_nU^{*})} \rd U= I(\mu_{A},\mu_{B}).
 \end{align*}

The asymptotics $I(\mu_{A},\mu_{B})$ of the Harish-Chandra-Itzykson-Zuber integral is characterized by a variational problem. We recall that for any probability measure $\mu\in \cM(\bR)$, we denote $\Sigma(\mu)$ the energy of its logarithmic potential, or  its non-commutative entropy, 
\begin{align}\label{e:ncentropy}
\Sigma(\mu)=\int\int \log |x-y|\rd \mu(x)\rd \mu(y).
\end{align} 
The following theorem is from \cite[Theorem 2.1]{MR2034487}.

\begin{theorem}[{\cite[Theorem 2.1]{MR2034487}}]\label{theoCMP}
We assume that  $\mu_A, \mu_B$ are both compactly supported, then $I(\mu_A, \mu_B)$ is given by
\begin{align}\label{e:Iexp}
I(\mu_A,\mu_B)=-\frac{1}{2}\inf \left(S(u,\rho)+\left(\Sigma(\mu_A)+\Sigma(\mu_B)\right)-\left(\int x^2 \rd \mu_A(x)+\int x^2 \rd \mu_B(x)\right)\right)+\text{const.}
\end{align}
where 
\begin{align}\label{e:funcS}
S(u,\rho)=\int_0^1 \int_\bR \left(u_t^2 \rho_t +\frac{\pi^2}{3}\rho_t^3 \right)\rd x \rd t,
\end{align}
the $\inf$ is taken over all the pairs $\{(\rho_t,u_t)\}_{0\leq t\leq 1}$ such that $\del_t \rho_t+\del_x(\rho_tu_t)=0$ in the sense of distributions, $\{\rho_t(\cdot)\}_{0\leq t\leq 1}\in{\mathcal{C}}([0,1],\cM(\bR))$  and its initial and terminal data are  given by 
\begin{align*}
\lim_{t\rightarrow 0+}\rho_t(x)\rd x=\rd\mu_A,\quad \lim_{t\rightarrow 1-}\rho_t(x)\rd x=\rd\mu_B,
\end{align*}
where convergence holds in the weak sense.
\end{theorem}

The infimum in \eqref{e:Iexp} is reached at
a unique probability measure-valued path
$\rho_t^*\rd x\in{\mathcal{C}}([0,1],\cM(\bR))$
such that for $t\in (0,1)$, $\rho_t^* \rd x$ is absolutely continuous with respect to Lebesgue measure. The measure-valued path
$\rho_t^*\rd x\in{\mathcal{C}}([0,1],\cM(\bR))$ describes the empirical particle locations of  nonintersecting Brownian bridge \eqref{e:density}: as $n$ goes to infinity
\begin{align*}
\left\{\frac{1}{n}\sum_{i=1}^n \delta_{x_i(t)}\right\}_{0\leq t\leq 1}\rightarrow \{\rho_t^* \rd x\}_{0\leq t\leq 1}.
\end{align*} 

The minimizer $\{(\rho_t^*(\cdot), u^*_t(\cdot))\}_{0\leq t\leq 1}$ of the variational problem \eqref{e:Iexp} can be described using the language of free probability, \cite[Theorem 2.6]{MR2034487}. 
\begin{theorem}[{\cite[Theorem 2.6]{MR2034487}}]\label{t:fbb}
 There exist two non-commutative operators $\mathsf a, \sf b$
with marginal distribution $(\mu_A, \mu_B)$ 
and  a non-commutative brownian motion $\{\sf s_t\}_{0\leq t\leq 1}$ independent of $\sf a, \sf b$ in a non-commutative probability
space $(\cal A,\tau)$. The following free brownian bridge
\begin{align}\label{e:frep}
\rd {\sf x}_t=\rd {\sf s}_t+\frac{{\sf b}-{\sf x}_t}{1-t}\rd t, \quad {\sf x}_0=\sf a, \quad {\sf x}_1=\sf b,
\end{align}
at time $t$ has the law given by $\rho_t^*$, and
\begin{equation}\label{e:condut} u_t^*=\frac{1}{t-1}\tau(\mathsf x_t-\mathsf b|\mathsf x_t) +H(\rho^*_t),\qquad \rho_{t}^{*}(x) \rd x \quad a.s. 
\end{equation}
where $H(\rho_t^*)$ is the Hilbert transform of $\rho_t^*$.
\end{theorem}

\section{Complex Burger's equation}
\label{s:burger}
We denote the minimizer of the variational problem \eqref{e:Iexp} as $\{(\rho_t^*(\cdot), u^*_t(\cdot))\}_{0\leq t\leq 1}$ and let $\Omega$ be the domain where $\rho^*_t(x)>0$,
\begin{align}\label{def:Omega}
\Omega=\{(x,t)\in \bR\times(0,1): \rho^*_t(x)>0\}.
\end{align}
%
Then it follows from \cite[Lemma 7.2]{BGH2020}, $\Omega$ is a bounded simply connected (it does not contain holes) open domain in $\bR\times [0,1]$. Moreover, the free brownian bridge representation, Theorem \ref{t:fbb}, implies that 
the infimum $\{(\rho_t^*(\cdot), u^*_t(\cdot))\}_{0\leq t\leq 1}$ is analytic in $\Omega$.
In the rest of the paper, we call $\del \Omega \cap \bR\times (0,1)$ the left and right boundary of $\Omega$; $\del \Omega \cap \bR\times \{0\}$ the bottom boundary of $\Omega$; $\del \Omega \cap \bR\times \{1\}$ the top boundary of $\Omega$.

As derived in \cite[Theorem 2.1]{MR2034487}, the Euler-lagurange equation for the variational problem \eqref{e:funcS} implies that the minimizer $\{(\rho_t^*(\cdot), u^*_t(\cdot))\}_{0\leq t\leq 1}$ satisfies
\begin{align*}
\del_t u_t^*=-\frac{1}{2}\del_x\left((u_t^*)^2-\pi^2(\rho_t^*)^2\right),\quad (x,t)\in \Omega.
\end{align*}
If we define 
\begin{align}\label{e:ft}
f_t(x)=u^*_t(x)-\ri \pi \rho^*_t(x),
\end{align}
then $f(x,t)$ satisfies the complex Burger's equation
\begin{align}\label{e:burgereq}
\del_t f_t(x)+f_t(x)\del_x f_t(x) =0, \quad (x,t)\in \Omega.
\end{align}
In physics
literature, the complex Burger's equation description of the Harish-Chandra-Itzykson-Zuber integral formula first appeared in the
work of Matytsin \cite{MR1257846}. Its rigorous mathematical study appears in the work
of  Guionnet \cite{MR2034487}.

The complex Burger's equation can be solved using characteristic method.  The argument in  \cite[Corollary 1]{MR2358053} implies that there exists an analytic function of $Q$ of two variables such that the solution of \eqref{e:burgereq} satisfies
\begin{align}\label{e:defQ}
Q(f_t(x), x-tf_t(x))=0,
\end{align}
where the analytic function $Q$ is determined by the boundary condition at $t=0,1$:
$-\Im[f_0(x)]\rd x/\pi=\rd\mu_A$ and $-\Im[f_1(x)]\rd x/\pi=\rd\mu_B$.
  By the definition \eqref{def:Omega}, on $\Omega$ we have $\Im[f_t(x)]<0$. We can recover $(x,t)$ from $(f_t(x), x-tf_t(x))$, by noticing that $t=-\Im[x-tf_t(x)]/\Im[f_t(x)]$.
Thus, the map 
\begin{align}\label{e:cover}
\pi_Q: (x,t)\in \Omega\mapsto (f_t(x),x-tf_t(x))
\end{align}
is an orientation preserving diffeomorphism from $\Omega$ onto its image. 


Thanks to the free Brownian bridge representation \eqref{e:frep}, $\rho_t^*$ is the law of 
\begin{align*}
(1-t)\sf a+t\sf b+\sqrt{t(1-t)}\sf s,
\end{align*}
where $\sf s$ is a free semi-circular law and it is free with $\sf a,\sf b$. As a direct consequence of
\cite{MR1488333}, for any $t\in (0,1)$,
and $x_0\in\bR$ such that $(t,x_0)\in \partial\Omega$, in a small neighborhood of $x_0$,
\begin{align*}
\rho_t^*(x)\leq \left({3\over 4\pi^3 t^2(1-t)^2}\right)^{1\over3}
(x-x_0)^{1\over 3}, \quad x\in \supp \rho_t^*.
\end{align*}
Especially, for $0<t<1$, $f_t(x)$ is continuous up to the boundary of $\Omega$. We can glue
$f_t(x)$ and its complex conjugate along the left and right boundary of $\Omega$ (where $f_t(x)\in \bR$). This map is orientation preserving and
unramified by the discussion above, and so a covering map of $Q(f,z)=0$. Moreover, the Riemann surface $Q(f,z)=0$ is of the same genus as $\Omega$, which has genus zero \cite[Lemma 7.2]{BGH2020}.
As a consequence the left and right boundary of the region $\Omega$ is piecewise analytic, with some cusp points. For any $0<t<1$, the density $\rho_t^*(x)$ is analytic. Moreover, $\rho_t^*(x)$ has square root behavior at  left and right boundary points (which are not cusp points) of its support. Let $\{(\sfa(t),t)\}_{0\leq t\leq 1}$ be the left boundary of $\Omega$. For any $0<t<1$, in the neighborhood of $x=\sfa(t)$,  $\rho_t^*(x)$ has square root behavior 
\begin{align}\label{e:square}
\rho_t^*(x)= \frac{\sfs(t)\sqrt{[x-\sfa(t)]_+}}{\pi}+\OO(|x-\sfa(t)|^{3/2}).
\end{align}

With the Riemann surface $Q(f,z)=0$, we can extend the function \eqref{e:ft} $f_t(x)$ from $\{x: (x,t)\in\Omega\}$ to a meromorphic function on the Riemann surface $Q(f,z)=0$.
At time $t$, the equation $\{(f,z)\in \bC\times \bC:  Q(f, z-tf)=0\}$ defines a Riemann surface $\cC_t$. We notice $\cC_0$ is simply the Riemann surface $Q(f,z)=0$, and $\cC_t$ can  be obtained from $\cC_0$ by the characteristic flow:
\begin{align}\label{e:flow}
Z_0=(f,z)\in \cC_0\mapsto Z_t=(f, z+tf)\in \cC_t.
\end{align}
As a consequence, the family of Riemann surfaces $\{\cC_t\}_{0\leq t\leq 1}$ are homeomorphic to each other.
Comparing with \eqref{e:defQ}, this gives a natural extension for $f_t(x)$ from $\{x: (x,t)\in\Omega\}$ to the Riemann surface $\cC_t$. In this way $f_t: Z=(f,z)\in \cC_t\mapsto f\in \bC$. If the context is clear we will simply write $f_t(Z)=f_t((f,z))$ as $f_t(z)$ for $Z=(f,z)\in \cC_t$.
The complex Burger's equation \eqref{e:burgereq} extend naturally to $f_t$:
\begin{align}\label{e:complexburgereq}
\del_t f_t(Z)+\del_z f_t(Z)f_t(Z)=0,\quad 0\leq t\leq 1, \quad Z\in \cC_t.
\end{align}

\subsection{Regularity Assumptions}\label{s:reg}
In the rest of this paper, we will make the assumption that $\mu_B$ is regular enough, such that the solution of the complex Burger's equation \eqref{e:burgereq} can be extended to $t=1+\tau$ with some $\tau>0$.
\begin{assumption}[Regularity]\label{a:reg}
The probability measures $\mu_A, \mu_B$ are regular in the sense:
There exists some constant $\tau>0$, 
the solution of the complex Burger's equation \eqref{e:burgereq}, with boundary conditions:
\begin{align*}
-\frac{1}{\pi}\Im[f_0(x)]\rd x=\rd\mu_A,\quad -\frac{1}{\pi}\Im[f_1(x)]\rd x=\rd\mu_B,
\end{align*}
can be extended up to  time $t=1+\tau$. 
\end{assumption}

Under Assumption \ref{a:reg}, we denote the extension by $\{f_t(x)\}_{1\leq t\leq 1+\tau}$, and 
\begin{align*}
-\frac{1}{\pi}\Im[f_{1+\tau}(x)]\rd x=\rd\mu_C.
\end{align*} 
for some probability measure $\mu_C$. 

For the following variational problem with boundary data $\mu_A, \mu_C$,
\begin{align}\label{e:varmuC}
-\frac{1}{2}\inf\left(\int_0^{1+\tau}\int_\bR \rho_t\left(u_t^2+\frac{\pi^2}{3}\rho_t^2\right)\rd x\rd t+\Sigma(\mu_A)+\Sigma(\mu_C)\right),
\end{align}
the $\inf$ is taken over all the pairs $\{(\rho_t,u_t)\}_{0\leq t\leq {1+\tau}}$ such that $\del_t \rho_t+\del_x(\rho_tu_t)=0$ in the sense of distributions, $\{\rho_t\}_{0\leq t\leq 1+\tau}\in{\mathcal{C}}([0,1+\tau],\cM(\bR))$  and its initial and terminal data are  given by 
\begin{align*}
\lim_{t\rightarrow 0+}\rho_t(x)\rd x=\rd\mu_A,\quad \lim_{t\rightarrow (1+\tau)-}\rho_t(x)\rd x=\rd\mu_C,
\end{align*}
where convergence holds in the weak sense. The minimizer $\{(\rho_t^*(\cdot), u^*_t(\cdot))\}_{0\leq t\leq 1+\tau}$ of \eqref{e:varmuC}, satisfies
\begin{align*}
\rho_B(x)\rd x\deq \rho_1^*(x)\rd x=\rd \mu_B,
\end{align*}
and the density $\rho_B(x)$ is analytic. For our analysis, we need to assume that $\rho_B(x)$ has a connected support and is non-critical.

\begin{assumption}[Non-critical]\label{a:ncritic}
We assume that $\mu_B$ has a connected support and is non-critical,
that is, there exists a constant $\fb>0$ such that,
\begin{align*}
\rho_B(x)=S(x)\sqrt{(x-\sfa(1))(\sfb(1)-x)},\quad S(x)\geq \fb,\quad x\in [\sfa(1), \sfb(1)].
\end{align*}
\end{assumption}

Under Assumptions \ref{a:reg} and \ref{a:ncritic}, by shrinking $\tau$ if necessary, we can take $\rd \mu_C=\rho_C(x)\rd x$ to be analytic, non-critical and have a connected support, i.e. $\rho_C$ also satisfies Assumption \ref{a:ncritic}.

We make the following assumptions on the boundary data $\mu_{A_n}, \mu_{B_n}$ of our nonintersecting Brownian bridges. We can take $\mu_{A_n}, \mu_{B_n}$ to be the $1/n$-quantiles of measures $\mu_A,\mu_B$. Assumptions \ref{a:A_n} and \ref{a:B_n} are slightly more general.
\begin{assumption}\label{a:A_n}
We assume the initial data $\mu_{A_n}$ of nonintersecting Brownian bridges satisfies: there exists a constant $\fa>0$ such that for any $\varepsilon>0$, and $n$ sufficiently large, 
\begin{align}\label{e:A_n}
\left|\int\frac{\rd\mu_A(x)-\rd\mu_{A_n}(x)}{x-z}\right|\leq \frac{(\log n)^\fa}{n},
\end{align}
uniformly for any $\{z\in \bC_+: \dist(z, \supp(\mu_A))\geq \varepsilon \}$.
\end{assumption}
\begin{assumption}\label{a:B_n}
We assume the terminal data $\mu_{B_n}$ of nonintersecting Brownian bridges satisfies: there exists a constant $\fa>0$ such
\begin{align}\label{e:B_n}
\gamma_{i-(\log n)^\fa}\leq b_i \leq \gamma_{i+(\log n)^\fa},\quad \frac{i+1/2}{n}=\int_{-\infty}^{\gamma_i}\rho_B(x)\rd x,\quad 1\leq i\leq n, 
\end{align}
where we use the convention $\gamma_{i}=-\infty$ for $i\leq 0$ and $\gamma_i=+\infty$ for $i\geq n$, and 
\begin{align*}
   b_1\geq {\sf a}(1)-\frac{(\log n)^{\fa}}{n^{2/3}},\quad b_{n}\leq {\sf b}(1)+ \frac{(\log n)^{\fa}}{n^{2/3}}.
\end{align*}
\end{assumption}

Assumption \ref{a:A_n} guaranteed that for any analytic test function $f(x)$, by a contour integral
\begin{align}\label{e:linearstat}
\left|\int f(x)(\rd \mu_{A}(x)-\rd \mu_{A_n}(x))\right|\lesssim \frac{(\log n)^\fa}{n}.
\end{align}
We will see in Section \ref{s:rigidity}, the assumption is necessary for the optimal rigidity estimates. 
We remark that the assumption \eqref{e:B_n} implies \eqref{e:A_n} with possibly a different $\fa$. In this sense Assumption \ref{a:B_n} is stranger than Assumption \ref{a:A_n}. Moreover, in Assumption \ref{e:A_n} we allow to have outliers which are distance $\oo(1)$ away from $\supp(\mu_A)$. 
But, we do not allow outliers with are distance $\OO(1)$ away from the support. 
Nonintersecting Brownian bridges with all but finitely many
leaving from and returning to $0$, have been studied in \cite{adler2010airy,adler2009dyson}.
 In this setting, the edge scaling limit is a new Airy process with wanderers,  governed by
an Airy-type kernel, with a rational perturbation.

The same as in \eqref{e:defQ}, we can use the minimizer $\{(\rho_t^*(\cdot), u^*_t(\cdot))\}_{0\leq t\leq 1+\tau}$ of the variational problem \eqref{e:varmuC}, to construct the family of Riemann surfaces $\{\cC_t\}_{0\leq t\leq 1+\tau}$, and extend $f_t(\cdot)$ to the Riemann surface $\cC_t$ for $0\leq t\leq 1+\tau$. They satisfy the complex Burger's equation
\begin{align}\label{e:complexburgereq2}
\del_t f_t(Z)+\del_z f_t(Z)f_t(Z)=0, \quad 0\leq t\leq 1+\tau, \quad Z\in \cC_t.
\end{align}
The boundary conditions are $-\Im[f_0(x)] \rd x/\pi=\rd\mu_A$ and $-\Im[f_{1+\tau}(x)]\rd x/\pi=\rd\mu_C$.

\begin{figure}
\begin{center}
 \includegraphics[scale=0.22,trim={0cm 5cm 0 7cm},clip]{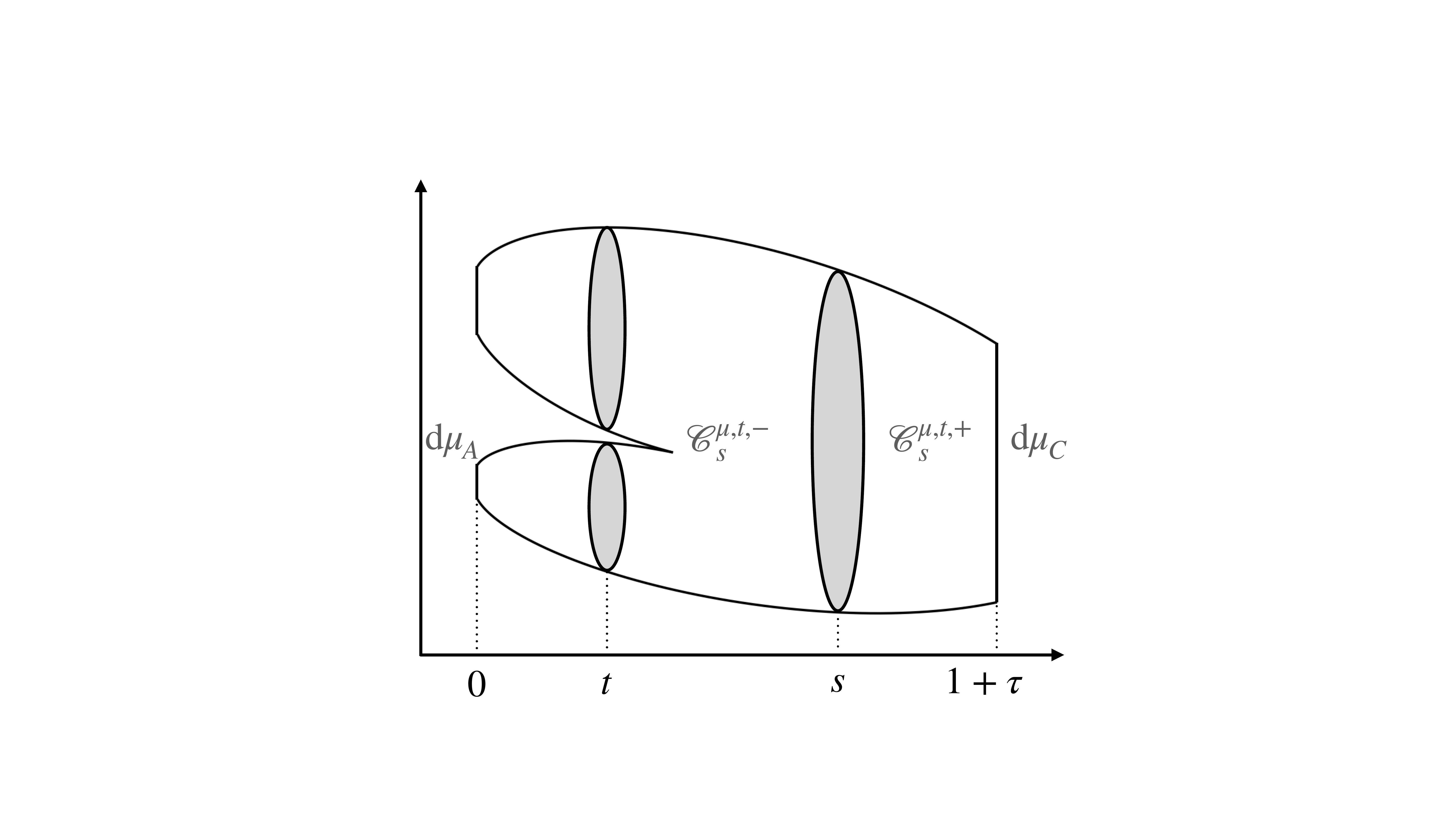}
 \caption{We can identify the Riemann surface $\cC_0$ with two copies of $\Omega$ gluing together.
These curves $\{( f_s(x;\mu,t),x): (x,s)\in\Omega^{\mu,t}\}$ and $ \{(\overline{ f_s(x;\mu,t)},x): (x,s)\in\Omega^{\mu,t}\}$
 cut the Riemann surface $\cC_s^{\mu,t}$ into two parts:  $\cC_s^{\mu,t,-}, \cC_s^{\mu,t,+}$. $\cC_s^{\mu,t,-}$ corresponds to $\{(x,r): (x,r)\in\Omega^{\mu,t}, t\leq r\leq s\}$, and $\cC_s^{\mu,t,+}$ corresponds to $\{(x,r): (x,r)\in\Omega, s\leq r\leq 1+\tau\}$.  }
 \label{f:cut}
 \end{center}
 \end{figure}

We can reformulate our boundary conditions $-\Im[f_0(x)] \rd x/\pi=\rd\mu_A$, $-\Im[f_{1+\tau}(x)]\rd x/\pi=\rd\mu_C$ in terms of the residuals of the $1$-form $f_0(Z)\rd z$ over $\cC_0$. Let $\Omega=\{(x,t)\in \bR\times(0,1+\tau): \rho^*_t(x)>0\}$. We recall that using the map \eqref{e:cover}, 
\begin{align}\label{e:covercopy}
\pi_Q: (x,t)\in \Omega\mapsto (f_t(x),x-tf_t(x))\in \cC_0,
\end{align}
the Riemann surface $\cC_0$ can be identified as gluing two copies of $\Omega$ along their left and right boundaries, and has cuts along the top and bottom boundaries of $\Omega$. The $1$-form $f_0(Z)\rd z$ has residuals along the bottom boundary of $\Omega$ corresponding to time $t=0$, and along the top boundary of $\Omega$ corresponding to time $t=1+\tau$, as in Figure \ref{f:cut}.
The boundary condition at time $t=0$, $-\Im[f_0(x)] \rd x/\pi=\rd\mu_A$ gives that the residual of $f_0(Z)\rd z$ along the bottom boundary of $\Omega$ is given by \begin{align}\label{e:topres}
\frac{1}{\pi}\Im[f_0(x)\rd x]=-\rd \mu_A(x).
\end{align} 
For the boundary condition at time $t={1+\tau}$, $-\Im[f_{1+\tau}(x)] \rd x/\pi=\rd\mu_C$, we use the characteristic flow \eqref{e:flow}, which send $Z_0=(f,z)\in \cC_0$ to $Z_{1+\tau}=(f, z+(1+\tau)f)\in \cC_{1+\tau}$
\begin{align*}
f_{1+\tau}(Z_{1+\tau})=f_{1+\tau}((f, z+(1+\tau)f))=f=f_0((f,z))=f_0(Z_0).
\end{align*}
Let 
\begin{align*}
f_{1+\tau}(x)=\gamma(x)-\ri \pi \rho_C(x),\quad x\in \supp(\mu_C),
\end{align*}
where $\gamma: \supp(\rho_C)\mapsto \bR$ is analytic.
Then the preimage of $\{(f_{1+\tau}(x),x): x\in \supp(\rho_C)\}$ under the characteristic flow consists of two curves in $\cC_0$, which can be parametrized as
\begin{align}\label{e:botres}
\left\{(f,z)=\left(\gamma(x)-\ri\pi\rho_C(x), x+(1+\tau)(-\gamma(x)+\ri\pi\rho_C(x))\right):x\in \supp(\rho_C)  \right\},
\end{align}
and its complex conjugate.
Thus the residual of $f_0(Z)\rd z$ along the top boundary of $\pi_Q(\Omega)$ is given by 
\begin{align*}
\frac{1}{\pi}\Im\left[(\gamma(x)-\ri\pi\rho_C(x))\rd(x+(1+\tau)(-\gamma(x)+\ri\pi\rho_C(x)))\right]
=-\rho_C(x)\rd x+(1+\tau)\rd \gamma(x)\rho_C(x).
\end{align*}

\begin{remark}
We remark that the residuals of $f_0(Z)\rd z$ along the bottom and top boundaries of $\Omega$, \eqref{e:topres} and \eqref{e:botres}, uniquely determines the Riemann surface $\cC_0:Q(f,z)=0$. 
However, it is not easy to directly write down the expression of $f$.
\end{remark}

\subsection{More General Boundary Condition}\label{s:bc}

Slightly more general than \eqref{e:varmuC}, we can consider 
 the following functional, for any probability measure $\mu$, and time $0\leq t\leq 1$,
\begin{align}\label{e:varWt}
W_t(\mu)=-\frac{1}{2}\inf\left(\int_t^{1+\tau}\int_\bR \rho_s\left(u_s^2+\frac{\pi^2}{3}\rho_s^2\right)\rd x\rd s+\Sigma(\mu)+\Sigma(\mu_C)\right),
\end{align}
the $\inf$ is taken over all the pairs $\{(\rho_s,u_s)\}_{t\leq s\leq {1+\tau}}$ such that $\del_s \rho_s+\del_x(\rho_su_s)=0$ in the sense of distributions, $\{\rho_s\}_{t\leq s\leq 1+\tau}\in{\mathcal{C}}([t,1+\tau],\cM(\bR))$  and its initial and terminal data are  given by 
\begin{align*}
\lim_{s\rightarrow t+}\rho_s(x)\rd x=\mu,\quad \lim_{s\rightarrow (1+\tau)-}\rho_s(x)\rd x=\mu_C,
\end{align*}
where convergence holds in the weak sense. Similarly to Theorem \ref{t:fbb}, \eqref{e:varWt} corresponds to certain free Brownian bridge $\{\sf x_s\}_{t\leq s\leq 1+\tau}$, with the law of $\sf x_t$ and $\sf x_{1+\tau}$ given by $\mu$ and $\mu_C$ respectively.
The same as \eqref{e:burgereq},
if we denote the minimizer of \eqref{e:varWt} as $\{(\rho_s^*(\cdot;\mu,t), u_s^*(\cdot;\mu,t))\}_{t\leq s\leq 1}$ and 
\begin{align}\label{e:ft2}
f_s(x;\mu,t)=u^*_s(x;\mu,t)-\ri \pi \rho^*_s(x;\mu,t),\quad t\leq s\leq 1+\tau,
\end{align}
then $f_s(x;\mu,t)$ satisfies the complex Burger's equation
\begin{align}\label{e:burgereq2}
\del_s f_s(x;\mu,t)+f_s(x;\mu,t)\del_x f_s(x;\mu,t) =0, \quad (x,s)\in \Omega^{\mu,t}\deq \{(x,s)\in \bR\times (t,1+\tau): \rho_s^*(x;\mu,t)>0\}.
\end{align}
There exists an analytic function $Q^{\mu,t}$ of two variables such that the solution of \eqref{e:burgereq2} satisfies
\begin{align}\label{e:defQatt}
Q^{\mu,t}(f_s(x;\mu,t), x-sf_s(x;\mu,t))=0,\quad t\leq s\leq 1+\tau.
\end{align}
 With the analytic function $Q^{\mu,t}$, we can define the Riemann surface:
$\cC^{\mu,t}_s=\{(f,z)\in \bC\times \bC: Q^{\mu,t}(f,z-sf)=0\}$. Comparing with \eqref{e:defQatt}, this gives a natural extension for $f_s(x;\mu,t)$ from $\{x: (x,s)\in\Omega^{\mu,t}\}$ to the Riemann surface $\cC_s^{\mu,t}$. In this way $f_s: Z=(f,z)\in \cC_s^{\mu,t}\mapsto f\in \bC$. The complex Burger's equation extend naturally to $f_s$:
\begin{align}\label{e:complexburgereq}
\del_s f_s(Z;\mu,t)+\del_z f_s(Z;\mu,t)f_s(Z;\mu,t)=0,\quad t\leq s\leq 1+\tau, \quad Z\in \cC_s^{\mu,t}.
\end{align}
Similarly to \eqref{e:flow}, 
 for $t\leq s\leq 1+\tau$, $\cC^{\mu,t}_s$ can  be obtained from $\cC^{\mu,t}_t$ by the characteristic flow:
\begin{align}\label{e:flowatt}
Z_t=(f,z)\in \cC^{\mu,t}_t\mapsto Z_s=(f, z+(s-t)f)\in \cC^{\mu,t}_s.
\end{align}
Using the map 
\begin{align}\label{e:cover2}
\pi_{Q^{\mu,t}}: (x,r)\in\Omega^{\mu,t} \mapsto (f_r(x;\mu,t),x-(r-t)f_r(x;\mu, t))\in \cC_t^{\mu,t},
\end{align}
we can identify $\cC_t^{\mu,t}$ with two copies of $\Omega^{\mu,t}$ gluing along the left and right boundaries of $\Omega^{\mu,t}$.
The bottom boundary $\{(x,t): x\in \supp(\mu)\}$  and top boundary $\{(x,1+\tau): x\in \supp(\mu_C)\}$ of $\Omega^{\mu,t}$ are mapped to $\{X(x)=(f_t(x;\mu,t), x)\}_{x\in \supp(\mu)}$ and $\{C(x)=(f_{1+\tau}(x;\mu,t), x-(1+\tau-t)f_{1+\tau}(x;\mu,t))\}_{x\in\supp(\mu_C)}$ of $\cC^{\mu,t}_t$. 
The same as \eqref{e:topres}, the residual of $f_t(Z;\mu,t)\rd z$ along 
the bottom boundary of $\Omega^{\mu,t}$
is given by $-\rd\mu(x)$, and 
\begin{align}\label{e:residual0}
f_t(Z;\mu,t)=\int \frac{\rd \mu(x)}{z-x}+\OO(1), \quad Z\rightarrow X(x).
\end{align}

The same as  \eqref{e:botres}, along the top boundary of $\Omega^{\mu,t}$, we have 
\begin{align}\label{e:f1tau}
f_{1+\tau}(x;\mu,t)=\gamma^{\mu,t}(x)-\ri \pi \rho_C(x),\quad x\in \supp(\mu_C),
\end{align}
where $\gamma^{\mu,t}: \supp(\rho_C)\mapsto \bR$ is analytic. The residual of $f_t(Z;\mu,t)\rd z$ along the top boundary of $\Omega^{\mu,t}$ is given by 
\begin{align}\label{residual1}
\left(-\rho_C(x)\rd x+(1+\tau-t)\rd \gamma^{\mu,t}(x)\rho_C(x)\right).
\end{align}

Using the characteristic flow \eqref{e:flowatt}, we can also map $\Omega^{\mu,t}$ to $\cC_s^{\mu,t}$, for any $t\leq s\leq 1$:
\begin{align}\label{e:coveratr}
\pi^s_{Q^{\mu,t}}: (x,r)\in \Omega^{\mu,t}\mapsto (f_s(x; \mu,t),x-(r-s)f_s(x;\mu,t))\in \cC_s^{\mu,t},
\end{align}
and identify $\cC_s^{\mu,t}$ as the gluing of two copies of $\Omega^{\mu,t}$ along the left and right boundaries. 

Using the above identification, the points $(f,z)\in \cC_s^{\mu,t}$ with $z$ real, i.e. $z\in \bR$, correspond to the left and right boundary of $\Omega^{\mu,t}$ where $f_s(x;\mu,t)\in \bR$, and the horizontal segments $\{(x,s):(x,s)\in \Omega^{\mu,t} \}$. The image of $\{(x,s):(x,s)\in \Omega^{\mu,t} \}$ under the map \eqref{e:coveratr} $\{( f_s(x;\mu,t),x): (x,s)\in\Omega^{\mu,t}\}$ and its 
complex conjugate $ \{(\overline{ f_s(x;\mu,t)},x): (x,s)\in\Omega^{\mu,t}\}$ glue together to cycles on the Riemann surface $\cC_s^{\mu,t}$. They
 cut the Riemann surface $\cC_s^{\mu,t}$ into two parts:  $\cC_s^{\mu,t,-}, \cC_s^{\mu,t,+}$. $\cC_s^{\mu,t,-}$ corresponds to $\{(x,r): (x,r)\in\Omega^{\mu,t}, t\leq r\leq s\}$, and $\cC_s^{\mu,t,+}$ corresponds to $\{(x,r): (x,r)\in\Omega^{\mu,t}, s\leq r\leq 1+\tau\}$, as shown in Figure \ref{f:cut}.
%
If we consider the map, 
\begin{align}\label{e:zmatr}
(f,z)\in \cC^{\mu,t,+}_s\mapsto z\in \bC,
\end{align}
it maps the cut $\{( f_s(x;\mu,t),x): (x,s)\in\Omega^{\mu,t}\}\subset \cC_s^{\mu,t}\}\cup \{(\overline{ f_s(x;\mu,t)},x): (x,s)\in\Omega^{\mu,t}\}\subset \cC_s^{\mu,t}\}$ to the segment $\{x: (x,s)\in \Omega^{\mu,t}\}$ twice. We remark that in this notation, $\cC_t^{\mu,t}=\cC_t^{\mu,t,+}$.

Let $m_{s}(z;\mu,t)$ be the Stieltjes transform of the measure $\rho_s^*(x;\mu,t)$,
\begin{align*}
m_{s}(z;\mu,t)=\int \frac{\rho_s^*(x;\mu,t)}{z-x}\rd x.
\end{align*}
As a meromorphic function over $\bC$, we can lift $m_{s}(z;\mu,t)$ to a meromorphic function over $\cC_s^{\mu,t,+}$ using the map \eqref{e:zmatr}. 
We notice that on $\{x: (x,s)\in \Omega^{\mu,t}\}$
\begin{align*}\begin{split}
&\lim_{\eta\rightarrow 0+}m_{s}(x+\ri\eta;\mu,t)=-\ri\pi \rho_s^*(x;\mu,t)=\Im[f_s(x;\mu,t)],\\ 
&\lim_{\eta\rightarrow 0-}m_{s}(x+\ri\eta;\mu,t)=\ri\pi \rho_s^*(x;\mu,t)=\Im[\overline{f_t(x;\mu,t)}].
\end{split}\end{align*}
Therefore, if we subtract $m_{s}(z;\mu,t)$ from $f_s(z;\mu,t)$, what remaining can be glued along the cut $\{( f_s(x;\mu,t),x): (x,s)\in\Omega^{\mu,t}\}\cup \{(\overline{ f_s(x;\mu,t)},x): (x,s)\in\Omega^{\mu,t}\}$, i.e. identifying $( f_s(x;\mu,t),x)$ with $(\overline{ f_s(x;\mu,t)},x)$, to an analytic function in a neighborhood of the cut. More precisely,
let
\begin{align}\label{e:gszmut}
f_s(Z;\mu,t)=m_{s}(z;\mu,t)+g_s(Z; \mu, t),\quad z=z(Z)
\end{align}
then $g_s(Z;\mu,t)$, as a meromorphic function over $\cC^{\mu,t,+}_s$, can be glued along the cut $\{( f_s(x;\mu,t),x): (x,s)\in\Omega^{\mu,t}\}\cup \{(\overline{ f_s(x;\mu,t)},x): (x,s)\in\Omega^{\mu,t}\}$, to an analytic function in a neighborhood of the cut.

Comparing \eqref{e:gszmut} with \eqref{e:ft2}, on the cut $\{( f_s(x;\mu,t),x): (x,s)\in\Omega^{\mu,t}\}\cup \{(\overline{ f_s(x;\mu,t)},x): (x,s)\in\Omega^{\mu,t}\}$, we have
\begin{align*}\begin{split}
f_s(x;\mu,t)&=u^*_s(x;\mu,t)-\ri \pi \rho^*_s(x;\mu,t)=g_s(x;\mu,t)+m_{s}(x;\mu,t)\\
&=g_s(x;\mu,t)+\int\frac{\rho_s^*(y;\mu,t)}{x-y}\rd y-\ri \pi \rho^*_s(x;\mu,t).
\end{split}\end{align*}
It follows that
\begin{align}\label{e:gtut15}
g_s(x;\mu,t)=u^*_s(x;\mu,t)-\int\frac{ \rho_s^*(y;\mu,t)}{x-y}\rd y,\quad t\leq s\leq 1+\tau,
\end{align}
and especially by taking $s=t$, we have
\begin{align}\label{e:gtut2}
g_t(x;\mu,t)=u^*_t(x;\mu,t)-\int\frac{1}{x-y}\rd \mu(y).
\end{align}

We can use \eqref{e:gtut15} to derive the differential equations of $m_{s}(z;\mu,t)$ the Stieltjest transform of $\rho_s^*(x;\mu,t)$. We recall from \eqref{e:varWt}, $\del_s \rho^*_s(x;\mu,t)=-\del_x(\rho^*_s(x;\mu,t)u_s(x;\mu,t))$
\begin{align}\begin{split}\label{e:dermt}
&\phantom{{}={}}\del_s m_{s}(z;\mu,t)=\int\frac{\del_s \rho_s^*(x;\mu,t)}{z-x}\rd x
=\int\frac{-\del_x(\rho^*_s(x;\mu,t)u_s(x;\mu,t))}{z-x}\rd x\\
&=\int\rho_s^*(x;\mu,t)\left(g_s(x;\mu,t)+\int\frac{\rho_s^*(y;\mu,t)}{x-y}\rd y\right)\del_x\left(\frac{1}{z-x}\right)\rd x\\
&=-\del_z m_{s}(z;\mu,t)m_{s}(z;\mu,t)+\int\frac{g_s(x;\mu,t)}{(z-x)^2}\rho^*_s(x;\mu,t)\rd x.
\end{split}\end{align} 
Moreover, $f_s(Z;\mu,t)$ with $Z=(f,z)\in \cC_s^{\mu,t}$ satisfies the complex Burger's equation
\begin{align}\begin{split}\label{e:burger}
0&=\del_s f_s(z;\mu,t)+\del_z f_s(Z;\mu,t) f_s(Z;\mu,t)\\
&=
\del_s (m_{s}(z;\mu,t)+ g_s(Z;\mu,t))
+\del_z(m_{s}(z;\mu,t)+g_s(Z;\mu,t))(m_{s}(z;\mu,t)+ g_s(Z;\mu,t)).
\end{split}\end{align}
We obtain the differential equation of $g_s(Z;\mu,t)$ by taking difference of \eqref{e:dermt} and \eqref{e:burger},
\begin{align}\begin{split}\label{e:gteq}
\del_s g_s(Z;\mu,t)&+g_s(Z;\mu,t) \del_z g_s(Z;\mu,t)\\
&+\int \frac{g_s(x;\mu,t)-g_s(Z;\mu,t)-(x-z)\del_z g_s(Z;\mu,t)}{(x-z)^2} \rho^*_s(x;\mu,t)\rd x=0.
\end{split}\end{align}

\subsection{Variational Formula}
In Sections \ref{s:eq}, we need to estimate the difference of two solutions of the complex Burger's equation \eqref{e:complexburgereq} with different initial data: $f_t(Z^1;\mu^1,t)-f_t(Z^0;\mu^0,t)$, where $z(Z^1)=z(z^0)=z$. To estimate it we interpolate the initial data 
\begin{align*}
\mu^\theta=(1-\theta)\mu^0+\theta \mu^1, \quad 0\leq \theta\leq 1,
\end{align*}
and write 
\begin{align*}
f_t(Z^1;\mu^1,t)-f_t(Z^0;\mu^0,t)=\int_0^1 \del_\theta f_t(Z^\theta;\mu^\theta,t)\rd \theta,\quad z(Z^\theta)=z.
\end{align*}
For any $0\leq \theta\leq 1$,  we have that
\begin{align}\label{e:zmattcopy}
(f_t(Z;\mu^\theta, t)),z)\in \cC_t^{\mu^\theta,t}\mapsto z\in\bC,
\end{align}
is holomorphic except for some branch points. Since $Q^{\mu^\theta,t}$ is real, the branch points of \eqref{e:zmattcopy} come in pairs. We denote the branch points of \eqref{e:zmattcopy} as $Z_1^\theta=(f_1^\theta, z^\theta_1), \bar Z_1^\theta=(\bar{f}^\theta_1, \bar{z}^\theta_1), Z_2^\theta=(f^\theta_2, z^\theta_2), \bar Z_2^\theta=(\bar{f}^\theta_2, \bar{z}^\theta_2), \cdots, Z_w^\theta=(f^\theta_w, z^\theta_w), \bar Z_w^\theta=(\bar{f}^\theta_w, \bar{z}^\theta_w)$ with ramification indices $d_1, d_1, d_2,d_2,\cdots, d_w, d_w$ respectively.


Given the Riemann surface, the Schiffer kernel \cite{MR39812,BE2018} $B(Z,Z';\mu^\theta,t)$ on $(Z,Z')=((f,z), (f',z'))\in\cC_t^{\mu^\theta,t}\times \cC_t^{\mu^\theta,t}$ is a symmetric meromorphic bilinear differential, i.e. meromorphic $1$-form of $Z$ tensored by
a meromorphic $1$-form of $Z'$ and $B(Z,Z';\mu^\theta,t)=B(Z',Z;\mu^\theta,t)$. And it has a double pole on the diagonal, such that in a small neighborhood $V$ of $(f,z)\in \cC^{\mu^\theta,t}_t$, for $(Z,Z')\in V\times V$,
\begin{align}\label{e:Schiffer}
B(Z,Z';\mu^\theta,t)=\frac{1}{(z-z')^2}\rd z\rd z'
+ \OO(1). 
\end{align}

We recall the map $\pi_{Q^{\mu^\theta,t}}$ from \eqref{e:cover2}, such that using  $\pi_{Q^{\mu^\theta,t}}$, $\cC_t^{\mu^\theta,t}$ can be identified as two copies of $\Omega^{\mu^\theta,t}$ gluing along the boundary of $\Omega^{\mu^\theta,t}$. Since 
$\Omega^{\mu^\theta,t}$ is simply connected, we denote $\phi$ the Riemann map from $\pi_{Q^{\mu^\theta,t}}(\Omega^{\mu^\theta,t})$ to the upper half plane. We can extend $\phi$ to
$\overline{\pi_{Q^{\mu^\theta,t}}(\Omega^{\mu^\theta,t})}$ by setting $\phi(\overline Z)=\overline{\phi( Z)}$, for any $Z\in \pi_{Q^{\mu^\theta,t}}(\Omega^{\mu^\theta,t})$. In this way, $\phi$ is defined over $\cC_t^{\mu^\theta, t}$. The Green's function of the domain $\pi_{Q^{\mu^\theta,t}}(\Omega^{\mu^\theta,t})$ can be expressed in terms of $\phi$:
\begin{align*}
G(Z,Z';\mu^\theta,t)=\log\left|\frac{\phi(Z)-\phi(Z')}{\phi(Z)-\overline{\phi(Z')}}\right|,
\end{align*}
and the Schiffer kernel on $\cC_t^{\mu^\theta,t}$ is
\begin{align}\label{e:Schiffer}
B(Z,Z';\mu^\theta,t)=\del_z \del_{z'}G(Z,Z';\mu^\theta,t)\rd z\rd z'=\frac{\phi'(Z)\phi'(Z')}{(\phi(Z)-\phi(Z'))^2}\rd z\rd z'. 
\end{align}
The integral of the Schiffer kernel
\begin{align}\label{e:third}
Q_{Z_1, Z_0}(Z; \mu^\theta,t)\deq \int_{Z'=Z_0}^{Z'=Z_1}B(Z,Z';\mu^\theta,t)
=\frac{\phi'(Z)\rd z}{\phi(Z)-\phi(Z_1)}
-\frac{\phi'(Z)\rd z}{\phi(Z)-\phi(Z_0)}. 
\end{align}
is a meromorphic $1$-form (called third kind form) over  $\cC_t^{\mu^\theta,t}$, having only poles at $Z_0, Z_1$ with residual $-1,1$.

We recall from \eqref{e:residual0} and \eqref{residual1}, as a meromorphic $1$-form over 
the Riemann surface $\cC_t^{\mu^\theta,t}$, $f(Z; \mu^\theta, t)\rd z$ has residual along the top and bottom boundaries of $\Omega^{\mu^\theta,t}$.
Therefore we can write down $f_t(Z;\mu^\theta,t)\rd z$ explicitly as a contour integral of the third kind form \eqref{e:third}: for $Z=(f,z)\in \cC_t^{\mu^\theta,t}$
\begin{align}\begin{split}\label{e:ftexp}
f_t(Z;\mu^\theta,t)\rd z
&=\frac{1}{2\pi\ri}\oint_{\sfC_0} Q_{Z', Z_0}(Z; \mu^\theta, t)f(Z';\mu^\theta,t)\rd z(Z')\\
&+\frac{1}{2\pi\ri}\oint_{\sfC_1} Q_{Z', Z_0}(Z; \mu^\theta, t)f(Z';\mu^\theta,t)\rd z(Z'),
\end{split}\end{align}
where the contour $\sfC_0$ encloses the bottom boundary of $\Omega^{\mu^\theta,t}$, and 
the contour $\sfC_1$ encloses the top boundary of $\Omega^{\mu^\theta,t}$.
We remark that the total residual of $f(Z;\mu^\theta,t)$ sum up to zero, the righthand side of the above expression does not have pole at $Z_0$, and is in fact independent of $Z_0$.
Using \eqref{e:residual0} and \eqref{residual1}, we can write \eqref{e:ftexp} in terms of the residuals of $f_t(Z;\mu^\theta,t)\rd z$,
\begin{align}\begin{split}\label{e:fexpf}
f_t(Z;\mu^\theta,t)\rd z
&=\int Q_{X_\theta(x), Z_0}(Z; \mu^\theta,t)\rd \mu^\theta(x)\\
&+\int Q_{C_\theta(x), Z_0}(Z; \mu^\theta,t)\left(-\rho_C(x)\rd x+(1+\tau-t)\rd \gamma^{\mu^\theta,t}(x)\rho_C(x)\right),
\end{split}\end{align}
where $\{X_\theta(x)=(f_t(x;\mu^\theta,t), x)\}_{x\in \supp(\mu^\theta)}$ corresponds to the bottom boundary of $\Omega^{\mu^\theta,t}$, and $\{C_\theta(x)=(f_{1+\tau}(x;\mu^\theta,t), x-(1+\tau-t)f_{1+\tau}(x;\mu^\theta,t))\}_{x\in\supp(\mu_C)}$ corresponds to the top boundary of $\Omega^{\mu^\theta,t}$.
%

We have the following simple formula for the derivative of $f_t(Z;\mu^\theta,t)\rd z$ with respect to $\theta$ in terms of the third kind form \eqref{e:third}.  
\begin{proposition}\label{p:derft}
The derivative of $f_t(Z;\mu^\theta,t)\rd z$ with respect to $\theta$ is given by
\begin{align*}
\del_\theta f_t(Z;\mu^\theta,t)\rd z = \int Q_{X_\theta(x), Z_0}(Z; \mu^\theta,t)\rd(\mu^1(x)-\mu^0(x)),
\end{align*}
where $\{X_\theta(x)=(f_t(x;\mu^\theta,t), x)\}_{x\in \supp(\mu^\theta)}$ 
\end{proposition}
\begin{proof}
The derivative of $f_t(Z;\mu^\theta,t)\rd z$ with respect to $\theta$ consists of three parts, either the derivative hits $\rd \mu^\theta(x)$ in the first term of \eqref{e:fexpf}, or $\rd \gamma^{\mu^\theta,t}(x)\rho_C(x)$ in the last term of \eqref{e:fexpf}, or the derivative hits one of those third kind form. In the first case, $\del_\theta \rd \mu^\theta(x)=\rd(\mu^1(x)-\mu^0(x))$, and we get the term
\begin{align}\label{e:remain}
 \int Q_{X_\theta(x), Z_0}(Z; \mu^\theta,t)\rd(\mu^1(x)-\mu^0(x)).
\end{align} 
In the second case, we get the term 
\begin{align}\begin{split}\label{e:dgamma}
&\phantom{[{}={}}\int Q_{C_\theta(x), Z_0}(Z; \mu^\theta,t)\left((1+\tau-t)\rd \del_\theta(\gamma^{\mu^\theta,t}(x))\rho_C(x)\right)\\
&=\frac{1}{\pi}\int Q_{C_\theta(x), Z_0}(Z; \mu^\theta,t)\Im\rd\left[\del_\theta z(C_\theta(x)) f_t(C_\theta(x);\mu^\theta,t)\right]\\
&=-\frac{1}{\pi}\int B(Z, C_\theta(x); \mu^\theta,t) \del_\theta z(C_\theta(x))\Im\left[ f_t(C_\theta(x);\mu^\theta,t)\right],
\end{split}\end{align}
where $\{C_\theta(x)=(f_{1+\tau}(x;\mu^\theta,t), x-(1+\tau-t)f_{1+\tau}(x;\mu^\theta,t))\}_{x\in\supp(\mu_C)}$ and we did an integration by part in the last line.

Finally, we study the case that the derivative  hits one of those third kind forms in \eqref{e:fexpf}.
We recall from \eqref{e:third} that the third kind forms are obtained from integrating the Schiffer kernel. The derivative of the Schiffer kernel $B(Z,Z';\mu^\theta,t)\rd z\rd z'$ with respect to the branch points $Z_1^\theta, \bar Z_1^\theta, \cdots, Z_w^\theta, \bar Z_w^\theta$ is described by the Rauch Variational Formula, the derivative of the Schiffer kernel $B(Z,Z';\mu^\theta,t)\rd z\rd z'$ with respect to the change of the upper boundary curve $C_\theta(x)$ is described by Hadamard's Variation Formula. We collect them in Theorem \ref{thm:Rauch}. For the derivative of the Schiffer kernel $B(Z,Z';\mu^\theta,t)\rd z\rd z'$ with respect to the branch points $Z_1^\theta, \bar Z_1^\theta, \cdots, Z_w^\theta, \bar Z_w^\theta$, using \eqref{e:Rauch}, the derivative corresponding to the branch point $Z_i^\theta$ is
\begin{align}\begin{split}\label{e:branchcc}
&\phantom{{}={}}\frac{\del_\theta Z_i^\theta}{2\pi\ri}
\left(\oint_{\sfS_i} \frac{B(Z,Z'';\mu^\theta,t)}{\rd z(Z'')} \int Q_{X_\theta(x), Z_0}(Z''; \mu^\theta,t)\rd \mu^\theta(x)\right.\\
&\left.+\oint_{\sfS_i}\frac{B(Z,Z'';\mu^\theta,t)}{\rd z(Z'')} \int Q_{X_\theta(x), Z_0}(Z''; \mu^\theta,t) \left(-\rho_C(x)\rd x+(1+\tau)\rd \gamma^{\mu^\theta,t}(x)\rho_C(x)\right)\right),
  \\
&=\frac{\del_\theta Z_i^\theta}{2\pi\ri}
\oint_{\sfS_i} \frac{B(Z,Z'';\mu^\theta,t)}{\rd z(Z'')}f(Z'';\mu^\theta,t)\rd z(Z'')\\
&=\frac{\del_\theta Z_i^\theta}{2\pi\ri}
\oint_{\sfS_i} B(Z,Z'';\mu^\theta,t)f(Z'';\mu^\theta,t)=0,
\end{split}\end{align}
where in the last line the integral vanishes, because the integrand is holomorphic at the 
branch point $Z_i^\theta$. Similarly the contribution from other branch points also vanishes. 

For the derivative of the Schiffer kernel $B(Z,Z';\mu^\theta,t)\rd z\rd z'$ with respect to the change of the upper boundary curve $C_\theta(x)$, using \eqref{e:bbcurve}, it gives
\begin{align}\begin{split}\label{e:boundarycc}
&\phantom{{}={}}\frac{1}{\pi}
\left(\int B(Z,C_\theta(x);\mu^\theta,t)\del_\theta z(C_\theta(x)) \int \Im\left[\frac{Q_{X_\theta(y),Z_0}(C_\theta(x);\mu^\theta,t)}{\rd z(C_\theta(x))}\right] \rd \mu^\theta(y)\right.\\
&\left.+\int B(Z,C_\theta(x);\mu^\theta,t)\del_\theta z(C_\theta(x)) \int \Im\left[\frac{Q_{X_\theta(y),Z_0}(C_\theta(x);\mu^\theta,t)}{\rd z(C_\theta(x))}\right]\left(-\rho_C(y)\rd y+(1+\tau)\rd \gamma^{\mu^\theta,t}(y)\rho_C(y)\right)\right),
  \\
&=\frac{1}{\pi}
\int B(Z,C_\theta(x);\mu^\theta,t)\del_\theta z(C_\theta(x))\Im[f_t(C_\theta(x);\mu^\theta,t)].
\end{split}\end{align}
We notice that \eqref{e:boundarycc} cancels with \eqref{e:dgamma}. Proposition \ref{p:derft} follows from combining \eqref{e:remain}, \eqref{e:dgamma}, \eqref{e:branchcc} and \eqref{e:boundarycc}.

\end{proof}


\begin{theorem}\label{thm:Rauch}
The derivative of the Schiffer kernel $B(Z,Z';\mu^\theta,t)$ with respect to any branch point $Z_i^\theta$ is given by
\begin{align}\begin{split}\label{e:Rauch}
\frac{\del_\theta Z_i^\theta}{2\pi\ri} \oint_{\sfS_i} B(Z,Z'';\mu^\theta,t)B(Z'',Z';\mu^\theta,t)/\rd z(Z''),
\end{split}\end{align}
where the contour $\sfS_i$ encloses $Z_i^\theta$. The derivative of the Schiffer kernel $B(Z,Z';\mu^\theta,t)$ with respect to the boundary curve $\{z(C_\theta(x)): x\in \supp(\rho_C)\}$ is given by
\begin{align}\begin{split}\label{e:bbcurve}
\frac{1}{\pi}\int B(Z,C_\theta(x);\mu^\theta,t)\del_\theta z(C_\theta(x)) \Im\left[\frac{B(C_\theta(x),Z';\mu^\theta,t)}{\rd z(C_\theta(x))}\right].
\end{split}\end{align}
\end{theorem}
\begin{proof}
Formula \eqref{e:Rauch} is the Rauch variational formula \cite{MR110798,MR110799} and \cite[(2.4)]{MR2131384}. Formula \eqref{e:bbcurve} follows from  Hadamard's Variation Formula for Green's Function \cite[(1)]{schiffer1946hadamard}. We recall the relation between Green's function and the Schiffer kernel \eqref{e:Schiffer}, i.e. the Schiffer kernel is obtained from the Green's function by taking derivatives
\begin{align}\label{e:GBcopy}
G(Z,Z';\mu^\theta,t)=\del_z \del_{z'}B(Z,Z';\mu^\theta,t)\rd z\rd z'.
\end{align}
Formula  \eqref{e:bbcurve} follows from taking derivatives on both sides of \cite[(1)]{schiffer1946hadamard} and plugging in \eqref{e:GBcopy}.
\end{proof}

We recall $g_t(Z;\mu,t)$ from \eqref{e:gszmut}, its derivative with respect to the measure $\mu$ can be expressed explicitly using Proposition \ref{p:derft}. Proposition \ref{p:derft} gives that 
\begin{align}\label{e:deltagt}
\frac{\delta f_t(Z;\mu,t)}{\delta \mu}= \frac{Q_{Z', Z_0}(Z;\mu,t)}{\rd z(Z)}.
\end{align}
We can take derivative with respect to $z'$ on both sides of \eqref{e:deltagt}
\begin{align}\label{e:deltagt}
\del_{z'}\frac{\delta f_t(Z;\mu,t)}{\delta \mu}= \frac{B(Z,Z';\mu,t)}{\rd z(Z)\rd z(Z')},
\end{align}
which is the Schiffer kernel \eqref{e:Schiffer}. We recall $g_t(Z;\mu,t)$ from \eqref{e:gszmut}, by taking $s=t$. Then  \eqref{e:deltagt} gives
\begin{align}\label{e:dmugt}
\del_{z'}\frac{\delta g_t(Z;\mu,t)}{\delta \mu}= \frac{B(Z,Z';\mu,t)}{\rd z(Z)\rd z(Z')}-\frac{1}{(z(Z)-z(Z'))^2}.
\end{align}
The double pole of the Schiffer kernel along the diagonal cancels with the double pole of $1/(z(Z)-z(Z'))^2$. We conclude that $\del_{z'}(\delta g_t(Z;\mu,t)/\delta \mu)$ no longer has double poles along the diagonal. 

We recall the region $\Omega^{\mu,t}$ from \eqref{e:burgereq2}, 
and $\cC_t^{\mu,t}$ can be identified with two copies of $\Omega^{\mu,t}$ using the map $\pi_{Q^{\mu,t}}$ from \eqref{e:cover2}. The bottom boundary of $\Omega^{\mu,t}$ corresponds to a cut in $\cC_t^{\mu,t}$. 
In the rest of this section, we identify $\cC^{\mu,t}$ with two copies of $\Omega^{\mu,t}$ gluing together. We have the following covering map:
\begin{align}\label{e:zmatt2}
(f_t(Z;\mu, t)),z)\in \cC_t^{\mu,t}\mapsto z\in\bC,
\end{align}
and the following proposition is on the branch points of the covering map \eqref{e:zmatt2}.
\begin{proposition}\label{p:branchloc}
Under the Assumptions \ref{a:reg}, \ref{a:ncritic}, there exists a neighborhood of $\supp(\mu)$ such that it is a bijection with its preimage under the covering map \eqref{e:zmatt2}. As a consequence, the covering map \eqref{e:zmatt2} does not have a branch point in a neighborhood of the bottom boundary of $\Omega^{\mu,t}$. 
\end{proposition}

\begin{proof}
The points $(f,z)\in \cC_t^{\mu,t}$ with $z$ real, i.e. $z\in \bR$, correspond to the left and right boundaries of $\Omega^{\mu,t}$ where $f_t(x;\mu,t)\in \bR$, and the bottom boundary $\{(x,t):x\in\supp(\mu) \}$. Under our assumption that $\rd\mu_C$ has a density supported on one interval, the top boundary of $\Omega^{\mu,t}$ is one interval. The rest boundary (left, right and bottom boundary) of $\Omega^{\mu,t}$ is connected. So it is a bijection to its image under the covering map \eqref{e:zmatt2}, which is an interval $I\subset \bR$. Since $I$ is compact, there exists a neighborhood of $I$, such that it is a bijection with its preimage under the covering map \eqref{e:zmatt2}. The support of $\mu$ is the image of the bottom boundary of $\cC_t^{\mu,t}$ under \eqref{e:zmatt2}, we have $\supp(\mu)\subset I$. We especially have that there exists a neighborhood of $\supp(\mu)$ such that it is a bijection with its preimage under the covering map \eqref{e:zmatt2}.

If there is a branch point in a small neighborhood of the bottom boundary of  $\cC_t^{\mu,t}$, then there are $W\neq W'$ close to the branch point, such that $z(W)=Z(W')$. This is impossible, since a neighborhood of $\supp(\mu)$ is a bijection with its preimage under the covering map \eqref{e:zmatt2}.
\end{proof}

%
%

Thanks to Proposition \ref{p:branchloc}, using the covering map \eqref{e:zmatt2} we can identify a neighborhood of the bottom boundary of $\cC_t^{\mu,t}$, with a neighborhood of $\supp(\mu)$ over $\bC$. In this neighborhood, $g_t(z;\mu,t)=f_t(z;\mu,t)-\int \rd \mu(x)/(z-x)$ is analytic, so is its derivatives with respect to $\mu$. .
We conclude the following estimates for the derivatives of $g_t(z;\mu,t)$ with respect to the measure $\mu$, which will be used in Sections \ref{s:eq} and \ref{s:rigidity}.

\begin{proposition}\label{p:derg}
Under Assumptions \ref{a:reg} and \ref{a:ncritic}, for any $0\leq t\leq 1$ and probability measure $\mu$, $g_t(z;\mu,t)$ satisfies: for any $z,z',z''$ in a neighborhood of the bottom boundary of $\cC_t^{\mu,t}$,
\begin{align*}
\left|\del_z g(z;\mu,t)\right|,\left|  \del_{z'}\frac{\delta g_t(z;\mu,t)}{\delta \mu}\right|, \left|\del_{z'}\frac{\delta \del_x g_t(z;\mu,t)}{\delta \mu}\right|\lesssim 1,
\end{align*}
and the higher derivatives are bounded
\begin{align*}
\left|\del_{z''}\frac{\delta \del_{z'} \frac{\delta g_t(z;\mu,t)}{\delta \mu}}{\delta \mu}\right|\lesssim 1.
\end{align*}

\end{proposition}


\section{Random Walk Representation}\label{s:walk}

In this section, we rewrite nonintersecting Brownian bridges  as a random walk in the Weyl chamber $\Delta_n$. We recall the transition probability from \eqref{e:tp0}
\begin{align}\label{e:tp}
p_t(\bmx, \bmy)=\det\left[\sqrt{\frac{n}{2\pi t}}e^{-n(x_i-y_j)^2/(2t)}\right]_{1\leq i,j\leq n}.
\end{align}
The density function for the nonintersecting Brownian bridge $\{\bmx(t)=(x_1(t), x_2(t),\cdots, x_n(t))\}_{0\leq t\leq 1}$ with starting point $\bmx(0)=\bma=(a_1,a_2,\cdots,a_n)\in \Delta_n$, and ending point  $\bmx(1)=\bmb=(b_1,b_2,\cdots,b_n)\in \Delta_n$ at times $0\leq t_1<t_2<\cdots<t_k\leq 1$ is given by
\begin{align}\label{e:density}
p(\bmx(t_1), \bmx(t_2), \cdots, \bmx(t_k))=\frac{p_{t_1}(\bma, \bmx(t_1))p_{t_2-t_1}(\bmx(t_1), \bmx(t_2))\cdots p_{1-t_k}(\bmx(t_k), \bmb)}{p_1(\bma,\bmb)}.
\end{align}
The above probability can be extended by taking limit to the case that $\bma, \bmb$ belong to $\overline{\Delta}_n$, the closure of the Wyel chamber.

Let $\bmx=(x_1, x_2,\cdots,x_n)\in \Delta_n$, and the transition probability
\begin{align}\label{e:defHt}
p_{1-t}(\bmx,\bmb)=\det \left[\sqrt{\frac{n}{2\pi (1-t)}}e^{-n(x_i-b_j)^2/(2(1-t))}\right]_{0\leq i,j\leq n} .
\end{align}
As a function of $\bmx$, $p_{1-t}(\bmx, \bmb)$ is the heat kernel in the Weyl chamber, and it satisfies 
\begin{align}\label{e:heat}
-\del_t  p_{1-t}(\bmx,\bmb)
=\frac{1}{2n}\Delta p_{1-t}(\bmx,\bmb).
\end{align}
Using $p_{1-t}(\bmx,\bmb)$, we can rewrite the nonintersecting Brownian bridge \eqref{e:density} as the following random walk:
\begin{align}\label{e:randomwalk}
\rd x_i(t)
&=
\frac{1}{\sqrt n} \rd B_i(t)+\frac{1}{n}\del_{x_i}\log p_{1-t}(\bmx(t),\bmb)\rd t,\quad 1\leq i\leq n,
\end{align}
where $\bmx(t)=(x_1(t), x_2(t), \cdots, x_n(t))$ and $\{B_1(t), B_2(t), \cdots, B_n(t)\}_{0\leq t\leq 1}$ are standard Brownian motions.
\begin{proposition}\label{p:randomwalk}
The joint law of the random walk $\{\bmx(t)=(x_1(t), x_2(t), \cdots, x_n(t))\}_{0\leq t\leq 1}$, as defined in \eqref{e:randomwalk}, with initial data $\bmx(0)=\bma$ is the same as the nonintersecting Brownian bridge \eqref{e:density}. 
\end{proposition}
\begin{proof}
We just need to check that the transition probability of \eqref{e:randomwalk} is given by
\begin{align}\label{e:transit}
p(\bmx(t)=\bmx|\bmx(s)=\bmy)
=\frac{p_{t-s}(\bmy, \bmx)p_{1-t}(\bmx,\bmb)}{p_{1-s}(\bmy,\bmb)}, \quad 0\leq s<t\leq 1.
\end{align}
We can take time derivative in both sides of \eqref{e:transit}
\begin{align*}\begin{split}
\del_{t}p(\bmx(t)=\bmx|\bmx(s)=\bmy)
&=\frac{\del_t p_{t-s}(\bmy, \bmx)p_{1-t}(\bmx,\bmb)+p_{t-s}(\bmy, \bmx)\del_t p_{1-t}(\bmx,\bmb)}{p_{1-s}(\bmy)}\\
&=\frac{1}{2n}\frac{\Delta p_{t-s}(\bmy, \bmx) p_{1-t}(\bmx,\bmb)-p_{t-s}(\bmy, \bmx)\Delta p_{1-t}(\bmx,\bmb)}{p_{1-s}(\bmy,\bmb)}\\
&=\frac{1}{2n}\frac{\Delta (p_{t-s}(\bmy, \bmx) p_{1-t}(\bmx,\bmb))-2\nabla (p_{t-s}(\bmy, \bmx) \nabla\log p_{1-t}(\bmx,\bmb))}{p_{1-s}(\bmy,\bmb)}
\end{split}\end{align*}
The righthand side is the generator for the random walk \eqref{e:randomwalk}, the claim follows.
\end{proof}

\subsection{Nonintersecting Brownian Bridges with random boundary data}

As a function of $\bmx$, $p_{1-t}(\bmx,\bmb)$ is the heat kernel in the Weyl chamber. It is singular as $t\rightarrow 1$. In fact it converges to the delta mass $\delta_{\bmb}$ as $t\rightarrow 1$.
As a consequence, the drift term $\log p_{1-t}(\bmx(t),\bmb)$ in the random walk \eqref{e:randomwalk}, as $t$ approaches $1$,  is singular and hard to analyze. Instead of directly studying the random walk \eqref{e:randomwalk}, which corresponds to  nonintersecting Brownian bridges with boundary data $\bma, \bmb$, we study a weighted version of it. The weighted version corresponds to nonintersecting Brownian bridges with random boundary data $\bmb$.

We recall the functional $W_t(\mu)$ from \eqref{e:varWt}: for any probability measure,
\begin{align}\label{e:varWtcopy}
W_t(\mu)=-\frac{1}{2}\inf\left(\int_t^{1+\tau}\int_\bR \rho_s\left(u_s^2+\frac{\pi^2}{3}\rho_s^2\right)\rd x\rd s+\Sigma(\mu)+\Sigma(\mu_B)\right),
\end{align}
the $\inf$ is taken over all the pairs $\{(\rho_s,u_s)\}_{t\leq s\leq 1+\tau}$ such that $\del_s \rho_s+\del_x(\rho_su_s)=0$ in the sense of distributions, $\{\rho_s\}_{t\leq s\leq 1}\in{\mathcal{C}}([t,1+\tau],\cM(\bR))$  and its initial and terminal data are  given by 
\begin{align*}
\lim_{s\rightarrow t+}\rho_s(x)\rd x=\mu,\quad \lim_{s\rightarrow (1+\tau)-}\rho_s(x)\rd x=\mu_C,
\end{align*}
where convergence holds in the weak sense. We will construct nonintersecting Brownian bridges with boundary data $\bmb$ weighted by $\Delta(\bmb)e^{n^2W_1(\bmb)}$, where $\Delta(\bmb)$ is the Vandermonde determinant \eqref{e:Vand}, and $W_1(\bmb)\deq W_1((1/n)\sum_{i=1}^n \delta_{b_i})$.

The density for the weighted nonintersecting Brownian bridges $\{\bmx(t)=(x_1(t), x_2(t),\cdots, x_n(t))\}_{0\leq t\leq 1}$ with starting point $\bmx(0)=\bma=(a_1,a_2,\cdots,a_n)\in \Delta_n$ at times $0\leq t_1<t_2<\cdots<t_k\leq 1$ is given by
\begin{align}\label{e:wdensity}
p(\bmx(t_1), \bmx(t_2), \cdots, \bmx(t_k),\bmb)=\frac{p_{t_1}(\bma, \bmx(t_1))p_{t_2-t_1}(\bmx(t_1), \bmx(t_2))\cdots p_{1-t_k}(\bmx(t_k), \bmb)\Delta(\bmb)e^{n^2W_1(\bmb)}}{Z(\bma)},
\end{align}
where the partition function $Z(\bma)$ is given by
\begin{align*}
{Z(\bma)}\deq \int_{\bmb\in \Delta_n} p_1(\bma, \bmb)\Delta(\bmb)e^{n^2W_1(\bmb)}\rd \bmb.
\end{align*}
The probability density for the boundary $\bmb$ is given by
\begin{align}\label{e:boundarydensity}
p(\bmb)=\frac{p_{1}(\bma, \bmb)\Delta(\bmb)e^{n^2W_1(\bmb)}}{Z(\bma)}.
\end{align}
The above probability can be extended by taking limit to the case that $\bma$ belongs to $\overline{\Delta}_n$, the closure of the Wyel chamber. The weighted nonintersecting Brownian bridges \eqref{e:wdensity} can be sampled in two steps. First we sample the boundary $\bmb$ using \eqref{e:boundarydensity}. Then we sample a nonintersecting Brownian bridge with boundary data $\bma, \bmb$. This gives a sample of the weighted nonintersecting Brownian bridges \eqref{e:wdensity}. As we will show in Section \ref{s:rigidity}, for $\bmb$ sampled from  \eqref{e:boundarydensity}, its empirical density $\mu_{B_n}$ concentrates around $\mu_B$.

Similarly to  nonintersecting Brownian bridges \eqref{e:density}, the weighted nonintersecting Brownian bridges \eqref{e:wdensity} can also be interpreted as a random walk. Let 
\begin{align}\label{e:defH}
H_t(\bmx)=\int_{\bmb\in \Delta_n} p_{1-t}(\bmx, \bmb)\Delta(\bmb)e^{n^2W_1(\bmb)}\rd \bmb,\quad \bmx\in \Delta_n.
\end{align}
Then $H_t(\bmx)$ also satisfies the backward heat equation \eqref{e:heat}
\begin{align}\label{e:wbackheat}
-\del_t H_t(\bmx)=\frac{1}{2n}\Delta H_t(\bmx),
\end{align}
and the boundary condition is given by
\begin{align}\label{e:boundary}
\lim_{t\rightarrow 1}H_t(\bmx)=\Delta(\bmx)e^{n^2W_1(\bmx)}.
\end{align}
The same argument as for Proposition \ref{p:randomwalk}, we have
\begin{proposition}
The joint law $\{\bmx(t)=(x_1(t), x_2(t), \cdots, x_n(t))\}_{0\leq t\leq 
1}$ of the following random walk
\begin{align}\label{e:weightrandomwalk}
\rd x_i(t)
&=
\frac{1}{\sqrt n} \rd B_i(t)+\frac{1}{n}\del_{x_i}\log H_{t}(\bmx(t))\rd t,\quad 1\leq i\leq n,
\end{align}
with initial data $\bmx(0)=\bma$ is the same as the weighted nonintersecting Brownian bridges \eqref{e:wdensity}. 
\end{proposition}

\subsection{Hamilton-Jacobi Equation}\label{s:HJ}
We denote the minimizer of \eqref{e:varWtcopy} as $\{(\rho_s^*(\cdot;\mu,t), u_s^*(\cdot;\mu,t))\}_{t\leq s\leq 1+\tau}$. In this section we derive the Hamilton-Jacobi equation for $W_t(\mu)$ as in \eqref{e:varWtcopy}.

The Euler-Lagrange equation, see \cite[Theorem 2.1]{MR2034487}, for the variational problem \eqref{e:varWtcopy} gives
\begin{align}\label{e:ELeq}
\del_s u_s^*=-\frac{1}{2}\del_x\left((u_s^*)^2-\pi^2(\rho_s^*)^2\right),\quad (x,s)\in \Omega^{\mu,t}= \{(x,s)\in \bR\times (t,1+\tau): \rho_s^*(x;\mu,t)>0\}.
\end{align}
We can compute the derivative of the functional $W_t(\mu)$ by perturbation. Let $m_s^*(x;\mu,t)=\rho_s^*(\cdot;\mu,t) u_s^*(\cdot;\mu,t)$ for $t\leq s\leq {1+\tau}$, and $\{(\Delta \rho_s(\cdot;\mu,t), \Delta m_s(\cdot;\mu,t))\}_{t\leq s\leq {1+\tau}}$ be the difference of the variational problem \eqref{e:varWtcopy} with initial data $\mu+\Delta\mu$ and $\mu$. Then 
$\del_s \Delta \rho_s=-\del_x \Delta m_s$, and we have
\begin{align*}\begin{split}
&\phantom{{}={}}W_t(\mu+\Delta\mu)
=-\frac{1}{2}\left(\int_t^{1+\tau}\int_\bR \left(\frac{(m_s^*+\Delta m_s)^2}{\rho_s^*+\Delta \rho_s}+\frac{\pi^2}{3}(\rho^*_s+\Delta \rho_s)^3\right)\rd x\rd s+\Sigma(\mu+\Delta \mu)+\Sigma(\mu_B)\right)\\
&=W_t(\mu)-\frac{1}{2}\left(\int_t^{1+\tau}\int_\bR \left(\frac{2m_s^*}{\rho_s^*}\Delta m_s+\left(\pi^2(\rho_s^*)^2-\frac{(m_s^*)^2}{(\rho_s^*)^2}\right)\Delta \rho_s\right)\rd x\rd s+2\int\ln |x-y|\rd \mu(x)\rd\Delta \mu(y)\right)+\OO(\Delta^2)\\
&=W_t(\mu)+\left(\left(\int_{-\infty}^x u_t^*(y)\rd y\right)\rd\Delta\mu(x)-\int\ln |x-y|\rd \mu(y)\rd\Delta \mu(x)\right)+\OO(\Delta^2),
\end{split}\end{align*}
where in the last line, we used \eqref{e:ELeq} and performed an integration by part. It follows that, $\mu$ almost surely
\begin{align}\label{e:derWtmu}
\frac{\del}{\del x} \frac{\delta W_t}{\delta \mu}=u_t^*(x; \mu,t)-H(\mu)(x),\quad H(\mu)(x)=\int\frac{1}{x-y}\rd\mu(y).
\end{align}
We notice that the expression \eqref{e:derWtmu} is exactly $g_t(x;\mu,t)$ as defined in \eqref{e:gtut2}. We get
\begin{align}\label{e:Wgt}
\frac{\del}{\del x} \frac{\delta W_t}{\delta \mu}=g_t(x;\mu,t),
\end{align}
and
 it can be extended to a meromorphic function over the Riemann surface corresponding to the minimizer of \eqref{e:varWtcopy}.

Similarly, we can compute the time derivative of the functional $W_t(\mu)$. Let $m_s^*(x;\mu,t)=\rho_s^*(\cdot;\mu,t) u_s^*(\cdot;\mu,t)$ for $t\leq s\leq {1+\tau}$, and $\{(\Delta \rho_s(\cdot;\mu,t), \Delta m_s(\cdot;\mu,t))\}_{t\leq s\leq {1+\tau}}$ be the difference of the variational problem \eqref{e:varWtcopy} with initial time $t+\Delta t$ and $t$. We have
\begin{align}\begin{split}\label{e:Wtdertt}
&\phantom{{}={}}W_ {t+\Delta t} ( \mu )-W_ {t} ( \mu )\\
&=-\frac{1}{2}\int_{t+\Delta t}^{1+\tau}\int_\bR \left(\frac{(m_s^*+\Delta m_s)^2}{\rho_s^*+\Delta \rho_s}+\frac{\pi^2}{3}(\rho^*_s+\Delta \rho_s)^3\right)\rd x\rd s +\frac{1}{2}\int_{t}^{{1+\tau}}\int_\bR \left(\frac{(m_s^*)^2}{\rho_s^*}+\frac{\pi^2}{3}(\rho^*_s)^3\right)\rd x\rd s\\
&=\int_\bR \left(\int_{-\infty}^x \int _\bR u_t^*(y)\rd y\right)\Delta \rho_{t+\Delta t}\rd x
+\frac{\Delta t}{2}\int_\bR \left(\frac{(m_t^*)^2}{\rho_t^*}+\frac{\pi^2}{3}(\rho^*_t)^3\right)\rd x+\OO((\Delta t)^2) \\
&=-\Delta t\int _\bR u_t^*m_t^*\rd x
+\frac{\Delta t}{2}\int_\bR \left(\frac{(m_t^*)^2}{\rho_t^*}+\frac{\pi^2}{3}(\rho^*_t)^2\right)\rd x+\OO((\Delta t)^2)\\
&=-
\frac{\Delta t}{2}\int _\bR \left((u_t^*)^2-\frac{\pi^2}{3}(\rho_t^*)^2\right) \rho_t^*\rd x
+\OO((\Delta t)^2).
\end{split}\end{align}
Using \eqref{e:derWtmu}, and the following identity of Hilbert transfrom
\begin{align*}
\frac{\pi^2}{3}\int_\bR (\rho^*_t)^3\rd x=\int_\bR (H(\rho_t^*)(x))^2\rd \rho_t^*(x),
\end{align*}
we 
 can  represent \eqref{e:Wtdertt}, and get the  Hamilton-Jacobi equation of $W_t(\mu)$,
\begin{align}\begin{split}\label{e:Wt}
-\del_t W_t(\mu)
&=\frac{1}{2}\int \left(\frac{\del}{\del x} \frac{\delta W_t}{\delta \mu}+H(\mu)(x)\right)^2\rd \mu(x)
-\frac{1}{2}\int (H(\mu)(x))^2\rd \mu(x)\\
&=\frac{1}{2}\int \left(\frac{\del}{\del x} \frac{\delta W_t}{\delta \mu}\right)^2\rd \mu(x)
+\int \frac{\del}{\del x} \frac{\delta W_t}{\delta \mu} H(\mu)(x)\rd \mu(x).
\end{split}\end{align}
 Using $g_t(x;\mu,t)$ as in \eqref{e:Wgt},
we can further rewrite the equation \eqref{e:Wt} as
\begin{align}\label{e:dWt}
-\del_t W_t(\mu)=\frac{1}{2}\int \left(g_t(x; \mu,t)\right)^2\rd \mu(x)
+\frac{1}{2}\int \frac{g_t(x;\mu,t)-g_t(y;\mu,t)}{x-y}\rd \mu(x)\rd \mu(y),
\end{align}
where when $x=y$,
\begin{align}\label{e:defdg}
\frac{g_t(x;\mu,t)-g_t(y;\mu,t)}{x-y}=\del_x g_t(x;\mu,t).
\end{align}
If we take the measure $\mu=(1/n)\sum_{i=1}^n\delta_{x_i}$, 
then we can rewrite \eqref{e:dWt} more explicitly
\begin{align}\label{e:newexp}
-\del_t W_t(\mu)=\frac{1}{2n^2}\sum_i\del_{x_i}g_t(x_i;\mu,t)+\frac{1}{2n}\sum_{i=1}^n\left(g_t(x_i; \mu,t)\right)^2
+\frac{1}{2n^2}\sum_{i\neq j} \frac{g_t(x_i;\mu,t)-g_t(x_j;\mu,t)}{x_i-x_j}.
\end{align}
We remark that $\del_{x_i}g_t(x_i;\mu,t)$ is slightly ambiguous, since $\mu$ also depends on $x_i$. We can interpret \eqref{e:defdg} as the definition of $\del_{x}g_t(x;\mu,t)$
in this way we have
\begin{align*}
\del_{x_i}(g_t(x_i;\mu,t))
=\del_{x_i}g_t(x_i;\mu,t)+\frac{1}{n}\del_{x_i}\frac{\delta g_t(x_i;\mu,t)}{\delta \mu}.
\end{align*}

\subsection{An Ansatz}\label{s:ansatz}
We recall the backward heat equation \eqref{e:wbackheat} for $H_t(\bmx)$, which is defined on the Weyl chamber $\Delta_n$. Since in the definition \eqref{e:defH} of $H_t(\bmx)$, $p_{1-t}(\bmx,\bmb)$ extends naturally to an anti-symmetric function on $\bR^n$. We also extend $H_t(\bmx)$ to an anti-symmetric function on $\bR^n$. 
To kill the anti-symmetric factors, we introduce
\begin{align}\label{e:copf}
S_t(\bmx)=\frac{1}{n}\log \frac{H_t(\bmx)}{\prod_{i<j}(x_i-x_j)}.
\end{align} 
In this way we have that $S_1(\bmx)=nW_1(\bmx)$. With the function $S_t(\bmx)$, we can rewrite the random walk \eqref{e:weightrandomwalk} as
\begin{align}\label{e:randomwalk2}
\rd x_i(t)
&=
\frac{1}{\sqrt n} \rd B_i(t)+\frac{1}{n}\sum_{j:j\neq i}\frac{\rd t}{x_i(t)-x_j(t)}+\del_{x_i}S_t(\bmx(t))\rd t,\quad 1\leq i\leq n,
\end{align}
Those are the Dyson's Brownian motion \cite{MR0148397} with time dependent drifts given by $\{\del_{x_i}S_t\}_{1\leq i\leq n}$.

By plugging \eqref{e:copf} into \eqref{e:wbackheat}, we get the following equation of $S_t(\bmx)$.
\begin{align}\label{e:key}
-\del_t S_t(\bmx)
=\frac{1}{2n}\Delta S_t+\frac{1}{2}\sum_{i=1}^n (\del_{x_i} S_t)^2+\sum_{i=1}^n\left(\frac{1}{n}\sum_{j: j\neq i}\frac{1}{x_i-x_j}\right)\del_{x_i}S_t.
\end{align}
By taking one more derivative with respect to $x_k$, we obtain
\begin{align}\label{e:derSt}
-\del_t\del_{x_k} S_t
={\frac{1}{2n}\Delta \del_{x_k}S_t}+\sum_{i=1}^n\left(\del_{x_i}  S_t+\frac{1}{n}\sum_{j: j\neq i}\frac{1}{x_i-x_j}\right)\del_{x_i}\del_{x_k} S_t
+\frac{1}{n}\sum_{i: i\neq k}\frac{\del_{x_i} S_t-\del_{x_k} S_t}{(x_i-x_k)^2},
\end{align}
which gives a system of equations for $\{\del_{x_k}S_t\}_{1\leq k\leq n}$. The boundary condition at $t=1$ is given by \eqref{e:boundary}, for $1\leq i\leq n$,
\begin{align}\label{e:dStboundary}
\del_{x_i} S_1(\bmx)
=n\del_{x_i}W_1(\bmx),\quad 1\leq i\leq n.
\end{align}

Up to a factor $1/n$, if we identify $\del_{x_i}S_t$ with $g_t(x_i;\mu,t)$, the righthand side of \eqref{e:key} and the righthand side of \eqref{e:newexp} look almost the same. The only difference is from $\Delta S_t$ and $\sum_i\del_{x_i}g_t(x_i;\mu,t)$.  Motivated by this observation, we make the following ansatz for $\{\del_{x_k}S_t\}_{1\leq k\leq n}$,

\begin{ansatz}\label{a:defE}
Let $\mu=(1/n)\sum_{i=1}^n\delta_{x_i}$ and recall $g(\cdot, \mu,t)$ as defined in \eqref{e:gszmut}, we make the following ansatz
\begin{align*}
\del_{x_k}S_t(\bmx)=g_t(x_k;\mu,t)+\cE^{(k)}_t(\bmx),\quad 1\leq k\leq n.
\end{align*}
\end{ansatz}

The weighted nonintersecting Brownian motion \ref{e:density} is constructed in a way such that 
\begin{align*}
\del_{x_k}S_1(\bmx)=n\del_{x_k}W_1(\bmx)=\left.\frac{\del}{\del x}\frac{\delta W_1(\mu)}{\delta \mu}\right|_{x=x_k}=g_1(x_k;\mu,1),\quad 1\leq i\leq n,
\end{align*}
where $\mu=(1/n)\sum_{i=1}^n \delta_{x_i}$. Thus in the Ansatz \ref{a:defE}, $\cE_1^{(k)}(\bmx)=0$ for $1\leq k\leq n$.

In Section \ref{s:eq1} we derive the equations for $\cE^{(k)}_t$ using the Hamilton-Jacobi type equation \eqref{e:derSt}, and solve for the leading order term of $\cE^{(k)}_t$ in Section \ref{s:solve}. It turns out the correction terms $\cE^{(k)}_t$ for $0\leq t\leq 1$ are of order $\OO(1/n^2)$. They have negligible influence on the random walk \eqref{e:weightrandomwalk}. Then we use the random walk \eqref{e:randomwalk2} to understand the weighted nonintersecting Brownian bridges \eqref{e:wdensity} in Section \ref{s:rigidity}.

\subsection{Calogero-Moser system}

The  Hamilton-Jacobi equation \eqref{e:key} was first derived by A. Matytsin \cite{MR1257846} to compute the asymptotics of the  the Harish-Chandra-Itzykson-Zuber integral formula \eqref{e:HCIZ}.  In \cite{MR1257846},  Matytsin argues that the first term $\Delta S_t/2n$ in \eqref{e:key} is negligible comparing with the remaining terms. Then he ignored the Laplacian term $\Delta S_t/2n$,  
\begin{align}\label{e:key2}
-\del_t S_t(\bmx)
=\frac{1}{2}\sum_{i=1}^n (\del_{x_i} S_t)^2+\sum_{i=1}^n\left(\frac{1}{n}\sum_{j: j\neq i}\frac{1}{x_i-x_j}\right)\del_{x_i}S_t.
\end{align}
and describes the limit of $S_t$ in terms of the complex Burger's equation.

Later it is noticed by G. Menon \cite{GM2017} that \eqref{e:key2}  is exact solvable. It describes the Calogero-Moser system:
\begin{align}\begin{split}\label{e:CM}
&\del_t x_i(t)=v_i(t),\\
&\del_t v_i(t)=-\frac{1}{n^2}\sum_{j: j\neq i}\frac{1}{(x_i-x_j)^3},
\end{split}\end{align}
where $x_1<x_2<\cdots<x_n$ is in the Weyl chamber $\Delta_n$.
We refer to \cite{olshanetsky1994integrable} for a systematic study of systems of Calogero type.
The Hamiltonian of the system \eqref{e:CM} is given by
\begin{align*}
H(\bmx,\bmv)=\frac{1}{2}\sum_i v_i^2-\frac{1}{2n^2}\sum_{i\neq j}\frac{1}{(x_i-x_j)^2}.
\end{align*}
The Lagrangian for the Calogero-Moser system \eqref{e:CM} is the function
\begin{align*}
L(\bmx,\bmv)=\frac{1}{2}\sum_i v_i^2+\frac{1}{2n^2}\sum_{i\neq j}\frac{1}{(x_i-x_j)^2},
\end{align*}
and we associate to any path $\gamma:[t,1]\mapsto \Delta_n$ the action
\begin{align*}
\int_t^1 L(\gamma(s), \dot{\gamma}(s))\rd s.
\end{align*}
The principle of least action asserts that the true path is a minimizer for the above variational principle. In this case, the Euler-Lagrange equations are the equations of motion for the  Calogero-Moser system \eqref{e:CM}:
\begin{align*}
\del_t v_i(t)=-\del_t \del_{v_i}L(\gamma(t), \dot{\gamma}(t))=\del_{x_i}L(\gamma(t), \dot{\gamma}(t))=-\frac{1}{n^2}\sum_{j: j\neq i}\frac{1}{(x_i-x_j)^3}
\end{align*}
If we define 
\begin{align*}
W_t(\bmx)\deq\min_{\gamma(1)=\bmb, \gamma(t)=\bmx}\int_t^1 L(\gamma(s), \dot{\gamma}(s))\rd s.
\end{align*}
Then by a deformation argument, we obtain the Hamilton-Jacobian equation
\begin{align*}
-\del_t W_t(\bmx)=H(\bmx, \nabla W_t(\bmx))
=\frac{1}{2}\sum_{i=1}^n \left(\del_{x_i} W_t\right)^2-\frac{1}{2n^2}\sum_{i\neq j}\frac{1}{(x_i-x_j)^2}
\end{align*}
which is the equation \eqref{e:key2} by a change of variable $S_t(\bmx)=W_t(\bmx)-\sum_{i< j}\log(x_i-x_j)$.

\section{Solve  the Correction Term}\label{s:eq}
We derive the partial differential equations of the correction terms $\cE_t^{(k)}(\bmx)$ in Section \ref{s:eq1}, and solve it using Feynman-Kac
formula in Section \ref{s:solve}.

\subsection{Equation of the Correction Term}\label{s:eq1}

In this Section, we derive the differential equation for the correction term $\cE^{(k)}_t(\bmx)$ defined in Ansatz \eqref{a:defE}. 
\begin{proposition}\label{p:eqEt}
The correction terms $\cE^{(k)}_t$ as in the ansatz \ref{a:defE} satisfy the following partial differential equations: for $1\leq k\leq n$,
\begin{align}\begin{split}\label{e:derE1}
-\del_t\cE^{(k)}_t
&=\frac{1}{2n}\Delta \cE^{(k)}_t+\sum_{i=1}^n\left(\del_{x_i}  S_t+\frac{1}{n}\sum_{j: j\neq i}\frac{1}{x_i-x_j}\right)\del_{x_i}\cE^{(k)}_t+\frac{1}{n}\sum_{i: i\neq k}\frac{ \cE^{(i)}_t-\cE^{(k)}_t}{(x_i-x_k)^2}+\cL^{(k)}_t,
\end{split}\end{align}
where $\mu=(1/n)\sum_{i=1}^n \delta_{x_i}$ and
\begin{align}\begin{split}\label{e:defLt}
\cL^{(k)}_t&=\del_{x_k}g_t(x_k;\mu,t)\cE^{(k)}_t+\frac{1}{n}\sum_{i=1}^n\del_{x_i}\frac{\delta g_t(x_k;\mu,t)}{\delta \mu}\cE^{(i)}_t\\
&+\frac{1}{2n^2}\del_{x_k}\frac{\delta \del_{x_k}g_t(x_k;\mu,t)}{\delta\mu}
+\sum_{i=1}^n\frac{1}{2n^3}\del_{x_i}\frac{\delta \del_{x_i}\frac{\delta g_t(x_k;\mu,t)}{\delta \mu}}{\delta \mu}.
\end{split}\end{align}
\end{proposition}

We recall from \eqref{e:gteq},
as a function over $\cC_s^{\mu,t,+}$, $g_s(Z;\mu,t)$  satisfies the following equation
\begin{align}\begin{split}\label{e:gteqcopy}
\del_s g_s(Z; \mu,t)
&+g_s(Z;\mu,t) \del_z g_s(Z;\mu, t)\\
&+\int \frac{g_s(x;\mu, t)-g_s(Z;\mu, t)-(x-z)\del_z g_s(Z;\mu, t)}{(x-z)^2} \rho^*_s(x;\mu,t)\rd x=0.
\end{split}\end{align}
The derivatives of $g_t(Z;\mu,t)$ with respect to the measure $\mu$ can be expressed explicitly using Proposition \ref{p:derft}. We recall the Schiffer kernel \eqref{e:Schiffer} $B(Z,Z';\mu,t)$ on $(Z,Z')=((f,z), (f',z'))\in\cC_t^{\mu,t}\times \cC_t^{\mu,t}$ of the Riemann surface $\cC_t^{\mu,t}$.  Since $g_t(Z;\mu,t)=f_t(Z;\mu,t)-\int \rd \mu(x)/(z(Z)-x)$, \eqref{e:deltagt} gives
\begin{align}\label{e:dmugt}
\del_{z'}\frac{\delta g_t(Z;\mu,t)}{\delta \mu}= \frac{B(Z,Z';\mu,t)}{\rd z(Z)\rd z(Z')}-\frac{1}{(z(Z)-z(Z'))^2},
\end{align}
which no longer have double poles along the diagonal.
We remark that the Schiffer kernel is symmetric, it gives that
\begin{align}\label{e:symmetric}
\del_{z'}\frac{\delta g_t(Z;\mu,t)}{\delta \mu}=\del_{z}\frac{\delta g_t(Z';\mu,t)}{\delta \mu}.
\end{align} 
Using Theorem \ref{thm:Rauch}, one can further take derivative with respect to $\mu$ on both sides of \eqref{e:symmetric}. Those higher order derivatives are also bounded as in Proposition \ref{p:derg}.

%
%

In the rest of this section we prove Proposition \ref{p:eqEt}. Let 
$\{(\rho^*_s(\cdot;\mu,t), u^*_s(\cdot;\mu,t))\}_{t\leq s\leq {1+\tau}}$ be the minimizer of \eqref{e:varWtcopy}.
For any $t\leq s\leq {1+\tau}$, $\{(\rho^*_r(\cdot;\mu,t), u^*_r(\cdot;\mu,t))\}_{t\leq r\leq {1+\tau}}$ restricted on time interval $[s,{1+\tau}]$ is the minimizer of 
\begin{align}\label{e:varWs}
W_s(\rho_s^*(\cdot;\mu,t))=-\frac{1}{2}\inf\left(\int_s^{1+\tau}\int_\bR \rho_r\left(u_r^2+\frac{\pi^2}{3}\rho_r^2\right)\rd x\rd r+\Sigma(\rho_s^*(\cdot;\mu,t))+\Sigma(\mu_C)\right),
\end{align}
the $\inf$ is taken over all the pairs $\{(\rho_r,u_r)\}_{s\leq r\leq {1+\tau}}$ such that $\del_r \rho_r+\del_x(\rho_ru_r)=0$ in the sense of distributions, $\{\rho_r\}_{s\leq r\leq {1+\tau}}\in{\mathcal{C}}([t,{1+\tau}],\cM(\bR))$  and its initial and terminal data are  given by 
\begin{align*}
\lim_{r\rightarrow s+}\rho_r(x)\rd x=\rho_s^*(x;\mu,t)\rd x,\quad \lim_{r\rightarrow (1+\tau)-}\rho_r(x)\rd x=\mu_C,
\end{align*}
where convergence holds in the weak sense.
Therefore, \eqref{e:gtut15} implies that
\begin{align}\label{e:gs}
g_s(x;\mu,t)=u^*_s(x;\mu,t)-\int\frac{1}{x-y}\rd \rho_s^*(y)=g_s(x; \rho_s^*, s)
\end{align}

For the time derivative in \eqref{e:gteqcopy}, using \eqref{e:gs}, we can rewrite it as
\begin{align}\begin{split}\label{e:dergs}
\left.\del_s g_s(x; \mu,t)\right|_{s=t}
&=\left.\del_s g_s(x; \rho_s^*, s)\right|_{s=t}
=\del_t g_t(x;\mu,t)+\int \frac{\delta g_t(x; \mu,t)}{\delta \mu}\left.\del_s \rho_s^*(y;\mu,t)\right|_{s=t}\rd y\\
&=\del_t g_t(x;\mu,t)-\int \frac{\delta g_t(z; \mu, t)}{\delta \mu}\del_y \left.(\rho_s^*(y;\mu,t)u_s^*(y;\mu,t))\right|_{s=t}\rd y\\
&=\del_t g_t(x;\mu,t)+\int \del_y\frac{\delta g_t(x;\mu,t)}{\delta \mu} \left(g_t(y;\mu,t)+\int\frac{\rd\mu(w)}{y-w}\right)\rd \mu(y),
\end{split}\end{align}
where in the last line we used \eqref{e:gtut15}.

Using \eqref{e:dergs}, we can take $s=t$ in \eqref{e:gteqcopy}, and $z=z(Z)$ with $Z\in \cC^{\mu,t}_t$
\begin{align}\begin{split}\label{e:ddWt}
\del_t g_t(Z;\mu,t)&+g_t(Z;\mu,t) \del_z g_t(Z;\mu, t)+\int \del_y\frac{\delta g_t(Z;\mu,t)}{\delta \mu} \left(g_t(y;\mu,t)+\int\frac{\rd\mu(w)}{y-w}\right)\rd \mu(y)\\
&+\int \frac{g_t(x;\mu, t)-g_t(Z;\mu, t)-(x-z)\del_z g_t(Z;\mu, t)}{(x-z)^2} \rd\mu(x)=0,
\end{split}\end{align}
%
After taking the measure $\mu=(1/n)\sum_{i=1}^n \delta_{x_i}$,  \eqref{e:ddWt} becomes
\begin{align}\begin{split}\label{e:limiteq}
\del_t g_t(Z;\mu,t)&+g_t(Z; \mu,t)\del_z g_t(Z; \mu,t)+\frac{1}{n}\sum_{i} \del_{z} \frac{\delta g_t(x_i;\mu,t)}{\delta \mu} g_t(x_i;\mu,t)\\
&+\frac{1}{2n^2}\sum_{i,j} \frac{\del_z\frac{\delta g_t(x_i;\mu,t)}{\delta \mu}-\del_z \frac{\delta g_t(x_j;\mu,t)}{\delta \mu}}{x_i-x_j}
+\frac{1}{n}\sum_{i}\del_z\left(\frac{g_t(Z;\mu,t)-g_t(x_i;\mu,t)}{z-x_i}\right)=0,
\end{split}\end{align}
where we used \eqref{e:symmetric}, and formally when $x=x'$, we take
\begin{align*}
\frac{\del_z \frac{\delta g_t(x;\mu,t)}{\delta \mu}-\del_z\frac{\delta g_t(x';\mu,t)}{\delta \mu}}{x-x'}
=\del_x \del_z\frac{\delta g_t(x;\mu,t)}{\delta \mu},
\end{align*}
and we $x=x'$, we take 
\begin{align*}
\frac{g_t(x;\mu,t)-g_t(x';\mu,t)}{x-x'}=\del_x g_t(x;\mu,t).
\end{align*}

We recall the partial differential equation \eqref{e:derSt} for $\del_{x_k} S_t$,
\begin{align}\label{e:derSt1}
-\del_t\del_{x_k} S_t
={\frac{1}{2n}\Delta \del_{x_k}S_t}+\sum_{i=1}^n\left(\del_{x_i}  S_t+\frac{1}{n}\sum_{j: j\neq i}\frac{1}{x_i-x_j}\right)\del_{x_i}\del_{x_k} S_t
+\frac{1}{n}\sum_{i: i\neq k}\frac{\del_{x_i} S_t-\del_{x_k} S_t}{(x_i-x_k)^2},
\end{align}
We can plug in \eqref{e:derSt1} the Ansatz \ref{a:defE}, $\del_{x_k}S_t=g_t(x_k;\mu,t)+\cE^{(k)}_t$, 
\begin{align*}\begin{split}
&-\del_t (g_t(x_k;\mu,t)+\cE^{(k)}_t)\\
&={\frac{1}{2n}\Delta (g_t(x_k;\mu,t)+\cE^{(k)}_t)}+\sum_{i=1}^n\left(\del_{x_i}  S_t+\frac{1}{n}\sum_{j: j\neq i}\frac{1}{x_i-x_j}\right)\del_{x_i}(g_t(x_k;\mu,t)+\cE^{(k)}_t)\\
&+\frac{1}{n}\sum_{i: i\neq k}\frac{ (g_t(x_i;\mu,t)+\cE^{(i)}_t)-(g_t(x_k;\mu,t)+\cE^{(k)}_t)}{(x_i-x_k)^2},
\end{split}\end{align*}
which can be written as
\begin{align}\begin{split}\label{e:derE2}
-\del_t\cE^{(k)}_t
&=\frac{1}{2n}\Delta \cE^{(k)}_t+\sum_{i=1}^n\left(\del_{x_i}  S_t+\frac{1}{n}\sum_{j: j\neq i}\frac{1}{x_i-x_j}\right)\del_{x_i}\cE^{(k)}_t+\frac{1}{n}\sum_{i: i\neq j}\frac{ \cE^{(i)}_t-\cE^{(k)}_t}{(x_i-x_k)^2}+\cL^{(k)}_t,
\end{split}\end{align}
where 
\begin{align}\begin{split}\label{e:defL}
\cL^{(k)}_t=\del_t g_t(x_k;\mu,t)
&+\frac{1}{2n}\Delta ( g_t(x_k;\mu,t))+\left(g_t(x_k;\mu,t)+\cE^{(k)}_t+\frac{1}{n}\sum_{j: j\neq k}\frac{1}{x_k-x_j}\right)\del_{x_k}g_t(x_k;\mu,t)\\
&+\sum_{i=1}^n\left(g_t(x_i;\mu,t)+\cE^{(i)}_t+\frac{1}{n}\sum_{j: j\neq i}\frac{1}{x_i-x_j}\right)\frac{1}{n}\del_{x_i}\frac{\delta g_t(x_k;\mu,t)}{\delta \mu}\\
&+\frac{1}{n}\sum_{i: i\neq k}\frac{ g_t(x_i;\mu,t)-g_t(x_k;\mu,t)}{(x_i-x_k)^2}.
\end{split}\end{align}
The equation \eqref{e:derE2} gives \eqref{e:derE1}. In the following we will use \eqref{e:Deltag} to simplify the expression \eqref{e:defL} of $\cL_t^{(k)}$.
We recall that $\mu=(1/n)\sum_{i=1}^n \delta_{x_i}$, the Laplacian term in \eqref{e:defL} can be written as
\begin{align}\begin{split}\label{e:Deltag}
&\phantom{{}={}}\frac{1}{2n}\Delta ( g_t(x_k;\mu,t))
=\frac{1}{2n}\del_{x_k}(\del_{x_k}g_t(x_k;\mu,t))
+\sum_{i=1}^n\frac{1}{2n^2}\del_{x_i}\left(\del_{x_i}\frac{\delta g_t(x_k;\mu,t)}{\delta \mu}\right)\\
&=\frac{1}{2n}\del_{x_k}^2g_t(x_k;\mu,t)+\frac{1}{2n^2}\del_{x_k}\frac{\delta \del_{x_k}g_t(x_k;\mu,t)}{\delta\mu}
+\sum_{i=1}^n\frac{1}{2n^2}\del_{x_i}^2\frac{\delta g_t(x_k;\mu,t)}{\delta \mu}
+\sum_{i=1}^n\frac{1}{2n^3}\del_{x_i}\frac{\delta \del_{x_i}\frac{\delta g_t(x_k;\mu,t)}{\delta \mu}}{\delta \mu}
\end{split}\end{align}
It follows from plugging \eqref{e:Deltag} into \eqref{e:defL}, and taking difference with \eqref{e:limiteq} for $z(Z)=x_k$, we obtain  
\begin{align*}\begin{split}
\cL^{(k)}_t&=\del_{x_k}g_t(x_k;\mu,t)\cE^{(k)}_t+\frac{1}{n}\sum_{i=1}^n\del_{x_i}\frac{\delta g_t(x_k;\mu,t)}{\delta \mu}\cE^{(i)}_t\\
&+\frac{1}{2n^2}\del_{x_k}\frac{\delta \del_{x_k}g_t(x_k;\mu,t)}{\delta\mu}
+\sum_{i=1}^n\frac{1}{2n^3}\del_{x_i}\frac{\delta \del_{x_i}\frac{\delta g_t(x_k;\mu,t)}{\delta \mu}}{\delta \mu}.
\end{split}\end{align*}
This finishes Proposition \ref{p:eqEt}.

\subsection{Solve the Correction Term}\label{s:solve}
The partial differential equations \eqref{e:derE1} of the correction terms $\cE_t^{(k)}$ can be solved by using Feynman-Kac formula. We recall the random walk corresponding to the weighted nonintersecting Brownian bridges from \eqref{e:randomwalk2}
\begin{align}\label{e:randomwalk3}
\rd x_i(t)=\frac{1}{\sqrt n} \rd B_i(t)+\frac{1}{n}\sum_{j: j\neq i}\frac{\rd t}{x_i(t)-x_j(t)}+(g_t(x_i(t);\mu,t)+\cE_t^{(i)}(x_1(t),x_2(t), \cdots, x_n(t)))\rd t,
\end{align}
for $1\leq i\leq n$. 
We can use \eqref{e:randomwalk3} to write down the stochastic differential equation of 
$\cE_t^{(k)}(\bmx(t))=\cE_t^{(k)}(x_1(t), x_2(t),\cdots, x_n(t))$,
\begin{align}\begin{split}\label{e:keyeq}
&\rd \cE_t^{(k)}(\bmx(t))=\rd \cM_t^{(k)}
-\frac{1}{n}\sum_{i:i\neq k}\frac{\cE_t^{(i)}(\bmx(t))-\cE_t^{(k)}(\bmx(t))}{(x_i(t)-x_k(t))^2}-\del_{x_k}g_t(x_k(t);\mu_t,t)\cE^{(k)}_t(\bmx(t))\\
&-\frac{1}{n}\sum_{i=1}^n\del_{x_i}\frac{\delta g_t(x_k(t);\mu_t,t)}{\delta \mu}\cE^{(i)}_t
-\frac{1}{2n^2}\del_{x_k}\frac{\delta \del_{x_k}g_t(x_k(t);\mu_t,t)}{\delta\mu}
-\sum_{i=1}^n\frac{1}{2n^3}\del_{x_i}\frac{\delta \del_{x_i}\frac{\delta g_t(x_k(t);\mu_t,t)}{\delta \mu}}{\delta \mu}.
\end{split}\end{align}
where $\mu_t=(1/n)\sum_{i=1}^n \delta_{x_i(t)}$, and $\rd \cM_t^{(k)}$ is the martingale
\begin{align*}
\rd \cM_t^{(k)}=\frac{1}{\sqrt n}\sum_{i=1}^n \del_{x_i}\cE_t^{(k)}(\bmx(t))\rd B_i(t).
\end{align*}

In this section we solve for the correction terms $\cE_t^{(k)}$, and show they are of order $\OO(1/n^2)$.
\begin{proposition}\label{p:cEbound}
Under Assumptions \ref{a:reg} and \ref{a:ncritic}, the correction terms $\cE_t^{(k)}(\bmx)$ as defined in Ansatz \ref{a:defE} satisfies
\begin{align*}
\left|\cE_t^{(k)}(\bmx)\right|\lesssim \frac{1}{n^2}.
\end{align*}
\end{proposition}

To solve \eqref{e:keyeq}, we define the following operator, for $\bmv=(v_1,v_2,\cdots,v_n)^\top$
\begin{align}\label{e:evolve}
(\mathscr B(t)\bmv)_k=-\sum_{i:i\neq k}\frac{1}{n}\frac{v_i-v_k}{(x_i(t)-x_k(t))^2}
-\del_{x_k}g_t(x_k(t);\mu_t,t)v_k-\frac{1}{n}\sum_{i=1}^n\del_{x_i}\frac{\delta g_t(x_k(t);\mu_t,t)}{\delta \mu}v_i,
\end{align}
and the generator $\cU(s,t)$ for $s\leq t$, with $\cU(s,s)=I$ and 
\begin{align*}
\del_t \cU(s,t)=-\cU(s,t)\mathscr B(t).
\end{align*}
Then we can rewrite \eqref{e:keyeq} as
\begin{align}\label{e:FK}
\rd (\cU(s,t)(\cE_t(\bmx(t)))=\cU(s,t)\rd \cM_t-\cU(s,t)\cR_t(\bmx(t))\rd t,
\end{align}
where $\cE_t= (\cE_t^{(1)},\cE^{(2)}_t,\cdots, \cE^{(n)}_t)^\top$, $\rd \cM_t=(\rd \cM_t^{(1)}, \rd \cM_t^{(2)},\cdots, \rd \cM_t^{(n)})^\top$, and $\cR_t=(\cR_t^{(1)}, \cR_t^{(2)},\cdots, \cR_t^{n})^\top$ is the vector
\begin{align*}
\cR_t^{(k)}(\bmx(t))
=-\frac{1}{2n^2}\del_{x_k}\frac{\delta \del_{x_k}g_t(x_k(t);\mu_t,t)}{\delta\mu}
-\sum_{i=1}^n\frac{1}{2n^3}\del_{x_i}\frac{\delta \del_{x_i}\frac{\delta g_t(x_k(t);\mu_t,t)}{\delta \mu}}{\delta \mu}.
\end{align*}


By  taking expectation on both sides of \eqref{e:FK} and integrate it from time $t=s$ to $t=1$, and condition on that $\bmx(s)=\bmx$ we obtain
\begin{align}\label{e:solveE_s}
\cE_s(\bmx)
=\bE\left[\left.\int_{s}^1 \cU(s,t)\cR_t(\bmx(t))\rd t\right| \bmx(s)=\bmx\right].
\end{align}
Since $(\cU(s,t)\cR_t)_i=\langle \cU(s,t)^\top \bme_i, \cR_t\rangle$. We can use duality to understand the operator $\cU(s,t)$, and compute $\cU(s,t)^\top \bme_i$.
In other words, we need to analyze the evolution:
\begin{align}\label{e:evolve2}
\bmv(t)=\cU(s,t)^\top \bme_i, \quad \del_t \bmv(t)=-\mathscr B(t) \bmv(t),
\end{align}
where $ \bmv(t)=(v_1(t), v_2(t),\cdots, v_n(t))$ and $\bmv(s)=\bme_i$. Then we can write the integrand in \eqref{e:solveE_s} as
\begin{align}\label{e:rewriteUR}
(\cU(s,t)\cR_t(\bmx(t)))_i=
\langle \cU(s,t)^\top \bme_i, \cR_t(\bmx(t))\rangle
=\langle \bmv(t), \cR_t(\bmx(t))\rangle
=\sum_{i=1}^n  v_i(t)\cR^{(i)}_t(\bmx(t)).
\end{align}
In the following we study the evolution \eqref{e:evolve2} and obtain an upper bound for the $L_1$ norm of $\bmv(t)$.
\begin{proposition}\label{c:L1bound}
Fix time $0\leq s\leq 1$ and assume the assumptions in Proposition \ref{p:cEbound}. $\bmv(t)$ as defined in \eqref{e:evolve2} satisfies the following $L_1$ norm bound: for $s\leq t\leq 1$,
\begin{align*}
\|\bmv(t)\|_{1}=\sum_{i=1}^n |v_i(t)|\leq  e^{\fC(t-s)},
\end{align*}
for some universal constant $\fC>0$.
\end{proposition}

The estimate Proposition \ref{p:cEbound} of the correction terms $\cE_t^{(k)}$ follow easily from Claim \ref{c:L1bound}.
\begin{proof}[Proof of Proposition \ref{p:cEbound}]
We recall from \eqref{e:solveE_s}
\begin{align}\label{e:cEsbb}
\cE^{(i)}_s(\bmx)
=\bE\left[\left.\int_{s}^1 (\cU(s,t)\cR_t(\bmx(t)))_{i}\rd t\right| \bmx(s)=\bmx\right].
\end{align}
As in \eqref{e:rewriteUR}, $(\cU(s,t)\cR_t(\bmx(t)))_{i}$ can be rewritten in terms of $\bmv(t)$ as defined in \eqref{e:evolve2}. Proposition \ref{p:derg} gives that  $|\cR_t^{(i)}|\lesssim 1/n^2$ for any $1\leq i\leq n$.
It follows from combining with Claim \ref{c:L1bound}
\begin{align}\label{e:URbb}
\left|(\cU(s,t)\cR_t(\bmx(t)))_{i}\right|
=\left|\sum_{i=1}^n  v_i(t)\cR^{(i)}_t(\bmx(t))\right|
\leq \|\bmv(t)\|_1\max_{1\leq i\leq n}\left|\cR^{(i)}_t(\bmx(t))\right|
\lesssim \frac{e^{\fC(t-s)}}{n^2}.
\end{align}
By plugging \eqref{e:URbb} into \eqref{e:cEsbb}, 
\begin{align*}
\left|\cE^{(i)}_s(\bmx)\right|
\lesssim \frac{1}{n^2}.
\end{align*}
This finishes the proof of Proposition \ref{p:cEbound}.
\end{proof}

\begin{proof}[Proof of Proposition \ref{c:L1bound}]
For any time $s\leq t\leq 1$, we denote the index sets:
\begin{align*}
I_t^+=\{1\leq i\leq n: v_i(t)\geq 0\},\quad I_t^-=\{1\leq i\leq n: v_i(t)< 0\}.
\end{align*}
Then the $L_1$ norm of $\bmv(t)$ is simply $\|\bmv(t)\|_{1}=\sum_{k\in I_t^+}v_k(t)-\sum_{k\in I_t^-}v_k(t)$. We can write down the evolution of $\sum_{k\in I_t^+}v_k(t)$ and $\sum_{k\in I_t^-}v_k(t)$ separately. Using the definition \eqref{e:evolve} of the operator $\mathscr B(t)$
\begin{align*}
\del_t\sum_{k\in I_t^+}v_i(t)
=\sum_{k\in I_t^+, i\in I_t^-}\frac{1}{n}\frac{v_i-v_k}{(x_i(t)-x_k(t))^2}
+\sum_{k\in I_t^+}\left(\del_{x_k}g_t(x_k(t);\mu_t,t)v_k+\frac{1}{n}\sum_{i=1}^n\del_{x_i}\frac{\delta g_t(x_k(t);\mu_t,t)}{\delta \mu}v_i\right).
\end{align*}
We notice that the first term is non-positive, and from Proposition \ref{p:derg}, the derivatives of $g_t$ are bounded by $\OO(1)$,
\begin{align}\label{e:upbb}
\del_t\sum_{k\in I_t^+}v_i(t)\leq \fC\|\bmv(t)\|_1.
\end{align}
We have a similar upper bound for the sum $-\sum_{k\in I_t^-}v_k(t)$,
\begin{align}\label{e:lobb}
\del_t\left(-\sum_{k\in I_t^-}v_i(t)\right)\leq \fC\|\bmv(t)\|_1.
\end{align}
The two estimates \eqref{e:upbb} and \eqref{e:lobb} together imply
\begin{align}\label{e:L1bound}
\del_t \|\bmv(t)\|_1
=\del_t\left(\sum_{k\in I_t^+}v_i(t)-\sum_{k\in I_t^-}v_i(t)\right)
\leq \fC\|\bmv(t)\|_1.
\end{align}
At time $t=s$ we have $\|\bmv(s)\|_1= \|\bme_i\|_1=1$. Therefore, \eqref{e:L1bound} implies an upper bound for $\|\bmv(t)\|_1$
\begin{align*}
\|\bmv(t)\|_1\leq e^{\fC(t-s)}.
\end{align*}
\end{proof}

\section{Optimal Particle Rigidity}\label{s:rigidity}
In this section, we show that the particle locations of the weighted nonintersecting Brownian bridge \eqref{e:wdensity} are close to their classical locations.

We denote the minimizer of the variational problem \eqref{e:varmuC} 
as $\{(\rho_t^*(\cdot), u^*_t(\cdot))\}_{0\leq t\leq 1+\tau}$. From the discussion after Assumption \ref{a:reg}, we know that after restricting on the time interval $[0,1]$, $\{(\rho_t^*(\cdot), u^*_t(\cdot))\}_{0\leq t\leq 1}$ is the minimizer of \eqref{e:Iexp}. We recall $f_t$ from \eqref{e:ft}. In the notation of Section \ref{s:bc}, we have $f_t(x)=f_t(x; \mu_A, 0)$.
There is a Riemann surface $\cC_t$ associated with the variational problem, which can be identified with two copies of the region $\Omega$ gluing along its left and right boundaries
\begin{align}\label{e:Qt}
\pi_{Q^t}: (x,s)\in\Omega=\{(x,t)\in \bR\times (0,1+\tau):\rho_t^*(x)>0\}\mapsto (f_t(x),x-(s-t)f_t(x))\in \cC_t.
\end{align}
$f_t$ can be extended to a meromorphic function over $\cC_t$ and it satisfies the complex Burger's equation from \eqref{e:complexburgereq2}
\begin{align}\label{e:dft}
\del_tf_t(Z)+\del_z f_t(Z) f_t(Z)=0,\quad Z=(f,z)\in \cC_t.
\end{align}
The complex Burger's equation can be solved 
using the characteristic flow \eqref{e:flow}: for any $Z=(f,z)\in \cC_0$, let
\begin{align}\label{e:transport2}
Z_t(Z)=(f, z+t f)\in \cC_t, \quad z_t(Z)=z+tf.
\end{align}
Then $f_t(Z_t(Z))=f$. Similar to \eqref{e:zmatr}, the cycles $\{(f_t(x),x): (x,t)\in\Omega\}\cup \{\overline{(f_t(x),x)}: (x,t)\in\Omega\}$ cut $\cC_t$ into two parts $\cC_t^{+}=\cC^{\mu_A,0,+}_{t}$ and $\cC_t^{-}=\cC^{\mu_A,0,-}_{t}$
If we denote 
\begin{align}\label{e:gtZdef}
g_t(Z)=g_t(Z;\mu_A,0)=f_t(Z)-m_t(z),\quad z=z(Z),\quad Z\in \cC_t^+,
\end{align}
then it can be glued to an analytic function in a neighborhood the cut.
Then \eqref{e:dermt} gives
\begin{align}\begin{split}\label{e:dermtse}
\del_t m_{t}(z)=-\del_z m_{t}(z)m_{t}(z)+\int\frac{g_t(x)}{(z-x)^2}\rho^*_t(x)\rd x,
\end{split}\end{align} 
and \eqref{e:gteq} gives
\begin{align}\label{e:gtsec8}
\del_t g_t(Z)+g_t(Z) \del_z g_t(Z)+\int \frac{g_t(x)-g_t(Z)-(x-z)\del_z g_t(Z)}{(z-x)^2} \rho_t^*(x)\rd x=0.
\end{align}

We denote the Stieltjes transform of the empirical particle density of the weighted nonintersecting Brownian bridge starting from $\mu_{A_n}$ and the minimizer of the variational problem \eqref{e:varmuC}  as
\begin{align*}
\tilde m_t(z)=\frac{1}{n}\sum_{i=1}^n\frac{1}{z-x_i(t)},\quad m_t(z)=\int \frac{\rho^*_t(x)}{z-x}\rd x,\quad 0\leq t\leq 1.
\end{align*}
With these notations, we can restate Assumption \ref{a:A_n} as $|\tilde m_0(z)-m_0(z)|\lesssim (\log n)^\fa/n$.

\begin{remark}
Assumption \ref{a:A_n} is necessary for the optimal rigidity estimates. We will see from our proof, the error for the Stieltjes transform on the region away from the support of $\mu_A$ propagates. In other words, if $|\tilde m_0(z)-m_0(z)|$ is large, we do not expect that $|\tilde m_t(z)-m_t(z)|$ will be small for $0\leq t\leq 1$.
\end{remark}

In this section we prove the following optimal particle rigidity estimates.

\begin{theorem}\label{t:rigidity}
Under Assumptions \ref{a:reg}, \ref{a:ncritic} and \ref{a:A_n}, for any small time $\ft>0$, the following holds for the particle locations of the weighted nonintersecting Brownian bridges \eqref{e:wdensity} starting from $\mu_{A_n}$.
With high probability, for any time $\ft\leq t\leq 1$ we have
\begin{enumerate}
\item We denote $\gamma_i(t)$ the $1/n$-quantiles of the density $\rho^*_t(\cdot)$, i.e. 
\begin{align}\label{e:gamma0}
\frac{i+1/2}{n}=\int_{-\infty}^{\gamma_i(t)}\rho_t^*(x)\rd x,\quad 1\leq i\leq n, 
\end{align}
then uniformly for $1\leq i\leq n$, the locations of $x_i(t)$ are close to their corresponding quantiles
\begin{align}\label{e:bulkrigid}
\gamma_{i-(\log n)^{\OO(1)}}(t)\leq x_i(t)\leq \gamma_{i+(\log n)^{\OO(1)}}(t).
\end{align}
\item If $\rho_t^*$ is supported on $[\sfa(t), \sfb(t)]$, then the particles close to the left and right boundary points of $\supp(\rho_t^*)$ satisfies
\begin{align}\label{e:edgerigid2}
   x_1(t)\geq {\sf a}(t)-\frac{(\log n)^{\OO(1)}}{n^{2/3}},\quad x_{n}(t)\leq {\sf b}(t)+ \frac{(\log n)^{\OO(1)}}{n^{2/3}}.
\end{align}
\end{enumerate}
\end{theorem}
\begin{remark}
The support of $\rho_t^*$ may consist of several intervals $[{\sf a}_1(t), {\sf b}_1(t)]\cup [{\sf a}_2(t), {\sf b}_2(t)]\cup \cdots \cup [{\sf a}_r(t), {\sf b}_r(t)]$. The argument of \eqref{e:edgerigid2} can be extended to give optimal rigidity estimates for particles  close to those inner endpoints, i.e. when $\sfa_{k+1}(t)- \sfb_k(t)\gtrsim 1$, we expect that 
\begin{align*}
   x_i(t)\leq {\sf b_{k}}(t)+\frac{(\log n)^{\OO(1)}}{n^{2/3}},\quad x_{i+1}(t)\geq {\sf a_{k+1}}(t)- \frac{(\log n)^{\OO(1)}}{n^{2/3}}, \quad i=n\int_{-\infty}^{\sfb_k(t)}\rho_t^*(x)\rd x.
\end{align*}
\end{remark}

Theorem \ref{t:rigidity} follows from estimates of the Stieltjes transform of the empirical particle density. It follows from \cite[Corollary 3.2]{MR4009708} that \eqref{e:bulkrigid} is equivalent to the following estimates of the Stieltjes transform of the empirical particle density.
\begin{proposition} \label{p:bulkrigid2}
Under the assumptions of Theorem \ref{t:rigidity}, with high probability, the Stieltjes transform of the weighted nonintersecting Brownian bridges \eqref{e:wdensity}
 satisfies the following estimates: for any $\ft\leq t\leq 1$, with high probability we have
\begin{align}\label{e:bulkrigid2}
    \left|\tilde m_t(z)-m_t(z)\right|\lesssim \frac{(\log n)^{ \fa+1}}{n\Im[z]}.
\end{align}
uniformly on   $\{z\in\bC: \dist(z, [\sfa(t), \sfb(t)])|\Im[m_t(z)]|\wedge|\Im[z]|\gtrsim (\log n)^{ \fa+2}/n)\}$.
\end{proposition}

The proof of Proposition \ref{p:bulkrigid2} follows similar arguments as in \cite[Theorem 3.1]{MR4009708}, and the proof of \eqref{e:edgerigid2} follows from similar argument as in \cite[Theorem 4.3]{adhikari2020dyson} with two  modifications.
Firstly, the drift term in the random walk \eqref{e:weightrandomwalk}  corresponding to the weighted nonintersecting Brownian bridges \eqref{e:wdensity} depends on time, however the drift terms in \cite{MR4009708, adhikari2020dyson} do not depend on time. Secondly, the rigidity estimates in \cite{MR4009708, adhikari2020dyson} are proven for short time $t=\oo(1)$. Here we need to show the rigidity estimates up to time $t=1$, which requires  more careful analysis.

Using the Ansatz \ref{a:defE}, we can rewrite the  random walk \eqref{e:randomwalk2} corresponding to the weighted nonintersecting Brownian bridges as
\begin{align}\label{e:randomwalk4}
\rd x_i(t)=\frac{1}{\sqrt n} \rd B_i(t)+\frac{1}{n}\sum_{j: j\neq i}\frac{\rd t}{x_i(t)-x_j(t)}+(g_t(x_i(t);\mu_t,t)+\cE_t^{(i)}(x_1(t),x_2(t), \cdots, x_n(t)))\rd t,
\end{align}
for $1\leq i\leq n$, where 
\begin{align*}
\mu_t=\frac{1}{n}\sum_{i=1}^n \delta_{x_i(t)},\quad \mu_0=\mu_{A_n}.
\end{align*}
We denote the difference between the Stieltjes transforms $\tilde m_t(z)$ and $m_t(z)$ as
\begin{align}\label{e:Delta_t}
\Delta_t(z)\deq\tilde m_t(z)- m_t(z)=\int \frac{1}{z-x}(\rd\mu_t-\rho_t^*(x)\rd x).
\end{align}
To write down the stochastic differential equation of $\Delta_t(z)$, we introduce the following quantity, 
\begin{align*}
\tilde f_t(Z)=\frac{1}{n}\sum_{i}\frac{1}{z-x_i(t)}+g_t(Z)=
\int \tilde m_t(z)+g_t(Z),\quad Z=(f,z)\in \cC_t.
\end{align*}
Using It\^{o} 's formula, we can write down the stochastic differential equation for $\tilde f_t(Z)$:
\begin{align}\begin{split}\label{e:df}
\rd \tilde f_t(Z)
&=\frac{1}{n\sqrt n}\sum_{i=1}^n\frac{\rd B_i(t)}{(z-x_i(t))^2}+\del_t g_t(Z)\rd t+\frac{1}{n^2}\sum_{i=1}^n \frac{1}{(z-x_i(t))^2}
\sum_{j=1}^n \frac{1}{z-x_j(t)}\rd t\\
&+\frac{1}{n}\sum_{i=1}^n\frac{g_t(x_i(t);\mu_t,t)+\cE^{(i)}_t(\bmx(t))}{(z-x_i(t))^2}\rd t.
\end{split}\end{align}
Then we can use \eqref{e:gtsec8},
to rewrite the righthand side of \eqref{e:df} as
\begin{align}\begin{split}\label{e:newdf}
\rd \tilde f_t(Z)&=\frac{1}{n\sqrt n} \sum_{i=1}^n\frac{\rd B_i(t)}{(z-x_i(t))^2}
-\tilde f_t(Z)\del_z \tilde f_t(Z)\rd t+
\int\frac{g_t(x)-g_t(Z)-(x-z)\del_z g_t(Z)}{(z-x)^2} (\rd \mu_t-\rho_t^*(x)\rd x)\rd t\\
&+\frac{1}{n}\sum_{i=1}^n\frac{(g_t(x_i(t);\mu_t,t)-g_t(x_i(t)))+\cE^{(i)}_t(\bmx(t))}{(z-x_i(t))^2}\rd t.
\end{split}\end{align}
%
%
%
To solve for \eqref{e:newdf}, we recall the characteristic flow \eqref{e:transport2}: for any $Z_0=(f,z)\in \cC_0$, then
$Z_t=(f, z+t f)\in \cC_t$ and $f_t(Z_t)=f$.
We will use the notations $z_t=z(Z_t)=z+tf$, and notice that
\begin{align}\label{e:dzt}
\del_t z_t=f=f_t(Z_t)=m_t(z_t)+g_t(Z_t).
\end{align}
We plug $Z=Z_t$ in \eqref{e:newdf}, and use \eqref{e:dzt} to get
\begin{align}\begin{split}\label{e:dftzt}
\rd \tilde f_t(Z_t)&=\frac{1}{n\sqrt n} \sum_{i=1}^n\frac{\rd B_i(t)}{(z_t-x_i(t))^2}
-\Delta_t(z_t)\del_z \tilde f_t(Z_t)\rd t+
\int\frac{g_t(x)-g_t(Z_t)-(x-z_t)\del_z g_t(Z_t)}{(z_t-x)^2} (\rd \mu_t-\rho_t^*(x)\rd x)\rd t\\
&+\frac{1}{n}\sum_{i=1}^n\frac{(g_t(x_i(t);\mu_t,t)-g_t(x_i(t)))+\cE^{(i)}_t(\bmx(t))}{(z_t-x_i(t))^2}\rd t.
\end{split}\end{align}
From our definition of the characteristic flow \eqref{e:transport2}, $f_t(Z_t)=f$, which is a constant. Thus we have $\del_t f_t(Z_t)=0$.
In this way, we obtain a stochastic differential equation of $\Delta_t$ by taking difference of \eqref{e:dftzt} with $\del_t f_t(Z_t)=0$.
\begin{align}\begin{split}\label{e:dDeltat}
\rd\Delta_t(z_t)&=\frac{1}{n\sqrt n} \sum_{i=1}^n\frac{\rd B_i(t)}{(z_t-x_i(t))^2}
-\Delta_t(z_t)\del_z \tilde f_t(Z_t)\rd t+
\int\frac{g_t(x)-g_t(Z_t)-(x-z_t)\del_z g_t(Z_t)}{(z_t-x)^2} (\rd \mu_t-\rho_t^*(x)\rd x)\rd t\\
&+\frac{1}{n}\sum_{i=1}^n\frac{(g_t(x_i(t);\mu_t,t)-g_t(x_i(t)))+\cE^{(i)}_t(\bmx(t))}{(z_t-x_i(t))^2}\rd t.
\end{split}\end{align}
%
%
%
%
 We can integrate both sides of \eqref{e:dDeltat} from $0$ to $t$ and obtain
\beq \label{eq:mzt}
\Delta_t(z_t)=\int_0^t  \rd \cE_1(s) +(-\Delta_s(z_s)\del_z \tilde f_s(Z_s)+\cE_2(s)+\cE_3(s))\rd s,
\eeq
where the error terms are
\begin{align}
\label{defcE1}& \rd \cE_1(s)=\frac{1}{n\sqrt n} \sum_{i=1}^n\frac{\rd B_i(s)}{(z_s-x_i(s))^2},\\
\label{defcE2}& \cE_2(s)=\int\frac{g_t(x)-g_t(Z_t)-(x-z_t)\del_z g_t(Z_t)}{(z_t-x)^2} (\rd \mu_t-\rho_t^*(x)\rd x),\\
\label{defcE3}& \cE_3(s)=\frac{1}{n}\sum_{i=1}^n\frac{(g_t(x_i(t);\mu_t,t)-g_t(x_i(t)))+\cE^{(i)}_t(\bmx(t))}{(z_t-x_i(t))^2}.
\end{align}
We remark that $\cal E_1,\cE_2$ and $\cal E_3$ implicitly depend on $Z_0=(f,z)$, the initial value of the flow $Z_t$.  The optimal rigidity estimates will eventually follow from an application of the Gr{\" o}nwall's inequality to \eqref{eq:mzt}.

We take a small time $0\leq \ft_0\leq 1$, after time $\ft_0$, the measure $\rho_t^*$ has square root behavior as in \eqref{e:square}. It is straightforward to show the following asymptotics of the Stieltjes transform of $\rho_t^*$.
\begin{proposition}\label{p:squareroot}
Under the assumptions of Theorem \eqref{t:rigidity}, for any $\ft_0\leq t\leq 1$ the measure $\rho_t^*$ has square root behavior in the following sense: in a small neighborhood of the left edge $\sfa(t)$ of $\rho_t^*$, its Stieltjes transform $m_t(z)$ satisfies
\begin{align}\label{e:ssmt}
\Im[m_t(\sfa(t)-\kappa+\ri\eta)]\asymp
\left\{
\begin{array}{cc}
\eta/\sqrt{\kappa+\eta}, & \kappa\geq 0,\\
\sqrt{|\kappa|+\eta} & \kappa \leq 0.
\end{array}
\right.
\end{align}
The same statement holds right edge of $\rho_t^*$.
\end{proposition}
We decompose the proof of \eqref{e:edgerigid2} and Proposition \ref{p:bulkrigid2} into two parts. In the first part, Section \ref{s:shorttime}, we show that for time $0\leq t\leq\ft_0$, the Stieltjes transform of the empirical particle density $\mu_t$ satisfies the same estimate \eqref{e:A_n} as $\mu_A$. In the second part, Section \ref{s:longtime}, we study the case $\ft_0\leq t\leq 1$. For which, we need certain estimates of the characteristic flow, which comes from the square root behavior of the density $\rho_t^*$ as in Proposition \ref{p:squareroot}, and this holds only for time $\ft_0\leq t\leq 1$.

We recall the domain
$\Omega=\{(x,t)\in \bR\times (0,1+\tau):\rho_t^*(x)>0\}$, and use the map $\pi_{Q^{t}}$ from \eqref{e:Qt}, $\cC_0$ can be identified with two copies of $\Omega$ gluing together, and $\cC_t^+$ can be identified with two copies of $\Omega\cap \bR\times[t, 1+\tau]$ gluing together.
For any $0\leq r\leq 1+\tau$, we define the spectral domain  
\begin{align*}
\cD_0(r)\deq 
\{(f,z)\in \cC_0: Z_s(f,z)\in \cC_s^+, 0\leq s\leq r \}\subset \cC_0.
\end{align*}
They correspond to points in $\{(x,s)\in \Omega: r\leq s\leq 1+\tau\}$.  
More generally for any $0\leq t\leq 1$ and $0\leq r\leq 1+\tau-t$, we define
\begin{align}\begin{split}\label{e:defDtr}
\cD_t(r)&=\{Z_t(f,z)\in \cC_t^+: Z_{t+s}(f,z)\in \cC_{t+s}^+, 0\leq s\leq r\}=Z_t(\cD_0(r+t))\subset \cC_t.
\end{split}\end{align}
They correspond to points in $\{(x,s)\in \Omega:  t+r\leq 1+\tau\}$, as in figure \ref{f:projection}.
\begin{figure}
\begin{center}
 \includegraphics[scale=0.22,trim={0cm 5cm 0 7cm},clip]{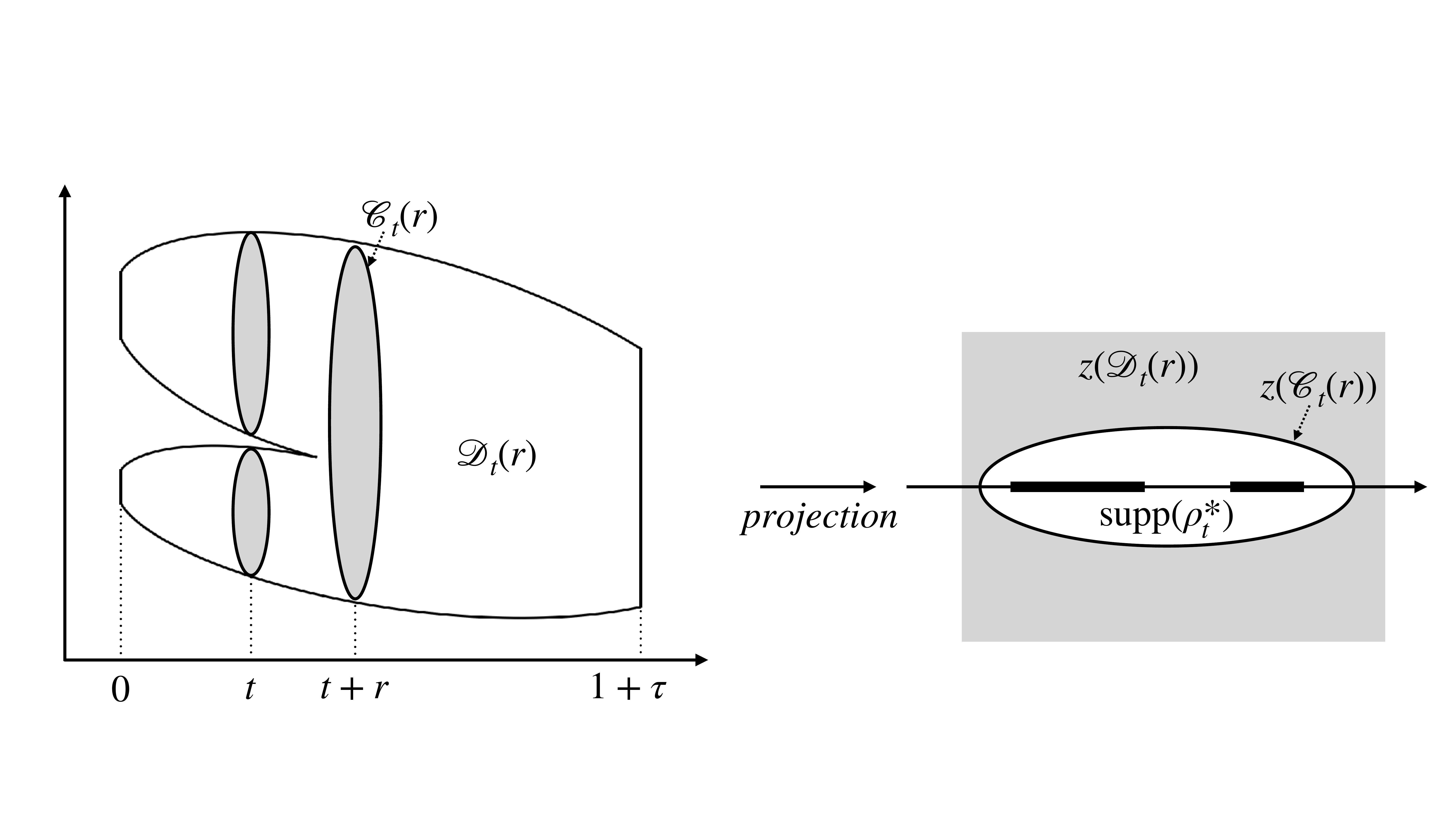}
 \caption{An illustration of the domain $\cD_t(r)$, the contour $\mathscr C_t(r)$ , and their projections on the $z$-plane.}
 \label{f:projection}
 \end{center}
 \end{figure}

The bottom boundary of $\cD_t(r)$ inside $\cC_t^+$ is a curve, 
we denote it by $\mathscr C_t(r)$, and its projection on the $z$-plane as 
$z(\mathscr C_t(r))$. Then $z(\mathscr C_t(r))$
encloses $\supp(\rho_t^*)$. Explicitly it is given 
\begin{align}\label{e:Ctr}
z(\mathscr C_t(r))=\{z=x-rf_{t+r}(x): (x,t+r)\in \Omega\}.
\end{align}

Thanks to Proposition \ref{p:branchloc}, 
a neighborhood of $ \supp(\rho_t^*)$ is a bijection with its preimage under the projection map $(f_t(Z),z)\in \cC_t^+\mapsto z\in \bC$, i.e. for any $w$ in a neighborhood of $ \supp(\rho_t^*)$, the cardinality of $z^{-1}(w)$ is one.
We can take  $\fr>0$ small enough, such that for any $0\leq t\leq 1$, $z(\cC_t^+\setminus \cD_t(\fr))$ is inside such neighborhood of $ \supp(\rho_t^*)$. Then the projections $z(\cD_t(\fr)), z(\cC_t^+\setminus \cD_t(\fr))$ do not intersect with each other, $z(\cD_t(\fr))\cap z(\cC_t^+\setminus \cD_t(\fr))=\emptyset$.


We define  the following lattice on the domain $\cD_0(\fr)$,
\beq\label{def:L}
\cal L\deq \left\{z^{-1}(E+\ri \eta)\in \dom_0(\fr): E\in \bZ/ n^2, \eta\in \bZ/n^2\right\}.
\eeq
It is easy to see from the characteristic flow \eqref{e:transport2}, the image of the lattice $\cL$ also gives a $1/n^2$ mesh of $Z_t(\cD_0(\fr))$ after projected to the $z$-plane.
\begin{proposition}\label{prop:continuity}
For any $0\leq t\leq 1$ and $W\in Z_t(\cD_0(\fr))$, there exists some lattice point $Z\in \cal L$ as in \eqref{def:L}, such that
\begin{align*}
|z_t(Z)-z(W)|\lesssim 1/n^2.
\end{align*}
where the implicit constant depends on $\mu_A$.
\end{proposition}


By symmetry, in the rest of this section, we restrict ourselves in the subdomain that $\{Z\in \cC_0: \Im[z(Z)]\geq 0\}$. In this case we have $\Im[f_0(Z)]\leq 0$, and for any $t\geq 0$ such that $Z_t(Z)\in \cC_t^+$, $z_t(Z)$ remains nonnegative, i.e. $\Im[z_t(Z)]\geq 0$. The estimates for $\{Z\in \cC_0: \Im[z(Z)]\leq 0\}$ follows by symmetry.

\subsection{Short time estimates}
\label{s:shorttime}
\begin{proposition} \label{p:extremeeig}
Under the assumptions of Theorem \ref{t:rigidity}, with high probability, the Stieltjes transform of the weighted nonintersecting Brownian bridge \eqref{e:wdensity}
satisfies the following estimates: with high probability, it holds uniformly for any $0\leq t\leq \ft_0$
\begin{align}\label{e:extreme}
\{x_1(t), x_2(t),\cdots, x_n(t)\}\in\cup_k [{\sf a}_k- \fC\sqrt{t}-\oo(1), {\sf b}_k+\fC\sqrt{t}+\oo(1)],
\end{align}
where $\mu_{A_n}$ is supported on $\cup_k [{\sf a}_k-\oo(1), {\sf b}_k+\oo(1)]$. As a consequence, $\supp\mu_{t}$ is contained in a $\fC\sqrt t+\oo(1)$ neighborhood of $\supp\rho^*_{t}$.
\end{proposition}

We will need the following  monotonicity statement for nonintersecting Brownian bridges,  which follows directly from \cite[Lemmas 2.6 and 2.7]{corwin2014brownian}. 
\begin{theorem}[{\cite[Lemmas 2.6 and 2.7]{corwin2014brownian}}]\label{thm:extreme}
Given two pairs of boundary data $(a_1\leq a_2\leq \cdots\leq a_n)$, $(b_1\leq b_2\leq \cdots\leq b_n)$,  $(a'_1\leq a'_2\leq \cdots\leq a'_n)$ and $( b_1'\leq  b'_2\leq \cdots\leq  b'_n)$. We consider nonintersecting Brownian bridges from $(a_1,a_2,\cdots, a_n)$ and $(a'_1, a'_2, \cdots,a'_n)$ to $(b_1, b_2, \cdots,b_n)$ and $(b'_1,b'_2, \cdots, b'_n)$: $x_1(t)\leq x_2(t)\cdots\leq x_n(t)$ and $x'_1(t)\leq x'_2(t)\cdots\leq x'_n(t)$. If $a_i\geq a'_i$ and $b_i\geq b'_i$ for all $1\leq i\leq n$, then there exists  a coupling, such that at any time $0\leq t\leq 1$, $x_i(t)\geq x'_i(t)$ for all $1\leq i\leq n$. 
\end{theorem}

\begin{proof}[Proof of Proposition \ref{p:extremeeig}]
For weighted nonintersecting Brownian bridges \eqref{e:wdensity} with boundary data $(a_1\leq a_2\leq \cdots\leq a_n)$ and $(b_1\leq b_2\leq \cdots\leq b_n)$, and $\mu_{A_n}$ is supported on $\cup_k [{\sf a}_k-\oo(1), {\sf b}_k+\oo(1)]$.
We can take $a'_1=a'_2=\cdots=a_n'=a_1={\sfa}_1-\oo(1)\leq a_1$, and $b'_1=b'_2=\cdots=b'_n= b_1$. Then the nonintersecting Brownian bridges $x'_1(t)\leq x'_2(t)\leq \cdots\leq x'_n(t)$ is the well-understood Brownian watermelon (after an affine shift), and we have that with high probability $x'_1(t)\geq \sfa_0-\fC \sqrt{t}+\oo(1)$. Then Theorem \ref{thm:extreme} implies that $x_1(t)\geq x'_1(t)\geq \sfa_0-\fC \sqrt{t}+\oo(1)$. By the same argument we can show that 
\begin{align*}
x_i(t)\leq \sfb_k+\fC \sqrt{t}+\oo(1),\quad x_{i+1}(t)\leq \sfa_{k+1}-\fC \sqrt{t}+\oo(1),\quad i=n\int_{-\infty}^{\sfb_k+\oo(1)} \rd\mu_{A_n},
\end{align*}
and the claim \ref{e:extreme} follows.
\end{proof}

\begin{proposition} \label{p:bulkrigid3}
Under the assumptions of Theorem \ref{t:rigidity}, with high probability, the Stieltjes transform of the weighted nonintersecting Brownian bridge \eqref{e:wdensity}
 satisfies the following estimates: for any $0\leq t\leq \ft_0$, and
 any $z\in \bC_+$ outside the contour $z(\mathscr  C_t(\fr))$ as defined in \eqref{e:Ctr}, it holds for $z=z(Z)$
\begin{align}\label{e:bulkrigid3}
    \left|\tilde m_t(z)-m_t(z)\right|\lesssim \frac{(\log n)^{ \fa}}{n}.
\end{align}
\end{proposition}

We define the stopping time 
\beq\label{stoptime}
\sigma\deq \inf_{t\geq0}\left\{\exists W\in \cD_t(\fr): w=Z(W), \left|\Delta_t(w)\right|\geq \frac{\fM(t)}{n}\right\}\wedge \ft_0,\quad 
\fM(t)\deq \fC e^{\fC t}(\log n)^{ \fa},
\eeq
for some sufficiently large $\fC$, which we will choose later.
We will prove that with high probability it holds that $\sigma=\ft_0$. Then this implies that Proposition \ref{p:bulkrigid3} holds for $z\in z(\cD_t(\fr))$. For $z\not\in z(\cD_t(\fr))$ and outside the contour $z(\mathscr C_t(\fr))$, the estimate of the Stieltjes transform follows from a contour integral
\begin{align}\label{e:outC}
 \left|\tilde m_t(z)-m_t(z)\right|
 =\left|\frac{1}{2\pi \ri}\oint_{z(\mathscr C_t(\fr))} \frac{\tilde m_t(w)-m_t(w)}{z-w}\rd w\right|\lesssim \frac{(\log n)^{ \fa}}{n}.
\end{align}

For any $Z\in \cL$ as in \eqref{def:L}, we define $t(Z)$ to be the earliest time such that $Z_t(Z)$ no longer belongs to $\dom_t(\fr)$ up to $\ft_0$,
\begin{align}\label{e:deftZ}t(Z)\deq \sup_{s\geq0}\{ Z_t(Z)\in \dom_t(\fr)\}\wedge \ft_0.\end{align}

\begin{claim}\label{p:error}
Under the assumptions of Theorem \ref{t:rigidity}. There exists an event $\cW$ that holds with high probability on which we have for every  $Z \in \cal L$ as in \eqref{def:L} and $0 \leq t \leq t (Z)\wedge \sigma$,
\beq\label{eq:estet}
 \left|\int_0^{t} \rd \cE_2(s)\right|\lesssim \frac{ \log n}{n}.
\eeq
\end{claim}

\begin{proof}

We have for any $0\leq t\leq t(Z)\wedge\sigma$, Proposition \ref{p:extremeeig} implies that 
\begin{align}\label{e:eigloc}
\{x_1(t), x_2(t), \cdots, x_n(t)\}\in \cup_k [{\sf a}_k- \fC\sqrt{t}-\oo(1), {\sf b}_k+\fC\sqrt{t}+\oo(1)],
\end{align}
which is a small neighborhood of $\supp(\mu_A)$. By taking $\ft_0$ sufficiently small and $n$ large enough, we can make sure that the righthand side of \eqref{e:eigloc} is distance $\OO(1)$ away from the contour $z(\mathscr C_t(\fr))$. By the definition of $t(Z)$ as in \eqref{e:deftZ}, we have that $Z_s(Z)\in \cD_s(\fr)$ for $0\leq s\leq t\leq t(Z)\wedge \sigma$. Especially $z_s(Z)$ is outside the contour $z(\mathscr C_s(\fr))$. We conclude that
\begin{align}\label{e:ddbound}
\dist(z_s, \{x_1(s), x_2(s), \cdots, x_n(s)\})\gtrsim 1,\quad z_s=z_s(Z)
\end{align}
Then it follows that 
\begin{align*}\begin{split}
&\left\langle\frac{1}{ n^{3/2}}\int_0^{\cdot\wedge\sigma}\sum_{i=1}^n \frac{{\rm d} B_i(s)}{(x_i(s)-z_s)^2} \right\rangle_{t}
= \frac{1}{n^{3}}\int_0^{t\wedge\sigma}\sum_{i=1}^n \frac{ \rd s}{|x_i(s)-z_s|^4}
\lesssim \frac{1}{n^2}.
\end{split}\end{align*} 
where we used \eqref{e:ddbound}.
Therefore, by Burkholder-Davis-Gundy inequality, for any $Z\in \cal L$, the following holds with high probability, i.e., $1-n^{-\fC}$,
\begin{align}\begin{split}\label{eq:dd2term}
&\sup_{0\leq t\leq t(Z)}\left|\frac{1}{ n^{3/2}}\int_0^{t\wedge\sigma}\sum_{i=1}^n \frac{{\rm d} B_i(s)}{(x_i(s)-z_s)^2} \right|
\lesssim \frac{ \log n}{n}.
\end{split}
\end{align}
We define $\cW$ to be the set of Brownian paths $\{B_1(s), \cdots, B_n(s)\}_{0\leq s\leq \ft_0}$ on which \eqref{eq:dd2term} holds for  any $Z\in \cal L$.
From the discussion above,  $\cW$ holds with overwhelming probability, i.e., $\bP(\cW)\geq 1-\OO(|\cal L|n^{-\fC})$. This finishes the proof of Claim \ref{p:error}.
\end{proof}
We now bound the second term of \eqref{defcE1}.

\begin{claim}\label{prop:gtsumbound}
Under the assumptions of Theorem \eqref{t:rigidity}, for any $Z\in \cal L$ and $0\leq t\leq t(Z)\wedge \sigma$ as in \eqref{e:deftZ}, we have
\beq\label{e:gtsumbound}
\int\frac{g_t(x)-g_t(Z_t)-(x-z_t)\del_z g_t(Z_t)}{(z_t-x)^2} (\rd \mu_t-\rho_t^*(x)\rd x)\lesssim \frac{\fM(t)}{n},
\eeq
where the implicit constant is independent of $\fM(t)$.
\end{claim}

\begin{proof}
We recall from \eqref{e:gtZdef}, 
$g_t(Z)$ can be glued to an analytic function in a neighborhood of the bottom of $\cC_t^+$, i.e. in a neighborhood of $\pi_{Q^t}(\{(x,t): x\in \supp(\rho_t^*)\})$ and its complex conjugate. The integrand 
can also be extended to a meromorphic function on $\cC_t^+$:
\begin{align}\label{e:integrand}
\frac{g_t(W)-g_t(Z_t)-(z(W)-z_t)\del_z g_t(Z_t)}{(z_t-z(W))^2},\quad W\in \cC^+_t.
\end{align}
It only has poles at $\{W\in \cC_t^+: W\neq Z_t, z(W)=z_t\}$.
By our choice of $\cD_t(\fr)$, $z(\cC_t^+\setminus \cD_t(\fr))$ is a bijection with its preimage under the projection map $z: \cC_t^+\mapsto \bC$. So the poles of \eqref{e:integrand} are 
 in $\cD_t(\fr)$.
 Then we can take the contour $\mathscr C_t(\fr)$, whose projection on the $z$-plane encloses $\supp \mu_t, \supp \rho_t^*$.
We can rewrite \eqref{e:gtsumbound} as a contour integral, 
\begin{align*}
\frac{1}{2\pi\ri}\oint_{\mathscr C_t(\fr)}\frac{g_t(W)-g_t(Z_t)-(w-z_t)\del_z g_t(Z_t)}{(z_t-w)^2} (\tilde m_t(w)-m_t(w)) \rd W,\quad w=z(W).
\end{align*}
On the contour $\mathscr C_t(\fr)$, the integrand \eqref{e:integrand} is of order $\OO(1)$, and our definition of the stopping time \eqref{stoptime} implies that $|\tilde m_t(w)-m_t(w)|\leq \fM(t)/n$. It follows that
\begin{align*}
\left| \frac{1}{2\pi\ri}\oint_{\sfS}\frac{g_t(W)-g_t(Z_t)-(w-z_t)\del_z g_t(Z_t)}{(z_t-w)^2} (\tilde m_t(w)-m_t(w)) \rd W\right|\lesssim  \oint_{\sfS} \frac{\fM(t)}{n}|\rd W|\lesssim \frac{\fM(t)}{n}.
\end{align*}
The claim \eqref{e:gtsumbound} follows.
\end{proof}

\begin{claim}\label{prop:gtchange}
Under the assumptions of Theorem \eqref{t:rigidity}, for any $Z\in \cal L$ and $0\leq t\leq t(Z)\wedge \sigma$ as in \eqref{e:deftZ}, we have
\beq\label{e:getchange1}
|(g_t(x_i(t);\mu_t,t)-g_t(x_i(t))|\lesssim \frac{\fM(t)}{n}
\eeq
and it follows that
\beq\label{e:getchange2}
\left|\frac{1}{n}\sum_{i=1}^n\frac{(g_t(x_i(t);\mu_t,t)-g_t(x_i(t)))+\cE^{(i)}_t(\bmx(t))}{(z_t-x_i(t))^2}\right|\lesssim \frac{\fM(t)|\Im[\tilde m_t(z_t)]|}{n\Im[z_t]},
\eeq
where the implicit constant is independent of $\fM(t)$. When $\dist(z_t, \{x_1(t), x_2(t),\cdots, x_n(t)\})\gtrsim 1$, the righthand side of \eqref{e:getchange2} simplifies to $\OO(\fM(t)/n)$.
\end{claim}

\begin{proof}
To estimate \eqref{e:getchange1} we interpolate the probability measures $\rd\mu_t$ and $\rho_t^*(x)\rd x$ 
\begin{align*}
\rd\mu^\theta_t=\theta\rd\mu_t+(1-\theta) \rho_t^*(x)\rd x, \quad 0\leq \theta\leq 1,
\end{align*}
and write 
\begin{align*}
g_t(x;\mu_t,t)-g_t(x)=\int_0^1 \del_\theta g_t(x;\mu_t^\theta,t)\rd \theta.
\end{align*}
We recall that $g_t(x;\mu_t^\theta, t)=f_t(x;\mu_t^\theta,t)-\int1/(x-y)\rd \mu_t^\theta(y)$. Using \eqref{p:derft}, we have 
\begin{align}\label{e:difg}
g_t(x;\mu_t,t)-g_t(x)
=\int_0^1 \int \left(Q_{X_\theta(y), Z_0}(x;\mu_t^\theta,t)
-\frac{1}{y-x}\right)(\rd \mu_t(y)-\rho_t^*(y)\rd y)\rd \theta,
\end{align}
where $\{X_\theta(x)=(f_t(x;\mu^\theta,t), x)\}_{x\in \supp(\mu^\theta)}$. 
The integrand as a function of $y$, can  be extended to a meromorphic function on $\cC_t^{\mu_t^\theta, t,+}$:
\begin{align}\label{e:integrand2}
Q_{W, Z_0}(x;\mu_t^\theta,t)
-\frac{1}{z(W)-x},\quad W\in \cC_t^{\mu_t^\theta, t,+}.
\end{align}
More importantly, it glues to an analytic in a neighborhood of the bottom boundary of $\cC_t^{\mu_t^\theta, t,+}$. We can take a contour $\sfS$ over $\cC_t^{\mu_t^\theta, t,+}$, such that its projection encloses a neighborhood of $\supp \mu_t, \supp \rho_t^*$, 
and is contained in $z(\cD_t(\fr))$, and rewrite \eqref{e:difg} as a contour integral,
\begin{align*}
\int_0^1 \frac{1}{2\pi\ri}\oint_{\sfS} \left(Q_{W, Z_0}(x;\mu_t^\theta,t)
-\frac{1}{z(W)-x}\right)(\tilde m_t(w)-m_t(w)) \rd W \rd\theta,\quad w=z(W).
\end{align*}
On the contour $\sfS$, the integrand \eqref{e:integrand2} is of order $\OO(1)$, and our definition of the stopping time \eqref{stoptime} implies that $|\tilde m_t(w)-m_t(w)|\leq \fM(t)/n$. It follows that
\begin{align*}
\left|\int_0^1 \frac{1}{2\pi\ri}\oint_{\sfS} \left(Q_{W, Z_0}(x;\mu_t^\theta,t)
-\frac{1}{z(W)-x}\right)(\tilde m_t(w)-m_t(w)) \rd W\rd\theta\right|\lesssim \int_0^1 \oint_{\sfS} \frac{\fM(t)}{n}|\rd W|\rd\theta\lesssim \frac{\fM(t)}{n}.
\end{align*}
The claim \eqref{e:getchange1} follows.

For  \eqref{e:getchange2}, using Proposition \ref{p:cEbound} and \eqref{e:getchange1}, 
\begin{align*}
\left|\frac{1}{n}\sum_{i=1}^n\frac{(g_t(x_i(t);\mu_t,t)-g_t(x_i(t)))+\cE^{(i)}_t(\bmx(t))}{(z_t-x_i(t))^2}\right|\lesssim
\frac{1}{n}\sum_{i=1}^n\frac{\fM(t)/n+1/n^2}{|z_t-x_i(t)|^2}
\lesssim \frac{\fM(t)|\Im[\tilde m_t(z_t)]|}{n\Im[z_t]}.
\end{align*}
This gives the claim  \eqref{e:getchange2}. When $\dist(z_t, \{x_1(t), x_2(t),\cdots, x_n(t)\})\gtrsim 1$, the above simplifies to $\OO(\fM(t)/n)$.
\end{proof}

We can now start analyzing \eqref{eq:mzt}. For any lattice point $Z\in \cal L$ and $0\leq t\leq t(Z)$ as in \eqref{e:deftZ}, by Claims \ref{p:error}, \ref{prop:gtsumbound} and \ref{prop:gtchange}, we have 
\begin{align}\label{e:Deltafar}
\left|\Delta_{t\wedge\sigma}(z_{t\wedge\sigma})\right|\lesssim |\Delta_{0}(z_{0})|+\int_0^{t\wedge\sigma}|\Delta_s(z_s)||\del_z \tilde f_s(Z_s)|\rd s+
\frac{\log n}{n}+\int_0^{t\wedge \sigma}\frac{\fM(s)\rd s}{n}.
\end{align}
Notice that for $s\leq t\wedge \sigma$,
we have
$\left|\del_z \tilde f_s(Z_s)\right|
\leq \left|\del_z \tilde m_s(z_s)\right|+\left|\del_z g_s(Z_s)\right|
\lesssim 1$. Assumption \ref{a:A_n} implies that $|\Delta_0(z_0)|\lesssim (\log n)^{ \fa}/n$.
We can simplify \eqref{e:Deltafar} as
\begin{align}\label{e:Deltafar2}
\left|\Delta_{t\wedge\sigma}(z_{t\wedge\sigma})\right|\lesssim
\frac{(\log n)^{ \fa}}{n}+\int_0^{t\wedge \sigma}\frac{\fM(s)\rd s}{n}.
\end{align}
We recall from \eqref{stoptime}, that $\fM(t)=\fC e^{\fC t}(\log n)^{ \fa}$. We can further simplify the righthand side of \eqref{e:Deltafar2} and conclude that on the event $\cW$,
\begin{align*}
\left| \Delta_{t\wedge\sigma}(z_{t\wedge\sigma})\right|\lesssim 
\frac{(1+e^{\fC (t\wedge \sigma)})(\log n)^{ \fa}}{n}\leq \frac{\fC e^{\fC (t\wedge \sigma)}(\log n)^{ \fa}}{2n},
\end{align*}
provided $\fC$ is large enough. By Proposition \ref{prop:continuity}, for any $W\in \dom_{t\wedge\sigma}(\fr)$ with $w=z(W)$, there exists some $Z\in \cal L$ such that $z_{t\wedge\sigma}(Z)\in \dom_{t\wedge\sigma}(\fr)$, and 
\begin{align*}
|z_{t\wedge\sigma}(Z)-z(W)|\lesssim 1/n^2.
\end{align*}
Moreover, on the domain $z(\cal D_{t\wedge \sigma}(\fr))$, both $\td m_{t\wedge \sigma}$ and $m_{t\wedge\sigma}$ are Lipschitz with Lipschitz constant $\OO(1)$. Therefore 
\begin{align*}\begin{split}
&\phantom{{}={}}\left|\td m_{t\wedge\sigma}(w)-m_{t\wedge\sigma}(w)\right|\leq 
\left|\td m_{t\wedge\sigma}(z_{t\wedge\sigma}(Z))-m_{t\wedge\sigma}(z_{t\wedge\sigma}(Z))\right|\\
&+
\left|\td m_{t\wedge\sigma}(w)-\td m_{t\wedge\sigma}(z_{t\wedge\sigma}(Z))\right|+
\left| m_{t\wedge\sigma}(w)-m_{t\wedge\sigma}(z_{t\wedge\sigma}(Z))\right|\\
&\leq \frac{\fC e^{\fC (t\wedge \sigma)}(\log n)^{ \fa}}{2n}+\OO\left(\frac{1}{n^2}\right)\leq \frac{\fC e^{\fC (t\wedge \sigma)}(\log n)^{\fa}}{n}.
\end{split}\end{align*}
It follows that $\sigma = \ft_0$ and  this completes the proof of Proposition \ref{p:bulkrigid3}.

\subsection{Long time estimates}
\label{s:longtime}

In this section, we prove Proposition \ref{p:bulkrigid2} and \eqref{e:edgerigid2}. For them we need to understand the Stieltjes transform close to the spectrum. We define the following spectral domain as illustrated in Figure \ref{f:domain} :
\begin{align}\begin{split}\label{e:bulkdomain}
\cD_t^{\rm bulk}&=\left(\cD_t(\max\{\fr+(\ft_0-t)/2,0\})\setminus \cD_t(2\fr)\right)\bigcap\\
&
\{(f,z)\in \cC_t^+: \dist(z, [\sfa(t), \sfb(t)])|\Im[f]|\wedge|\Im[z]|\geq (\log n)^{ \fa+2}/n)\}.
\end{split}\end{align}
The estimates of Stieltjes transform on $\{(f,z)\in \cC_t^+: \dist(z, [\sfa(t), \sfb(t)])|\Im[f]|\wedge|\Im[z]|\geq (\log n)^{ \fa+2}/n)\}$ with $\Re[z]\in [\sfa(t), \sfb(t)]$ give the information of particles inside the bulk, and with $\Re[z]\notin [\sfa(t), \sfb(t)]$ can be used to control the locations of extreme particles.

\begin{figure}
\begin{center}
 \includegraphics[scale=0.22,trim={0cm 9cm 0 14cm},clip]{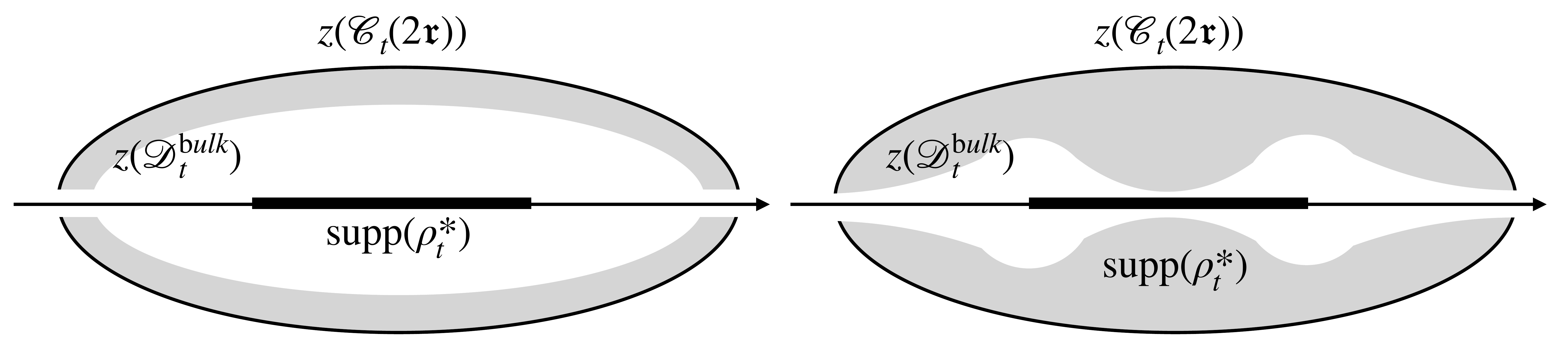}
 \caption{The projection of the spectral domain $\cD_t^{\rm bulk}$ on the $z$-plane. The left subplot is the projection of $\cD_t(\max\{\fr+(\ft_0-t)/2,0\})\setminus \cD_t(2\fr)$ on the $z$-plane with $\fr+(\ft_0-t)/2>0$. The right subplot is the projection of $\{(f,z)\in \cC_t^+: \dist(z, [\sfa(t), \sfb(t)])|\Im[f]|\wedge|\Im[z]|\geq (\log n)^{ \fa+2}/n)\}$ on the $z$-plane.}
 \label{f:domain}
 \end{center}
 \end{figure}

We recall the contours $\mathscr C_t(r)$ as defined in \eqref{e:Ctr}, which are the bottom boundary curves of $\cD_t(r)$. For any $Z=(f,z)\in \cC_t^+$, $f=m_t(z)+g_t(Z)$ (by \eqref{e:gtZdef}). Since $g_t(Z)$ is real for $z\in \bR$, we have $|\Im[g_t(Z)]|=\OO(\Im[z])\lesssim |\Im[m_t(z)]|$. 
Combining a contour integral similar to \eqref{e:outC}, the estimates on $\cD_t^{\rm bulk}$ and $\cD_t(2\fr)$ give full information of the  Stieltjes transform on 
\begin{align*}
\{z\in\bC: \dist(z, [\sfa(t), \sfb(t)])|\Im[m_t(z)]|\wedge|\Im[z]|\gtrsim (\log n)^{ \fa+2}/n)\},
\end{align*}
which is the domain in Proposition \ref{p:bulkrigid2}.

We recall from \eqref{e:Ctr} that the curve $z(\mathscr C_t(r))$ is characterized by $\Im[z]+r\Im[f_t(z)]=0$, where $(f_t(z),z)\in \cC_t^+$.
Proposition \ref{p:squareroot} implies that for $\ft_0\leq t\leq 1$, the density  $\rho_t^*$ has square root behavior. Thus $f_t(z)$ has square root behavior in a small neighborhood of $\sfa(t)$,
\begin{align}\label{e:localft}
f_t(z)=f_t(\sfa(t))+\Omega(1)\sqrt{\sfa(t)-z}+\OO(|z-\sfa(t)|^{3/2})
\end{align}
where the square root is the branch with nonpositive imaginary part. 
Therefore, for $r\ll1$, in a small neighborhood of $\sfa(t)$, 
  if $z=\sfa(s)-\kappa+\ri\eta\in z(\mathscr C_t(r))$, solving $\Im[z]+r\Im[f_t(z)]=0$, we get that $\kappa+\eta\asymp r^2=\oo(1)$ for $\kappa\geq 0$, and 
  $\eta^2/(|\kappa|+\eta)\asymp r^2=\oo(1)$ for $\kappa\leq 0$.

For any $Z\in \cL$ as in \eqref{def:L}, we define $t(Z)$ such that $Z_t(Z)$ no longer belong to $\dom_t^{\rm bulk}$ after time $t(Z)$:
\begin{align}\label{e:deftz2}
t(Z)\deq \sup_{t\geq 0}\{Z_t(Z)\in \cD_t(2\fr)\cup\cD_t^{\rm bulk}\}\wedge 1,
\end{align}
and the stopping time
\begin{align}\begin{split}\label{stoptime2}
\sigma&\deq \inf_{t\geq0}\left\{\exists W\in \cD_t(2\fr): w=z(W), \left|\Delta_t(w)\right|\geq \frac{\fM_1(t)}{n}\right\}\\
&\wedge\inf_{t\geq 0}\left\{\exists W\in \cD_t^{\rm bulk}: w=z(W), \left|\Delta_t(w)\right|\geq \frac{\fC(\log n)^{ \fa+1}}{n\sqrt{\Im[w]\dist(w, [\sfa(t), \sfb(t)])}}\right\}\wedge 1,
\end{split}\end{align}
for some sufficiently large $\fC$, which we will choose latter.
Thanks to Proposition \ref{p:bulkrigid3}, on the event $\cW$ as defined in Claim \ref{p:error}, we have that $\sigma\geq \ft_0$.

\begin{claim}\label{p:gap}
Under the assumptions of Theorem \ref{t:rigidity}, there exists a universal constant $\fC>0$ such that for any $\ft_0\leq t\leq 1$ and $Z\in \cD_0(\fr)$ if $z_t(Z)=\sfa(t)-\kappa_t(Z)+\ri\eta_t(Z)$ and $\kappa_t(Z)\geq 0$, then 
for any $\ft_0\leq s\leq t$ either $\eta_s(Z)\gtrsim 1$, or $\kappa_s(Z)\gtrsim 1$, or
\begin{align}\label{e:sqrtkappa}
\del_s\sqrt{\kappa_s(Z)}\leq -\fC\sqrt{\kappa_s(Z)+\eta_s(Z)}.
\end{align}
The same statement holds for characteristic flows close to the right edge. As a consequence 
if $Z_t(Z) \in \cD_t^{\rm bulk}\cup \cD_t(2\fr)$, then $Z_s(Z) \in \cD_s^{\rm bulk}\cup \cD_s(2\fr)$.
If we further assume that $\ft_0\leq t\leq \sigma$ as defined in \eqref{stoptime2} and $Z_t(Z)\in \cD_t^{\rm bulk}$, then we have
\begin{enumerate}
\item We denote
\begin{align}\begin{split}\label{e:defsAB}
&\sfA(t)\deq \sfa(t)-(\log n)^{ 2\fa+4}/n^{2/3} \wedge \min\{E: E+\ri \eta\in z(\mathscr C_t(\max\{\fr+(\ft_0-t)/2, 0\}))\},\\
&\sfB(t)\deq  \sfb(t)+(\log n)^{ 2\fa+4}/n^{2/3} \vee \max\{E: E+\ri \eta\in z(\mathscr C_t(\max\{\fr+(\ft_0-t)/2, 0\}))\}.
\end{split}\end{align}
The particles are all in the interval $[\sfA(t), \sfB(t)]$,
\begin{align}\label{e:extremepp}
\sfA(t)\leq x_1(t)\leq x_2(t)\leq \cdots\leq x_n(t)\leq \sfB(t).
\end{align}
\item Let $z_t=z(Z_t)$ the Stieltjes transform satisfies
\beq\label{e:replacebound}
|\Delta_{t}(z_t)|\lesssim(\log n)^{\fa+1}/(n\sqrt{\eta_t\dist(z_t, [\sfa(t), \sfb(t)])})\lesssim \Im[m_{t}(z_t)]/\log n.
\eeq
 \item For any $\ft_0\leq s\leq t$, 
\begin{align}\label{e:gap}
\dist(z_s(Z), \{x_1(s), x_2(s),\cdots, x_n(s)\})\gtrsim \eta_s(Z)+\sqrt{\kappa_s(Z)+\eta_s(Z)}(t-s).
\end{align}
\end{enumerate}
\end{claim}


\begin{proof}

The left edge $\sfa(s)$ of the density $\rho_s^*(x)$ is a critical point of $f_s$, which is characterized by $f_s'(\sfa(s))=\infty$ and satisfies the differential equation 
\begin{align}\label{e:edgeEqn}
\del_s \sfa(s)=f_s(\sfa(s)).
\end{align}
By taking difference between \eqref{e:edgeEqn} and  the characteristic flow $\del_s z_s(Z)=f_s(z_s(Z))$, and then taking the real part we get
\begin{align}\label{e:diff}
\del_s \kappa_s
=\Re[f_s(\sfa_s)-f_s(z_s(Z))].
\end{align}
There are several cases, either $\Im[z_s(Z)]\gtrsim 1$, then $\Im[z_{s'}(Z)]\geq\Im[z_s(Z)] \gtrsim 1$ for all $0\leq s'\leq s$; or $\kappa_s(Z)\gtrsim 1$; or $z_s(Z)$ is in a sufficiently small neighborhood of $\sfa(s)$. In this case, 
%
using the square root behavior \eqref{e:localft} of $f_t(z)$, \eqref{e:diff} implies that
\begin{align}\label{e:diff2}
\del_s \kappa_s
=\Re[f_s(\sfa_s)-f_s(z_s(Z))]\lesssim -\sqrt{\eta_s(Z)+\kappa_s(Z)}.
\end{align}
This gives \eqref{e:sqrtkappa}.

If $Z_t(Z) \in \cD_t^{\rm bulk}\cup \cD_t(2\fr)$, then $Z_t(Z)\in \cD_t(\max\{\fr+(\ft_0-t)/2,0\})$. By our construction \eqref{e:defDtr}
$Z_s(Z) \in \cD_s(\max\{\fr+(\ft_0-s)/2,0\})$. Moreover, since along the characteristic flow, for $s\leq r\leq t$, $\del_r f_r(Z_r)=0$, $\del_r |\Im[z_r]|\leq 0$, and when $z_r=\sfa(r)-\kappa_r(Z)+\ri \eta_r(Z)$ with $\kappa_r(Z)\geq 0$ is in a small neighborhood of $\sfa(r)$, $\del_r \sqrt{\kappa_r(Z)}\leq 0$ by \eqref{e:sqrtkappa}. We conclude that $\dist(z_s, [\sfa(s), \sfb(s)])|\Im[f_s(Z_s)]|\wedge|\Im[z_s]|\geq (\log n)^{ \fa+2}/n$, and $Z_s(Z) \in \cD_s^{\rm bulk}\cup \cD_s(2\fr)$

The estimate \eqref{e:extremepp} follows from estimates of the Stieltjes transform. More precisely, we take $\kappa\geq \sfA(t)-\sfa(A)\geq (\log n)^{2\fa+4}/n^{2/3}$ and $\eta=n^{-2/3}$.
Then for $t\leq \sigma$, we have
\begin{align}\label{e:mtbound}
    |\tilde m_t(\sfa(t)-\kappa+\ri \eta)- m_t(\sfa(t)-\kappa+\ri \eta)|\lesssim\frac{(\log n)^{ \fa+1}}{n\sqrt{\eta(\kappa+\eta)}}.
\end{align}
Thanks to the square root behavior of $m_t(z)$ from Proposition \ref{p:squareroot}, $\Im[m_t(\sfa(t)-\kappa+\ri \eta)]\asymp \eta/\sqrt{\kappa+\eta}$. Thus we have
\begin{align}\label{e:tmeig}
\tilde m_t(\sfa(t)-\kappa+\ri \eta)\lesssim \frac{\eta}{\sqrt{\kappa+\eta}}+\frac{(\log n)^{ \fa+1}}{n\sqrt{\eta(\kappa+\eta)}}.
\end{align}
Thus by our choice take $\kappa\geq (\log n)^{2\fa+4}\eta$ and $\eta=n^{-2/3}$, \eqref{e:tmeig} implies that  $\Im[\tilde m_t(\sfa(t)-\kappa+\ri \eta)]\ll 1/n\eta$. 
However, if there exists a particle $x_i(t)$ such that $|x_i(t)-\sfa(t)+\kappa|\leq \eta$, we will have
 that 
\begin{align}\label{e:sumerror}
    \Im[m_t(\sfa(t)-\kappa+\ri \eta)]=\frac{1}{n}\sum_{i=1}^n \frac{\eta}{(x_i(t)-\kappa-\sfa(t))^2+\eta^2}\geq \frac{1}{2n\eta}.
\end{align}
This leads to a contradiction. Since we can take any $\kappa\geq \sfA(t)-\sfa(t)$, we conclude that $x_i(t)\geq \sfA(t)$ for all $1\leq i\leq n$. The same argument implies that $x_i(t)\leq \sfB(t)$ for all $1\leq i\leq n$. This finishes the proof of \eqref{e:extremepp}.

The estimate \eqref{e:replacebound} follows from the definition of the spectral domain $\cD_t^{\rm bulk}$ as in \eqref{e:bulkdomain}.

For \eqref{e:gap}, since $\kappa_t(Z)\geq 0$, using \eqref{e:sqrtkappa}, either $\eta_s(Z)\gtrsim 1$ or $\kappa_s(Z)\geq 0$. If $\eta_s(Z)\gtrsim 1$, \eqref{e:gap} follows trivially,
\begin{align*}
\dist(z_s(Z), \{x_1(s), x_2(s),\cdots, x_n(s)\})\geq \eta_s(Z)
\gtrsim 1\gtrsim \eta_s(Z)+\sqrt{\kappa_s(Z)+\eta_s(Z)}(t-s).
\end{align*}
We consider the case that $\kappa_s(Z)\geq 0$. Using \eqref{e:extremepp}, we have 
\begin{align}\label{e:gap1}
\dist(z_s(Z), \{x_1(s), x_2(s),\cdots, x_n(s)\})\geq \dist(z_s(Z), [\sfA(s), \sfB(s)])\geq \eta_s(Z).
\end{align}
We recall $\sfA(t)$ and $\sfB(t)$ from \eqref{e:defsAB}, and the behavior of the curve $z(\mathscr C_t(r))$ near $\sfa(t)$ from the discussion after \eqref{e:localft}. 
We take $r=\max\{\fr+(\ft_0-s)/2, (\log n)^{\fa+2}/n^{1/3}\}$, then $[\sfA(s), \sfB(s)]$ is inside the contour $z(\mathscr C_t(r))$. Let $r'$ to be such that $Z_s(Z)\in \mathscr C_s(r')$. Then we have
\begin{align}\begin{split}\label{e:gap2}
\dist(z_s(Z), [\sfA(s), \sfB(s)])
&\geq 
\dist(\{z\in z(\mathscr C_s(r')): \Re[z]\leq \sfa(s)\}, \{z\in z(\mathscr C_s(r)): \Re[z]\leq\sfa(s)\}).
\end{split}\end{align}
We have that $r'\geq (t-s)/2+ \max\{\fr+(\ft_0-s)/2, (\log n)^{\fa+2}/n^{1/3}\}=(t-s)/2+r$.
%

 If $r'\gtrsim 1$,  $\dist(\{z\in z(\mathscr C_s(r')): \Re[z]\leq \sfa(s)\}, \sfa(s))\gtrsim 1$.  Moreover, since $f_s(z)$ is analytic away from the spectrum, in this case \eqref{e:gap2} implies
 \begin{align}\begin{split}\label{e:gap3}
\dist(z_s(Z), [\sfA(s), \sfB(s)])
&\geq 
\dist(\{z\in z(\mathscr C_s(r')): \Re[z]\leq \sfa(s)\}, \{z\in z(\mathscr C_s(r)): \Re[z]\leq\sfa(s)\})\\
&\gtrsim (r'-r)\gtrsim (t-s)\gtrsim \sqrt{\kappa_s(Z)+\eta_s(Z)}(t-s).
\end{split}\end{align}
The claim \eqref{e:gap} follows from combining \eqref{e:gap1} and \eqref{e:gap3}.

If $r'=\oo(1)$, let $z=\sfa(s)-\kappa+\ri\eta\in z(\mathscr C_s(r'))$, solving $\Im[z]+r'\Im[f_s(z)]=0$, we get that $\kappa+\eta\asymp (r')^2=\oo(1)$. It follows that $r'\asymp \sqrt{\kappa_s(Z)+\eta_s(Z)}$. Similarly, for any $z=\sfa(s)-\kappa+\ri\eta\in z(\mathscr C_s(r))$, we have that $\kappa+\eta\asymp r^2=\oo(1)$.
Thus \eqref{e:gap2} implies
\begin{align}\begin{split}\label{e:gap4}
\dist(z_s(Z), [\sfA(s), \sfB(s)])
&\geq 
\dist(\{z\in z(\mathscr C_s(r')): \Re[z]\leq \sfa(s)\}, \{z\in z(\mathscr C_s(r)): \Re[z]\leq\sfa(s)\})\\
&\gtrsim r'(r'-r)
\gtrsim (t-s)\sqrt{\kappa_s(Z)+\eta_s(Z)}.
\end{split}\end{align}
The claim \eqref{e:gap} follows from combining \eqref{e:gap1} and \eqref{e:gap4}.

\end{proof}

The following estimates on the integral of Stieltjes transform and characteristic flows will be used 
 in later proofs.
\begin{proposition}\label{p:propzt}
Under the assumptions of Theorem \ref{t:rigidity}, for any $Z\in \cD_0(\fr)$, and $0\leq s\leq t\leq t(Z)$, let $z_\tau=z(Z_\tau(Z))$, then it holds
\begin{align}\label{e:propzt}
\int_{s}^t\frac{\Im[m_\tau(z_\tau)] \rd \tau}{\Im[z_\tau]}\leq \fC(t-s)+\log \frac{\Im[z_s]}{\Im[z_t]},\qquad \int_{s}^t\frac{\Im[m_\tau(z_\tau)]\rd \tau}{\Im[z_\tau]^p} \leq \frac{\fC}{\Im[z_t]^{p-1}},\quad p>1
\end{align}
\end{proposition}
\begin{proof}

%
For \eqref{e:propzt} we have 
\begin{align*}
\int_{s}^t\frac{\Im[m_\tau(z_\tau)] \rd \tau}{\Im[z_\tau]^p}=
\int_{s}^t\frac{\Im[f_\tau(Z_\tau)-g_\tau(Z_\tau)] \rd \tau}{\Im[z_\tau]^p}
\leq 
\int_{s}^t\frac{\del_\tau \Im[z_\tau]}{\Im[z_\tau]^p}
+\int_{s}^t\frac{\fC\rd \tau}{\Im[z_\tau]^{p-1}}
\end{align*}
where we used that $\Im[g_\tau(Z_\tau)]\lesssim \Im[z_\tau]$, and $f_\tau(Z_\tau)=\del_\tau z_\tau$.
The claim \eqref{e:propzt} follows.
\end{proof}

\begin{claim}\label{c:error}
Under the assumptions of Theorem \ref{t:rigidity},  there exists an event $\cW$  with high probability, such that on $\cW$ we have for every  $Z \in \cal L$ and $0 \leq t \leq  \sigma$,
if $z_t(Z)\in \cD_t(2\fr)$ then
\beq\label{eq:estet1}
 \left|\int_0^{t} \rd \cE_2(s)\right|\lesssim \frac{ \log n}{n};
\eeq
if $z_t(Z)\in \cD_t^{\rm bulk}$ then
\beq\label{eq:estet}
 \left|\int_0^{t} \rd \cE_2(s)\right|\lesssim \frac{ \log n}{n\sqrt{\Im[z_{t}(Z)] \dist(z_t(Z), [\sfa(t), \sfb(t)])}},
\eeq
where $\rho_t^*$ is surppored on $ [\sfa(t), \sfb(t)]$
\end{claim}

\begin{proof}
If $z_t(Z)\in \cD_t(2\fr)$, the $z_s(Z)\in \cD_s(2\fr)$ for any $0\leq s\leq t$, and we have $\dist(z_s(Z), \{x_1(s), x_2(s),\cdots, x_n(s)\})\gtrsim 1$ from \eqref{e:extremepp}. The first claim \eqref{eq:estet1} follows from the same argument as in Claim \ref{p:error}.

For \eqref{eq:estet}, there are two cases
i) $\dist(z_t(Z), [\sfa(t), \sfb(t)])\leq \Im[z_t(Z)]$, then the righthand side of \eqref{eq:estet} simplifies to $\OO(\log n/ n\Im[z_t(Z)])$.
ii) $\dist(z_t(Z), [\sfa(t), \sfb(t)])\geq \Im[z_t(Z)]$, then it is necessary that $\Re[z_t(Z)]\not\in \supp(\rho_t^*)$. Without loss of generality, we assume that $\dist(z_t(Z), \supp(\rho_t^*))$ is achieved at the left edge $\sfa(t)$ of $\rho_t^*$. Then we can write $z_t(Z)=\sfa(t)-\kappa_t(Z)+\ri \eta_t(Z)$, with $\kappa_t(Z)\geq 0$, and the righthand side of \eqref{eq:estet} simplifies to $\OO(\log n/ n\sqrt{\eta_t(Z)(\kappa_t(Z)+\eta_t(Z))})$.

In the first case that $\dist(z_t(Z), [\sfa(t), \sfb(t)])\leq \Im[z_t(Z)]$, we decompose the time interval $[0,t]$ in the following way.  First we set $t_0=0$, and define 
\begin{align*}
t_{i+1}(Z):=\sup_{s\geq t_i(Z)}\left\{ \Im[z_s(Z)] \geq \frac{\Im[z_{t_i}(Z)]}{2}\right\}\wedge t,\quad i=0,1,2,\cdots.
\end{align*}
From our choice of domain $\cD_t^{\rm bulk}$, $\Im[z_0(Z)]/\Im[z_t(Z)]\lesssim n$. The above sequence will terminate at some $t_k(Z)=t$ for $k=\OO(\log n)$ depending on $Z$.
Moreover, since $\Im[z_s(Z)]$ is monotone decreasing 
for any $t_i(Z)\leq t\leq t_{i+1}(Z)$,
\beq\label{e:Imzsztbound}
\Im[z_{t_{i}}(Z)]\leq \Im[z_{t}(Z)]\leq \Im[z_{t_{i+1}}(Z)]\leq 2\Im[z_{t_{i}}(Z)].
\eeq
Then it follows that 
\begin{align*}\begin{split}
&\left\langle\frac{1}{ n^{3/2}}\int_0^{\cdot\wedge\sigma}\sum_{i=1}^n \frac{{\rm d} B_i(s)}{(x_i(s)-z_s(Z))^2} \right\rangle_{t_i}
= \frac{1}{n^{3}}\int_0^{t_i\wedge\sigma}\sum_{i=1}^n \frac{ \rd s}{|x_i(s)-z_s(Z)|^4}\\
\leq& \frac{1}{n^{3}}\int_0^{t_i\wedge\sigma}\sum_{i=1}^n \frac{ \rd s}{\Im[z_s(Z)]^2|x_i(s)-z_s(Z)|^2}
\leq\int_0^{t_i\wedge\sigma}\frac{\Im[\td m_s(z_s(Z))]\rd s}{n^2\Im[z_s(Z)]^3}\\
\lesssim&  
\int_0^{t_i\wedge\sigma}\frac{\Im[m_s(z_s(Z))]\rd s}{n^2\Im[z_s(Z)]^3}
\lesssim \frac{1}{n^2\Im[z_{t_i\wedge \sigma}(Z)]^2},
\end{split}
\end{align*} 
where we used \eqref{e:replacebound} and \eqref{e:propzt}.
Therefore, by Burkholder-Davis-Gundy inequality, for any $Z\in \cal L$ and $t_i$, the following holds with high probability, i.e., $1-n^{-\fC}$,
\begin{align}\begin{split}\label{eq:2term0}
&\sup_{0\leq t\leq t_i}\left|\frac{1}{ n^{3/2}}\int_0^{t\wedge\sigma}\sum_{i=1}^n \frac{{\rm d} B_i(s)}{(x_i(s)-z_s(Z))^2} \right|
\lesssim \frac{ \log n}{n\Im[z_{t_i\wedge \sigma}(Z)]}.
\end{split}
\end{align}
Moreover, for any $t\in[t_{i-1},t_i]$, the bound \eqref{eq:2term0} and \eqref{e:Imzsztbound} yield
\begin{align}\begin{split}\label{e:continuityarg}
\left|\int_0^{t\wedge \sigma} \rd \cE_2(s)\right|
\lesssim& \frac{ \log n}{n\Im[z_{t_i\wedge\sigma}(Z)]}\lesssim  \frac{\log n}{n\Im[z_{t\wedge\sigma}(Z)]}.
\end{split}
\end{align}
where we used \eqref{e:Imzsztbound}. 
 
For the second case that $z_t(Z)=\sfa(t)-\kappa_t(Z)+\ri \eta_t(Z)$, with $\kappa_t(Z)\geq 0$. We decompose the time interval $[0,t]$ in the following way.  First we set $t_0(Z)=0$, and define 
\begin{align}\label{e:chooseti}
t_{i+1}(Z):=\sup_{s\geq t_i(Z)}\left\{ \eta_s(Z)(\kappa_s(Z)+ \eta_s(Z))\geq \frac{\eta_{t_i}(Z)(\kappa_{t_i}(Z)+ \eta_{t_i}(Z))}{2}\right\}\wedge t,\quad i=0,1,2,\cdots.
\end{align}
Since $\eta_t(Z)\geq n^{-2/3}$ and $\eta_0(Z), \kappa_0(Z)\lesssim 1$, we have $t_k(Z)=t$ for $k=\OO( \log n)$. We compute the quadratic variance of $\rd \cE_2(s)$,
\begin{align*}\begin{split}
&\phantom{{}={}}\left\langle\frac{1}{ n^{3/2}}\int_0^{\cdot\wedge\sigma}\sum_{i=1}^n \frac{{\rm d} B_i(s)}{(x_i(s)-z_s)^2} \right\rangle_{t_i}
= \frac{1}{n^{3}}\int_0^{t_i\wedge\sigma}\sum_{i=1}^n \frac{ \rd s}{|x_i(s)-z_s|^4}\\
&\leq \frac{1}{n^{3}}\int_0^{t_i\wedge\sigma}\sum_{i=1}^n \frac{ \rd s}{\dist(z_s, \{x_j(s)\}_{1\leq j\leq n})^2|x_i(s)-z_s|^2}
=\frac{1}{n^{2}}\int_0^{t_i\wedge\sigma}\frac{\Im[\td m_s(z_s)]\rd s}{\dist(z_s, \{x_j(s)\}_{1\leq j\leq n})^2\eta_s}\\
&\lesssim  
\frac{1}{n^{2}} \int_{0}^{t_i \wedge \sigma} \frac{\Im[m_s(z_s)] }{\eta_{s}((\kappa_s+\eta_s)^{1/2}(t_i-s)+ \eta_{s})^2}\rd s
\lesssim
\frac{1}{n^{2}} \int_{0}^{t_i \wedge \sigma} \frac{ \eta_{s}/{\sqrt{\kappa_{s}+\eta_s}}}{\eta_{s}((\kappa_s+\eta_s)^{1/2}(t_i-s)+ \eta_{s})^2}\rd s \\
&\lesssim \frac{1}{n^{2}} \int_{0}^{t_i \wedge \sigma} \frac{\rd s}{(\kappa_{s}+\eta_s)^{3/2}(t_i-s+\eta_s/(\kappa_s+\eta_s)^{1/2})^2}\\
&\leq \frac{1}{n^{2}} \int_{0}^{t_i \wedge \sigma} \frac{\rd s}{(\kappa_{t_i \wedge \sigma}+\eta_{t_i \wedge \sigma})^{3/2}(t_i-s+\eta_{t_i \wedge \sigma}/(\kappa_{t_i \wedge \sigma}+\eta_{t_i \wedge \sigma})^{1/2})^2}
\lesssim  \frac{1} {n^{2}} \frac{1}{\eta_{t_i \wedge \sigma} (\kappa_{t_i \wedge \sigma}+\eta_{t_i \wedge \sigma})},
\end{split}
\end{align*} 
where in the third line we used \eqref{e:ssmt}, \eqref{e:replacebound} and \eqref{e:gap}, in the fifth line we used that 
\begin{align*}
\frac{\eta_s}{\sqrt{\kappa_s+\eta_s}}\asymp
\Im[m_s(z_s)]\asymp \Im[f_s(z_s)]=\Im[ f_{t_i\wedge \sigma}(z_{t_i\wedge \sigma})]\asymp\Im[ m_{t_i\wedge \sigma}(z_{t_i\wedge \sigma})]\asymp \frac{\eta_{t_i\wedge \sigma}}{\sqrt{\kappa_{t_i\wedge \sigma}+\eta_{t_i\wedge \sigma}}}.
\end{align*}
Therefore, by Burkholder-Davis-Gundy inequality, for any $Z\in \cL$ and $t_i(Z)$, the following holds with high probability, i.e., $1-n^{-\fC}$,
\begin{align}\begin{split}\label{eq:2term}
&\sup_{0\leq t\leq t_i}\left|\frac{1}{ n^{3/2}}\int_0^{t\wedge\sigma}\sum_{i=1}^n \frac{{\rm d} B_i(s)}{(x_i(s)-z_s(Z))^2} \right|
\lesssim \frac{ \log n}{n\sqrt{\eta_{t_i\wedge \sigma}(Z)(\kappa_{t_i\wedge\sigma}(Z)+\eta_{t_i\wedge \sigma}(Z))}}.
\end{split}
\end{align}
Moreover, for any $t\in[t_{i-1},t_i]$, the bound \eqref{eq:2term} with \eqref{e:chooseti} yield
\begin{align}\begin{split}\label{e:continuityarg0}
\left|\int_0^{t\wedge \sigma} \rd \cE_2(s)\right|
\lesssim& \frac{ \log n}{n\sqrt{\eta_{t_i\wedge \sigma}(Z)(\kappa_{t_i\wedge\sigma}(Z)+\eta_{t_i\wedge \sigma}(Z))}}\lesssim  \frac{ \log n}{n\sqrt{\eta_{t\wedge \sigma}(Z)(\kappa_{t\wedge\sigma}(Z)+\eta_{t\wedge \sigma}(Z))}}.
\end{split}
\end{align}

We define $\cW$ to be the set of Brownian paths $\{B_1(s), \cdots, B_n(s)\}_{0\leq s\leq t}$ on which \eqref{eq:estet1}, \eqref{e:continuityarg} and \eqref{e:continuityarg0} hold for  any $Z\in \cL$ and times $t_0(Z), t_1(Z), t_2(Z)\cdots$.
The above discussions  imply that  $\cW$ holds with overwhelming probability, i.e., $\bP(\cW)\geq 1-\OO(\log n|\cL| n^{-\fC})$. This finishes the proof of Claim \ref{c:error}.
\end{proof}

With Proposition \ref{p:bulkrigid3}  as input,  we can now start to prove Proposition \ref{p:bulkrigid2}.
\begin{proof}[Proof of Proposition \ref{p:bulkrigid2}]
For $z\in \bC_+$ outside the contour $z(\mathscr  C_t(2\fr))$, the estimate \eqref{e:bulkrigid2} follows from the same argument as Proposition \ref{p:bulkrigid3}. In the following we prove \eqref{e:bulkrigid2} for $z$ close to the spectrum. 
We notice that the Claims \ref{p:error}, \ref{prop:gtsumbound} and \ref{prop:gtchange} still hold for any $Z\in \cL$ with new $t(Z)$ and stoping time $\sigma$ defined in \eqref{e:deftz2} and \eqref{stoptime2} respectively.
Moreover, for any $Z\in \cL$,  Assumption \ref{a:A_n} implies that $|\Delta_0(z_0)|\lesssim (\log n)^{ \fa}/n$. Thus, for $Z\in \cL$ with $z_t(Z)\in \cD_t^{\rm bulk}$, we can write \eqref{eq:mzt} as 
\begin{align}\label{e:deltaee}
|\Delta_{t\wedge\sigma}(z_{t\wedge\sigma})|\leq \int_0^{t\wedge\sigma}|\Delta_s(z_s)||\del_z \tilde f_s(Z_s)\rd s|
+\OO\left(
\frac{(\log n)^{ \fa}}{n\sqrt{\Im[z_{t\wedge\sigma}] \dist(z_{t\wedge\sigma}, [\sfa(t\wedge\sigma), \sfb(t\wedge\sigma)])}}\right).
\end{align}
For the integrand of \eqref{e:deltaee}, we notice that for $s\leq t\wedge \sigma$,
\begin{align}\begin{split}\label{e:aterm}
\left|\del_z \tilde f_s(Z_s)\right|
&\leq \left|\del_z \tilde m_s(z_s)\right|+ \left|\del_z g_s(Z_s)\right|
\leq \frac{\Im[ \tilde m_s(z_s)]}{\Im[z_s]}+\OO(1)\\
&=\frac{\Im[ \tilde f_s(Z_s)-g_s(Z_s)]}{\Im[z_s]}+\OO(1)
\leq \frac{\Im[ \tilde f_s(Z_s)]}{\Im[z_s]}+\OO(1),
\end{split}\end{align}
where we used that $g_s(Z)$ is real for $z=z(Z)\in \bR$, and thus $\Im[g_s(Z_s)]=\OO(\Im[z_s])$.
Since $z_s(Z)\in \cal D^{\rm bulk}_s$, by the definition of $\cal D^{\rm bulk}_s$, we have $\Im[f_s(Z_s)]\geq (\log n)^{ \fa+2}/ (n\Im[z_s])$. Moreover, since $s\leq \sigma$, we have $|\tilde f_s(Z_s)-f_s(Z_s)|=|\Delta_s(z_s)|\lesssim(\log n)^{ \fa+1}/(n\Im[z_s])\lesssim \Im[f_s(Z_s)]/\log n$.  Therefore,
\begin{align}\begin{split}\label{e:aterm}
\left|\del_z \td f_s(Z_s)\right|
&\leq \frac{\Im[ \td f_s(Z_s)]}{\Im[z_s]}+\OO(1)
\leq \frac{\Im[ f_s(Z_s)]+|\Delta_s(z_s)|}{\Im[z_s]}+\OO(1)\\
&= \left(1+\frac{\OO(1)}{\log n}\right)\frac{\Im[ f_s(Z_s)]}{\Im[z_s(u)]}+\OO(1)=\left(1+\frac{\OO(1)}{\log n}\right)\frac{\del_s\Im[z_s]}{\Im[z_s]}+\OO(1)
.
\end{split}\end{align}
We denote
\begin{align} \label{eqn:betabd}
\beta_s\deq \left|\del_z \tilde f_s(Z_s)\right|=\left(1+\frac{\OO(1)}{\log n}\right)\frac{\del_s \Im[z_s]}{\Im[z_s]}+\OO(1)= \OO\left(\frac{\del_s\Im[z_s]}{\Im[z_s]}\right).
\end{align}
With $\beta_s$, we can rewrite \eqref{e:deltaee} as
\begin{align*}\begin{split}
\left|\Delta_{t\wedge \sigma}(z_{t\wedge \sigma})\right|
\leq \int_0^{t\wedge\sigma}\beta_s\left|\Delta_s(z_s)\right|\rd s
+\OO\left(\frac{(\log n)^{\fa}}{n\sqrt{\Im[z_{t\wedge\sigma}] \dist(z_{t\wedge\sigma}, [\sfa(t\wedge\sigma), \sfb(t\wedge\sigma)])}}\right).
\end{split}
\end{align*}
By Gr{\" o}nwall's inequality, this implies the estimate
\begin{align}\begin{split}\label{e:midgronwall}
\left|\Delta_{t\wedge\sigma}(z_{t\wedge\sigma})\right|
&\lesssim \frac{(\log n)^{ \fa}}{n\sqrt{\Im[z_{t\wedge\sigma}] \dist(z_{t\wedge\sigma}, [\sfa(t\wedge\sigma), \sfb(t\wedge\sigma)])}}\\
&+\int_0^{t\wedge\sigma}\beta_se^{\int_s^{t\wedge\sigma} \beta_\tau\rd \tau}\frac{(\log n)^{ \fa}}{n\sqrt{\Im[z_{s}] \dist(z_s, [\sfa(s), \sfb(s)])}}\rd s.
\end{split}
\end{align}
 For the integral of $\beta_\tau$, we have
\begin{align*}
e^{\int_s^{t\wedge\sigma} \beta_\tau\rd \tau}
\leq e^{\OO(t-s)} e^{\left(1+\frac{\OO(1)}{\log n}\right)\log \left(\frac{\Im[z_{s}(u)]}{\Im[z_{t\wedge\sigma}(u)]}\right)}
=e^{\OO(t-s)} \left(\frac{\Im[z_{s}(u)]}{\Im[z_{t\wedge\sigma}(u)]}\right)^{1+\frac{\OO(1)}{\log n}}
\leq \OO(1)  \frac{\Im[z_s(u)]}{\Im[z_{t\wedge\sigma}(u)]}.
\end{align*}
In the last inequality, we used that by our construction of $\cD_t^{\rm bulk}$,  we have $\Im[z_{s}(u)]/\Im[z_{t\wedge\sigma}(u)]=\OO(n)$.
Combining the above inequality with \eqref{eqn:betabd} we can bound the last term in \eqref{e:midgronwall} by
\begin{align}\begin{split}\label{e:term2}
&\phantom{{}={}}\int_0^{t\wedge\sigma}\beta_s e^{\int_s^{t\wedge\sigma} \beta_\tau\rd \tau}\frac{(\log n)^{ \fa}}{n\sqrt{\Im[z_{s}] \dist(z_s, [\sfa(s), \sfb(s)])}}\rd s\\
&\lesssim \int_0^{t\wedge\sigma}\frac{\del_s\Im[ z_s]}{\Im[z_{t\wedge\sigma}]}\frac{(\log n)^{\fa}}{n\sqrt{\Im[z_{s}] \dist(z_s, [\sfa(s), \sfb(s)])}}\rd s\\
&\lesssim 
\frac{(\log n)^{ \fa}}{n\Im[z_{t\wedge \sigma}(u)]}\int_0^{t\wedge\sigma}\frac{\del_s\Im[ z_s]}{\sqrt{\Im[z_{s}] \dist(z_s, [\sfa(s), \sfb(s)])}}\rd s.
\end{split}\end{align} 

There are two cases
i) $\dist(z_{t\wedge\sigma}(Z), [\sfa({t\wedge\sigma}), \sfb({t\wedge\sigma})])\leq \Im[z_{t\wedge\sigma}(Z)]$.
ii)$\dist(z_{t\wedge\sigma}(Z), [\sfa({t\wedge\sigma}), \sfb({t\wedge\sigma})])\geq \Im[z_{t\wedge\sigma}(Z)]$, then it is necessary that $\Re[z_{t\wedge\sigma}(Z)]\not\in \supp(\rho_{t\wedge\sigma}^*)$. Without loss of generality, we assume that $\dist(z_{t\wedge\sigma}(Z), \supp(\rho_{t\wedge\sigma}^*))$ is achieved at the left edge $\sfa_1({t\wedge\sigma})$ of $\rho_{t\wedge\sigma}^*$.
 Then we can write $z_{t\wedge\sigma}(Z)=\sfa({t\wedge\sigma})-\kappa_{t\wedge\sigma}(Z)+\ri \eta_{t\wedge\sigma}(Z)$, with $\kappa_{t\wedge\sigma}(Z)\geq 0$, and $\dist(z_{t\wedge\sigma}(Z), \sfa({t\wedge\sigma}))\asymp \kappa_t(Z)+\eta_t(Z)$.
 
In the first case, using \eqref{e:propzt} we have
\begin{align}\begin{split}\label{e:term23}
 \int_0^{t\wedge\sigma}\frac{\del_s\Im[ z_s]}{\sqrt{\Im[z_{s}] \dist(z_s, [\sfa(s), \sfb(s)])}}\rd s
\lesssim
 \int_0^{t\wedge\sigma}\frac{\del_s\Im[ z_s]}{\Im[z_{s}] }\rd s
=
\log\left(\frac{\Im[z_0]}{\Im[z_{t\wedge \sigma}]}\right)
\lesssim \log n,
\end{split}\end{align} 
It follows by combining  \eqref{e:midgronwall},  \eqref{e:term2} and \eqref{e:term23}, we get that
\begin{align}\label{e:Deltabb}
\left|\Delta_{t\wedge\sigma}(z_{t\wedge\sigma})\right|
\lesssim  \frac{(\log n)^{\fa+1}}{n\Im[z_{t\wedge \sigma}]}\leq \frac{\fC(\log n)^{\fa+1}}{2n\sqrt{z_{t\wedge \sigma}\dist(z_{t\wedge\sigma}(Z), [\sfa({t\wedge\sigma}), \sfb({t\wedge\sigma})])}},
\end{align}
provided we take $\fC$ large enough. 

In the second case, 
we can bound the last term in \eqref{e:midgronwall} by
\begin{align}\begin{split}\label{e:term3}
 &\int_0^{t\wedge\sigma}\frac{\del_s\Im[ z_s]}{\sqrt{\Im[z_{s}] \dist(z_s, [\sfa(s), \sfb(s)])}}\rd s\lesssim\int_0^{t\wedge\sigma}\frac{\del_s\eta_s}{\sqrt{\eta_s(\kappa_s+\eta_s)}}\rd s\\
&\lesssim 
\frac{\eta_{t\wedge \sigma}}{\sqrt{\eta_{t\wedge \sigma}+\kappa_{t\wedge \sigma}}}\int_0^{t\wedge\sigma}\frac{\del_s \eta_s}{(\eta_s)^{3/2}}\rd s
\lesssim
\frac{\sqrt{\eta_{t\wedge\sigma}}}{\sqrt{\eta_{t\wedge\sigma}+\kappa_{t\wedge \sigma}}},
\end{split}\end{align} 
It follows by combining \eqref{e:midgronwall}, \eqref{e:term2} and \eqref{e:term3}, we get that
\begin{align}\label{e:Deltabb2}
\left|\Delta_{t\wedge\sigma}(z_{t\wedge\sigma})\right|
\lesssim  \frac{(\log n)^{ \fa}}{n\sqrt{\eta_{t\wedge\sigma}(\eta_{t\wedge\sigma}+\kappa_{t\wedge \sigma})}}\leq \frac{\fC(\log n)^{ \fa+1}}{2n\sqrt{z_{t\wedge \sigma}\dist(z_{t\wedge\sigma}(Z), [\sfa({t\wedge\sigma}), \sfb({t\wedge\sigma})])}},
\end{align}
provided we take $\fC$ large enough.

Similarly to the proof of Proposition \ref{p:bulkrigid2}, we can approximate $W\in \cD_{t\wedge \sigma}^{\rm bulk}$ with $w=z(W)$, by the image of some lattice point $Z\in \cL$. And on the domain $\cD_{t\wedge\sigma}^{\rm bulk}$, we also have that both both $\td m_{t\wedge \sigma}$ and $m_{t\wedge\sigma}$ are Lipschitz with Lipschitz constant $\OO(n)$. Therefore \eqref{e:Deltabb} and \eqref{e:Deltabb2} imply 
\begin{align*}\begin{split}
&\phantom{{}={}}\left|\td m_{t\wedge\sigma}(w)-m_{t\wedge\sigma}(w)\right|\leq \frac{\fC(\log n)^{ \fa +1}}{n\Im[z_{t\wedge \sigma}]},
\end{split}\end{align*}
uniformly for $w\in z(\cD_{t\wedge \sigma}^{\rm bulk})$. It follows that $\sigma=1$, and this finishes the proof of 
Proposition \ref{p:bulkrigid2}.

\end{proof}

\section{Optimal particle rigidity for nonintersecting Brownian bridges}\label{s:rigidityBB}
In Theorem \ref{t:rigidity}, we have proved the Optimal particle rigidity for the weighted nonintersecting Brownian bridges. In this section, using a coupling argument, we show the optimal particle rigidity for nonintersecting Brownian bridges with fixed boundary data.

\begin{theorem}\label{t:rigidityBB}
Under Assumptions \ref{a:reg}, \ref{a:ncritic}, \ref{a:A_n} and \ref{a:B_n}, for any small time $\ft>0$, the following holds for the particle locations of nonintersecting Brownian bridges between $\mu_{A_n}$ and $\mu_{B_n}$.
With high probability, for any time $\ft\leq t\leq 1$ we have
\begin{enumerate}
\item We denote $\gamma_i(t)$ the $1/n$-quantiles of the density $\rho^*_t(\cdot)$, i.e. 
\begin{align*}
\frac{i+1/2}{n}=\int_{-\infty}^{\gamma_i(t)}\rho_t^*(x)\rd x,\quad 1\leq i\leq n, 
\end{align*}
then uniformly for $1\leq i\leq n$, the locations of $x_i(t)$ are close to their corresponding quantiles
\begin{align}\label{e:bulkrigidBB}
\gamma_{i-(\log n)^{\OO(1)}}(t)\leq x_i(t)\leq \gamma_{i+(\log n)^{\OO(1)}}(t).
\end{align}
\item If $\rho_t^*$ is supported on $[\sfa(t), \sfb(t)]$, then the particles close to the left and right boundary points of $\supp(\rho_t^*)$ satisfies
\begin{align}\label{e:edgerigid2BB}
   x_1(t)\geq {\sf a}(t)-\frac{(\log n)^{\OO(1)}}{n^{2/3}},\quad x_{n}(t)\leq {\sf b}(t)+ \frac{(\log n)^{\OO(1)}}{n^{2/3}}.
\end{align}
\end{enumerate}
\end{theorem}
%

\begin{proof}
We recall from \eqref{e:wdensity}, the weighted nonintersecting Brownian bridges \eqref{e:wdensity} can be sampled in two steps. First we sample the boundary $\bmb'=(b_1', b_2',\cdots, b_n')$ using  the density \eqref{e:boundarydensity}. Then we sample  nonintersecting Brownian bridges with boundary data $\bma, \bmb'$.  Under Assumptions \ref{a:reg}, \ref{a:ncritic} and \ref{a:A_n}, in Theorem \ref{t:rigidity}, we showed that for any time $\ft\leq t\leq 1$, the particles of the weighted nonintersecting Brownian bridges satisfy optimal rigidity estimates. As an easy consequence, there exists a boundary data $\bmb'$ (sampled from the density \eqref{e:boundarydensity}), it satisfies:
\begin{align*}
\gamma_{i-(\log n)^{\OO(1)}}(1)\leq b_i'\leq \gamma_{i+(\log n)^{\OO(1)}}(1).
\end{align*}
and 
\begin{align*}
 b'_1\geq {\sf a}(1)-\frac{(\log n)^{\OO(1)}}{n^{2/3}},\quad b'_{n}\leq {\sf b}(1)+ \frac{(\log n)^{\OO(1)}}{n^{2/3}}.
\end{align*}
Moreover,  the nonintersecting Brownian bridges between $\bma, \bmb'$ satisfies the optimal rigidity estimates as in Theorem \ref{t:rigidity}.

We denote the nonintersecting Brownian bridges between $\bma$ and $\bmb'$ as $\{y_1(t)\leq y_2(t)\leq \cdots \leq y_n(t)\}_{0\leq t\leq 1}$. Then with high probability, it holds
\begin{align}\label{e:bulkrigidBBc}
\gamma_{i-(\log n)^{\OO(1)}}(t)\leq y_i(t)\leq \gamma_{i+(\log n)^{\OO(1)}}(t).
\end{align}
\item If $\rho_t^*$ is supported on $[\sfa(t), \sfb(t)]$, then the particles close to the left and right boundary points of $\supp(\rho_t^*)$ satisfies
\begin{align}\label{e:edgerigid2BBc}
   y_1(t)\geq {\sf a}(t)-\frac{(\log n)^{\OO(1)}}{n^{2/3}},\quad y_{n}(t)\leq {\sf b}(t)+ \frac{(\log n)^{\OO(1)}}{n^{2/3}}.
\end{align}
Since $b_i \leq b'_{i+(\log n)^{\OO(1)}}$, and $a_i \leq a_{i+(\log n)^{\OO(1)}}$, using the monotonicity of nonintersecting brownian bridges, Theorem \ref{thm:extreme}, we can couple the nonintersecting Brownian bridges $\{x_1(t)\leq x_2(t)\leq \cdots \leq x_n(t)\}_{0\leq t\leq 1}$ and $\{y_1(t)\leq y_2(t)\leq \cdots \leq y_n(t)\}_{0\leq t\leq 1}$,
\begin{align}\label{e:ubb}
x_i(t)\leq y_{i+(\log n)^{\OO(1)}},\quad 0\leq t\leq 1,
\end{align}
where we used the convention that $y_i=+\infty$ for $i>n$. The coupling \eqref{e:ubb} and \eqref{e:bulkrigidBBc} together implies 
\begin{align}\label{e:xub}
x_i(t)\leq y_{i+(\log n)^{\OO(1)}}(t)\leq \gamma_{i+(\log n)^{\OO(1)}}, \quad \ft\leq t\leq 1,
\end{align}
with high probability. A similar coupling using $b_i \geq b'_{i-(\log n)^{\OO(1)}}$ and $a_i \geq a_{i-(\log n)^{\OO(1)}}$ implies 
\begin{align}\label{e:xlb}
x_i(t)\geq y_{i-(\log n)^{\OO(1)}}(t)\geq \gamma_{i-(\log n)^{\OO(1)}},\quad \ft\leq t\leq 1,
\end{align}
where we used the convention that $y_i=-\infty$ for $i\leq 0$.
The claim \eqref{e:bulkrigidBB} follows from combining \eqref{e:xub}  and \eqref{e:xlb}.

To prove \eqref{e:edgerigid2BB}, we construct the new boundary data:
\begin{align}\label{e:defb+-bb}
\bmb^{\pm}=(b'_1\pm (\log n)^{\fC}n^{-2/3}, b'_2\pm (\log n)^{\fC}n^{-2/3}, \cdots, b'_n\pm (\log n)^{\fC}n^{-2/3}).
\end{align}
Thanks to Assumption \ref{a:B_n}, if we take $\fC$ large enough, then
\begin{align*}
b_i^-\leq b_i, b_i'\leq b_i^+,\quad 1\leq i\leq n.
\end{align*}
Since affine shifts preserve nonintersecting Brownian bridges, the nonintersecting Brownian bridges between $\bma$ and 
$\bmb^{\pm}$, denoted as $y_1^{\pm}(t)\leq y^{\pm}_2(t)\leq \cdots \leq y^{\pm}_n(t)$, they can be coupled with the nonintersecting Brownian bridge $y_1(t)\leq y_2(t)\leq \cdots \leq y_n(t)$
\begin{align*}
y_i^{-}(t)=y_i(t)-\frac{t (\log n)^{\fC}}{n^{2/3}},\quad y_i^{+}(t)=y_i(t)+\frac{t (\log n)^{\fC}}{n^{2/3}}, \quad 0\leq t\leq 1.
\end{align*}
Thanks to Theorem \ref{t:rigidity}, it holds with high probability,
\begin{align*}
 \sfa(t)-\frac{(\log n)^{\OO(1)}}{n^{2/3}}\leq y_1^-(t)\leq y_1^+(t) \leq 
 y_n^-(t)\leq y_n^+(t) \leq \sfb(t)+\frac{(\log n)^{\OO(1)}}{n^{2/3}}, \quad 0\leq t\leq 1.
\end{align*}
Using the monotonicity of nonintersecting brownian bridges,  Theorem \ref{thm:extreme}, we can couple the nonintersecting Brownian bridges  $x_1(t)\leq x_2(t)\leq \cdots \leq x_n(t)$
with the nonintersecting Brownian bridges $\{y_1^{\pm}(t)\leq y^{\pm}_2(t)\leq \cdots \leq y^{\pm}_n(t)\}_{t_0\leq t\leq 1}$, and conclude that with high probability
\begin{align}\label{e:couple}
 \sfa(t)-\frac{(\log n)^{\OO(1)}}{n^{2/3}}\leq y_1^-(t)\leq x_1(t)\leq x_n(t)\leq y_n^+(t)\leq \sfb(t)+\frac{(\log n)^{\OO(1)}}{n^{2/3}}, \quad \ft\leq t\leq 1.
\end{align}
This finishes the proof of \eqref{e:edgerigid2BB}.

\end{proof}

\section{Edge Universality}
In this section we prove edge universality for nonintersecting Brownian bridges Theorem \ref{t:universality}. The proof consists of two parts. In Section \ref{s:edgew}, we prove the 
edge universality for the weighted nonintersecting Brownian bridges \eqref{e:wdensity}. The edge universality for nonintersecting Brownian bridges follows from a coupling with weighted nonintersecting Brownian bridges. The proof is given in Section \ref{s:edge}.

\subsection{Edge Universality for weighted nonintersecting Brownian bridges}\label{s:edgew}

We recall the random walk corresponding to the weighted nonintersecting Brownian bridges from \eqref{e:weightrandomwalk}
\begin{align}\label{e:wr2}
\rd x_i(t) = \frac{1}{\sqrt{n}} \rd B_i(t) +\frac{1}{n}\sum_{j:j\neq i}\frac{\rd t}{x_i(t)-x_j(t)}+g_t(x_i(t), \mu_t,t)+\cE_t^{(i)}(\bmx(t)).
\end{align}
In this section we prove edge universality for the weighted nonintersecting Brownian bridges \eqref{e:wr2}. 
We recall from \eqref{e:square}, for any $0<t<1$, in the neighborhood of $x=\sfa(t)$,  the limiting density $\rho_t^*(x)$ has square root behavior 
\begin{align}\label{e:square2}
\rho_t^*(x)= \frac{\sfs(t)\sqrt{[x-\sfa(t)]_+}}{\pi}+\OO(|x-\sfa(t)|^{3/2}).
\end{align}

\begin{proposition}\label{p:weu}
Fix small $\ft>0$. Under the Assumptions \ref{a:reg}, \ref{a:ncritic} and \ref{a:A_n} as $n$ goes to infinity,
 the fluctuations of extreme particles of the weighted nonintersecting Brownian bridges \eqref{e:wr2} after proper rescaling, converge to the Airy point process: for any time $\ft\leq t\leq1$,
 the 
\begin{align*}
(\sfs(t)n)^{2/3}(x_1(t)-\sfa(t), x_2(t)-\sfa(t),x_3(t)-\sfa(t),\cdots)\rightarrow \text{Airy Point Process}
\end{align*}
The same statement holds for particles close to the right edge.
\end{proposition}

Under Assumptions \ref{a:reg}, \ref{a:ncritic} and \ref{a:A_n}, in Theorem \ref{t:rigidity}, we showed that for any time $\ft\leq t\leq 1$, the particles of the random walk \eqref{e:wr2} satisfy optimal rigidity estimates. Fix time $\ft_0=\ft-n^{-1/3+\oo(1)}$, we can condition on the weighted nonintersecting Brownian bridge \eqref{e:wr2} at time $\ft_0$, such that $\bmx(\ft_0)$ satisfies the optimal rigidity estimates. 
Then from Claim \ref{prop:gtchange}, we have that with high probability for any $\ft_0\leq t\leq 1$, $g_t(x_i(t), \mu_t,t)=g_t(x_i(t))+\OO(\log n/n)$, and we can rewrite \eqref{e:wr2} as
\begin{align}\label{e:wr3}
\rd x_i(t) = \frac{1}{\sqrt{n}} \rd B_i(t) +\frac{1}{n}\sum_{j:j\neq i}\frac{\rd t}{x_i(t)-x_j(t)}+g_t(x_i(t))+\OO\left(\frac{\log n}{n}\right).
\end{align}
The error term $\OO(\log n/n)$ does not affect edge fluctuation, which is on the scale $\OO(n^{-2/3})$. If we ignore the error term, \eqref{e:wr3} is the Dyson's Brownian motion with drift $g_t(\cdot)$. For Dyson's Brownian motion with general drift, the short time edge universality is well-understood.

In \cite{adhikari2020dyson}, joint with A. Adhikari, we considered the $\beta$-Dyson's Brownian motion with general potential $V$,
\begin{align}\label{e:DBM}
\rd y_i(t) = \sqrt\frac{2}{\beta n} \rd B_i(t) +\frac{1}{n}\sum_{j:j\neq i}\frac{\rd t}{y_i(t)-y_j(t)}-\frac{1}{2}V'(y_i(t))\rd t,\quad i=1,2,\cdots, n.
\end{align}
For any probability density $\rho_0$, we denote $\rho_t$ the solution of the McKean-Vlasov equation \cite[Equation (2.2)]{adhikari2020dyson} associated to \eqref{e:DBM} with initial data $\rho_0$. If the initial data of the Dyson's Brownian motion \eqref{e:DBM}, weakly converges to $\rho_0$, then the empirical particle density of \eqref{e:DBM} at time $t\geq 0$ weakly converges to $\rho_t$. Moreover, if $\rho_0$ has square root behavior in a small neighborhood of its left edge, i.e. $\rho_0(x)=S(0)\sqrt{[x-E(0)]_+}/\pi+\OO(|x-E(0)|^{3/2})$, then $\rho_t$ also has square root behavior, $\rho_t(x)=S(t)\sqrt{[x-E(t)]_+}/\pi+\OO(|x-E(t)|^{3/2})$.

One result of \cite{adhikari2020dyson} states that the extreme particles of the Dyson's Brownian motion for general $\beta$ and potential $V$ converge to the Airy-$\beta$ point process in a short time.
\begin{theorem}{\cite[Theorem 6.1]{adhikari2020dyson}} \label{t:edgeUniv}
Suppose the potential $V$ is analytic, and near the left edge the initial data of \eqref{e:DBM} satisfies rigidity estimates on the optimal scale $(\log n)^{\OO(1)} /n^{2/3}$ with respect to a measure $\rho_0$, which has square root behavior in a small neighborhood of its left edge . Let $t \asymp N^{-1/3+\oo(1)}$, then with high probability, 
\begin{align*}
(S(t)n)^{2/3}(y_1(t)-E(t), y_2(t)-E(t), y_3(t)-E(t),\cdots)\rightarrow \text{Airy-$\beta$ Point Process}.
\end{align*}
The same statement holds for particles close to the right edge.
\end{theorem}

Although the potential for the $\beta$-Dyson's Brownian motion \eqref{e:DBM} studied in \cite{adhikari2020dyson} does not depend on time. It is straightforward to adapt the proof in \cite{adhikari2020dyson} for the case that the potential smoothly depends on time. Proposition \ref{p:weu} follows from applying Theorem \ref{t:edgeUniv} to \eqref{e:wr3}.

\subsection{Edge Universality for nonintersecting Brownian bridge}\label{s:edge}

In this section we prove the main result of this paper
Theorem \ref{t:universality}, edge universality of nonintersecting Brownian bridges. 

\begin{proof}[Proof of Theorem \ref{t:universality}]
Fix time $t_0\leq t$ such that $n^{-2/3+\oo(1)}\leq t-t_0\ll (\log n)^{-\fC}$, where the constant $\fC$ will be chosen later. 
We construct the new boundary data:
\begin{align}\label{e:defb+-}
\bmb^{\pm}=(b_1\pm (\log n)^{\fC}n^{-2/3}, b_2\pm (\log n)^{\fC}n^{-2/3}, \cdots, b_n\pm (\log n)^{\fC}n^{-2/3}),
\end{align}
with $\fC$ large enough, such that 
\begin{align*}
b_i^-\leq b_i\leq b_i^+,\quad 1\leq i\leq n.
\end{align*}
We denote nonintersecting Brownian bridges between $\bma$ and $\bmb$  after conditioning on time $t_0$, as $\{x_1(s)\leq x_2(s)\leq \cdots \leq x_n(s)\}_{t_0\leq s\leq 1}$. The affine shifts preserve nonintersecting Brownian bridges.
If we denote the nonintersecting Brownian bridges between $x_1(t_0)\leq x_2(t_0)\leq \cdots \leq x_n(t_0)$, and 
$\bmb^{\pm}$, as $x_1^{\pm}(s)\leq x^{\pm}_2(s)\leq \cdots \leq x^{\pm}_n(s)$, we can couple it with the nonintersecting Brownian bridges between $\bma$ and $\bmb$ 
\begin{align*}
x_i^-(s)=x_i(s)-\frac{s-t_0}{1-t_0}\frac{(\log n)^{\fC}}{n^{2/3}},\quad
x_i^+(s)=x_i(s)+\frac{s-t_0}{1-t_0}\frac{(\log n)^{\fC}}{n^{2/3}}, \quad t_0\leq s\leq 1
\end{align*}
Especially, at time $s=t$ with $t-t_0\ll (\log n)^{-\fC}$, it holds that 
\begin{align}\label{e:ulbb}
x_i^-(t), x_i^+(t)=x_i(t)+\oo(n^{-2/3})
\end{align}

We denote $\{y_1(s)\leq y_2(s)\leq \cdots\leq y_n(s)\}_{t_0\leq s\leq 1}$ the weighted nonintersecting Brownian bridges starting at time $t_0$ with initial data $x_1(t_0)\leq x_2(t_0)\leq \cdots \leq x_n(t_0)$. Then the same argument as in Proposition  \ref{p:weu} gives that at time $s=t$, the extreme particles of $\{y_1(t)\leq y_2(t)\leq \cdots\leq y_n(t)\}$ are asymptotically given by Airy point process, 
\begin{align}\label{e:uni}
(\sfs(t)n)^{2/3}(y_1(t)-\sfa(t), y_2(t)-\sfa(t),y_3(t)-\sfa(t),\cdots)\rightarrow \text{Airy Point Process}.
\end{align}
Moreover, Theorem \ref{t:rigidity} implies that with high probability
$y_i(1)$ is close to the corresponding $1/n$-quantiles of the density $\rho_B$, as defined in \eqref{e:gamma0}
\begin{align}\label{e:bulkrigidu}
\gamma_{i-(\log n)^{\OO(1)}}(1)\leq y_i(1)\leq \gamma_{i+(\log n)^{\OO(1)}}(t),
\end{align}
and the particles close to the left and right boundary points of $\supp(\rho_B)$ satisfies
\begin{align}\label{e:edgerigidu2}
   y_1(t)\geq {\sf a}(1)-\frac{(\log n)^{\OO(1)}}{n^{2/3}},\quad y_{n}(t)\leq {\sf b}(1)+ \frac{(\log n)^{\OO(1)}}{n^{2/3}}.
\end{align}
Thanks to Assumption \ref{a:B_n}, if we take $\fC$ large enough, then with high probability
\begin{align*}
b_i^-\leq y_i(1)\leq b_i^+.
\end{align*}
Using the monotonicity of nonintersecting Brownian bridges, Theorem \ref{thm:extreme}, we can couple the the weighted nonintersecting Brownian bridges starting at time $t_0$ with initial data $x_1(t_0)\leq x_2(t_0)\leq \cdots \leq x_n(t_0)$
with the nonintersecting Brownian bridges $\{x_1^{\pm}(s)\leq x^{\pm}_2(s)\leq \cdots \leq x^{\pm}_n(s)\}_{t_0\leq s\leq 1}$, and conclude that with high probability
\begin{align}\label{e:couple}
x_i^-(s)\leq y_i(s)\leq x_i^+(s).
\end{align}
Estimates \eqref{e:ulbb} and \eqref{e:couple} imply that with high probability at time $s=t$ with $t-t_0\ll (\log n)^{-\fC}$,
\begin{align}\label{e:xibb}
y_i(t)=x_i(t)+\oo(n^{-2/3}).
\end{align}
Theorem \ref{t:universality} follows from combining \eqref{e:uni} and \eqref{e:xibb}.
\end{proof}

\bibliography{References.bib}
\bibliographystyle{abbrv}

\end{document}